\documentclass[11pt,reqno]{amsart}
\usepackage{color, amsmath,amssymb, amsfonts, amstext,amsthm, latexsym}
\allowdisplaybreaks
\newcommand{\RR}{{\mathbb R}}

\newcommand{\EE}{{\mathbb E}}
\newcommand{\ee}{{\tilde{e}}}

\newcommand{\LL}{{\mathbb L}}
  
\newcommand{\PP}{{\mathbb P}}

\renewcommand{\AA}{{\tilde{A}}}
\newcommand{\QQ}{{\tilde{Q}}}
\newcommand{\WW}{{\tilde{W}}}
\newcommand{\GG}{{\tilde{G}}}
\newcommand{\KK}{{\tilde{K}}}
\renewcommand{\SS}{{\tilde{S}}}

\numberwithin{equation}{section}
\newtheorem{theorem}{Theorem}[section]

\newtheorem{lemma}[theorem]{Lemma}

\newtheorem{prop}[theorem]{Proposition}

\begin{document}
\title[Strong Convergence of time Euler schemes for 2D Boussinesq]
{Speed of convergence of time Euler schemes  
 \\ for a stochastic  2D Boussinesq model}

\author[H. Bessaih]{Hakima Bessaih}
\address{Florida International University, Mathematics and Statistics Department, 11200 SW 8th Street, Miami, FL 33199, United States}
\email{hbessaih@fiu.edu}

\author[A. Millet]{ Annie Millet}
\address{SAMM, EA 4543,
Universit\'e Paris 1 Panth\'eon Sorbonne, 90 Rue de
Tolbiac, 75634 Paris Cedex France {\it and} Laboratoire de
Probabilit\'es, Statistique et Mod\'elisation, UMR 8001, 
  Universit\'es Paris~6-Paris~7} 
\email{amillet@univ-paris1.fr}




\begin{abstract}
 We prove that an implicit time Euler scheme for the 2D-Boussinesq model on the torus $D$ converges.
Various moment of the $W^{1,2}$-norms of the velocity and temperature,   as well as   their discretizations,  are computed. 
We obtain the optimal speed of convergence in probability, and a logarithmic speed of convergence in $L^2(\Omega)$. These results are
deduced from a time regularity of the solution both in $L^2(D)$ and $W^{1,2}(D)$,  and from an $L^2(\Omega)$ convergence restricted to a subset
where the $W^{1,2}$-noms of the solutions are bounded. 
\end{abstract}

\maketitle

\section{Introduction}\label{s1} \smallskip

The Boussinesq equations have been used as a model in many geophysical applications. They have been widely studied in a both
 the deterministic and stochastic settings. 
 We take random forcings into account and formulate the B\'enard convection problem as a system of stochastic partial differential equations (SPDEs). 
 The need to take stochastic effects into account for modeling complex systems has now become widely recognized. Stochastic 
partial differential equations (SPDEs) arise naturally as mathematical models for nonlinear macroscopic dynamics under random influences.
 The Navier-Stokes equations are coupled with a transport equation for the temperature and with diffusion. 
 Here, the system will be subject to a multiplicative random perturbation which will be defined later. 
Here, $u$ describes the fluid velocity field, while $\theta$ describes the temperature  of the buoyancy  driven  fluid, and $\pi$ is the  fluid's  pressure.

We study the multiplicative stochastic Boussinesq equations
\begin{align} \label{velocity}
 \partial_t u - \nu \Delta u + (u\cdot \nabla) u + \nabla \pi & = \theta + G(u) \,  dW\quad \mbox{\rm in } \quad (0,T)\times D,\\
 \partial_t \theta - \kappa \Delta \theta + (u\cdot \nabla \theta) &= \tilde{G} (\theta)\,  d\tilde{W} \quad \mbox{\rm in } \quad (0,T)\times D, 
 \label{temperature}\\
 \mbox{\rm div }u&=0 \quad \mbox{\rm in } \quad (0,T)\times D,  \nonumber 
 \end{align}
 where  $T>0$. The processes 
$u: \Omega\times (0,T)\times D  \to \RR^2$, and $\theta: \Omega\times (0,T)\times D  \to \RR$ 
 have  initial conditions $u_0$ and $ \theta_0$ in $D$ respectively. 
The  parameters $\nu>0$ denotes  the  kinematic viscosity of the fluid, and $\kappa>0$ its thermal diffusivity.
These fields satisfy 
  periodic boundary conditions $u(t,x+L v_i)=u(t,x)$, $\theta (t,x+L v_i) = \theta(t,x)$ on  $(0,T)\times \partial D$, 
 where $v_i$, $i=1,2$
 denotes the canonical basis of $\RR^2$, and $\pi : \Omega\times (0,T)\times D  \to \RR$  is  the  pressure.

 In dimension 2 without any stochastic perturbation, this system has been extensively studied with a complete picture
  about its well posedness and longtime behavior.
 In the deterministic setting,  more investigations have been extended to the cases where $\nu=0$ and/or $\kappa=0$ with some partial results.
  
  If the $(L^2)^2$ (resp. $L^2$) norms of $u_0$ and $\theta_0$ are square integrable, it is known that the random system 
 \eqref{velocity}--\eqref{temperature} is well-posed, 
 and that there exists a unique solution $(u\times \theta)$ in
 $C([0,T];(L^2)^2 \times L^2) \cap L^2(\Omega ; (H^1)^2 \times H^1)$; see e.g. \cite{ChuMil} and \cite{DuaMil}. 

Numerical schemes and algorithms have been introduced to best approximate the solution to non-linear PDEs. The time approximation
is either an implicit Euler or a time splitting scheme coupled with Galerkin approximation or finite elements to approximate the space variable. 
The literature on numerical analysis for SPDEs is now very extensive. 
In many papers  the models are either linear, have global Lipschitz properties,
 or more generally some monotonicity property. In this case the convergence is proven to be in mean square. 
When nonlinearities are involved that are not of Lipschitz or monotone type,  then a rate of  convergence 
 in mean square is more difficult to  obtain.  
  Indeed, because of the stochastic perturbation, one may not use the Gronwall lemma after taking the expectation of the
   error bound   
  since it involves a nonlinear term which is often  quadratic; such a nonlinearity requires some localization.

In a random setting, the discretization of the Navier-Stokes equations on the torus has been intensively investigated.  
Various space-time numerical schemes have been studied for the stochastic Navier-Stokes equations with a multiplicative or an additive noise, that is 
where in the
right hand side of  \eqref{velocity} (with no $\theta$) we have either $G(u) \, dW$  or $dW$.
 We refer to \cite{BrCaPr,Dor, Breckner, CarPro, BreDog}, where convergence in probability is stated with various rates of convergence
  in a multiplicative setting for a time Euler scheme, and \cite{BeBrMi} for a time splitting scheme. 
 As stated previously, the main tool to get the convergence in probability  is the localization of the nonlinear term over a space of large probability.
  We studied the strong (that is  $L^2(\Omega)$) rate of convergence of the time  implicit Euler scheme
   (resp. space-time implicit Euler scheme coupled with
  finite element space discretization) in our previous papers \cite{BeMi_multi}  (resp.  \cite{BeMi_FEM}) for an $H^1$-valued initial condition. 
  The method is based on the fact that the solution (and the scheme) have finite  moments (bounded uniformly on the mesh). 
  For a general multiplicative noise, the rate is logarithmic. When the diffusion coefficient is bounded (which is a slight extension of an additive noise),
  the supremum of the $H^1$-norm of the solution has exponential  moments;  we used this property in \cite{BeMi_multi} and \cite{BeMi_FEM} 
  to  get an explicit polynomial strong rate
   of convergence. However, this rate depends on the viscosity and the strength of the noise, and is strictly less than 1/2 for the time parameter 
   (resp. than $1$ for the
   spatial one). For a given viscosity, the time rates on convergence increase to 1/2 when the strength of the noise converges to 0. For an additive noise,
  if the strength of the noise is not too large,  the strong ($L^2(\Omega)$) rate of convergence in time is the optimal one, that is almost $1/2$ (see
  \cite{BeMi_additive}).  Once more this is based on exponential moments of the supremum of the
  $H^1$-norm of the solution (and of its  scheme for the space discretization); this enabled us to have strong polynomial time rates. 
\smallskip

In the current paper, we study the time approximation of the Boussinesq equations \eqref{velocity}-\eqref{temperature} in a multiplicative setting.  
 To 
the best of our knowledge,  it is the first result where a time-numerical scheme is implemented 
for a more general hydrodynamical model with a multiplicative noise.
We use a fully implicit time Euler scheme and once more have to assume that the initial conditions $u_0$ and $\theta_0$ belong to $H^1(D)$
in order to prove a rate of convergence in $L^2(D)$ uniformly in time. We prove the existence of finite moments of the $H^1$-norms of the 
velocity and the temperature uniformly in time. Since we are on the torus, this is quite easy for the velocity. However, for the temperature, due
to the presence of the velocity in the bilinear term,  the
argument is more involved and has to be done in two steps. It requires higher moments on the $H^1$-norm of the initial condition. 
The time regularity  of the solutions $u, \theta$ is the same as that of  $u$ in the Navier-Stokes equations.
We then study rates of convergence in probability and in $L^2(\Omega)$. 
 The rate of convergence in probability is  optimal (almost $1/2$); we have to impose higher moments on the initial conditions
than what is needed for the velocity described by a stochastic Navier-Stokes equations. Once more we first obtain an $L^2(\Omega)$-convergence
on a set where we bound the $L^2$ norm of the gradients of both the velocity and the temperature. We deduce an optimal rate of convergence
in probability, that is strictly less than 1/2. When $H^1$-norm of the initial condition has all moments  (for example it is a Gaussian  $H^1$-valued 
random variable), the rate of convergence in $L^2(\Omega)$ is any negative
exponent of the logarithm of the number of time steps. These results extend those established for the Navier-Stokes equations subject to a 
multiplicative stochastic perturbation. 

 The paper is organized as follows. In section \ref{preliminary} we describe the model, the assumptions on the noise and the diffusion coefficients, and describe the fully implicit time Euler scheme. 
In section \ref{main} we state the global well-posedeness of the solution to  \eqref{velocity}--\eqref{temperature},
  moment estimates  of  the gradient of $u$ and $\theta$ uniformly in time, the existence of the scheme. We then formulate the main results
 of the paper about the rates of convergence in probability and  in $L^2(\Omega)$ of the scheme to the solution. 
 In section \ref{s_regularity_ut}  we prove  moment estimates in $H^1$  of $u$ and $\theta$ uniformly on the time interval $[0,T]$
  if we start with  more
regular ($H^1$) initial conditions. This is essential to be able
to deduce a rate of convergence from the localized result. Section \ref{s-increments}  states time regularity
results of the solution $(u,\theta)$ both in $L^2(D)$ and $H^1(D)$; this a crucial ingredient of the final results.
 In section \ref{sEuler} we prove that the time Euler scheme  is well defined 
and prove its moment estimates in $L^2$ and $H^1$. Section \ref{s-loc-converg} deals with the localized convergence of the scheme 
 in $L^2(\Omega)$. This preliminary step is necessary due to the bilinear term, which requires some control of the $H^1$ norm of $u$ and $\theta$.
  In section \ref{sec_speed} we prove
the rate of convergence in probability and in $L^2(\Omega)$.  Finally, section \ref{sconclusion} summarizes the interest of the model, 
and describes some further necessary/possible extensions of this work.

As usual, except if specified otherwise, $C$ denotes a positive constant that may change  throughout the paper,
 and $C(a)$ denotes  a positive constant depending on some parameter $a$. 

\section{Preliminaries and assumptions}\label{preliminary} 
In this section, we describe the functional framework, the driving noise, the evolution equations, 
and the fully implicit time Euler scheme. 

\subsection{The functional framework}\label{frame}
Let  $D= [0,L]^2$  with periodic boundary conditions, 
${\mathbb L}^p:=L^p(D)^2$ (resp. ${\mathbb W}^{k,p}:=W^{k,p}(D)^2$)   be  the usual Lebesgue and Sobolev spaces of 
vector-valued functions
endowed with the norms  $\|\cdot \|_{{\mathbb L}^p}$  (resp. $\|\cdot \|_{{\mathbb W}^{k,p}}$). 

Let $V^0:=\{ u\in \LL^2 : {\rm div}(u)=0\;  {\rm on} \; D\}$.    Let  $\Pi : \mathbb{L} ^2 \to V^0$ denote the Leray projection, 
and let   $A=- \Pi \Delta$ denote the Stokes operator, with  domain  
  $\mbox{\rm Dom}(A)={\mathbb W}^{2,2}\cap V^0 $. 

Let $\tilde{A} = -\Delta$ acting on $L^2(D)$.  For any non negative real number $k$ let
\[  H^k={\rm Dom}\big(\tilde{A}^{\frac{k}{2}}\big) , \;   V^k={\rm Dom} \big(A^{\frac{k}{2}}\big),\; \mbox{\rm endowed with the norms }
 \|\cdot\|_{H^k} \; \mbox{\rm and } \|\cdot\|_{V^k}.\]
Thus $H^0=L^2(D)$  and $H^k=W^{k,2}$. 
 Moreover,   let  $V^{-1}$ be the dual space of $V^1$ 
   with respect to the  pivot space $V^0$,  and 
  $\langle\cdot,\cdot\rangle$ denotes the duality between $V^1$ and $V^{-1}$.

  
  Let $b:(V^1)^3 \to \RR$ denote the trilinear map defined by 
  \[ b(u_1,u_2,u_3):=\int_D  \big(\big[ u_1(x)\cdot \nabla\big] u_2(x)\big)\cdot u_3(x)\, dx. \]
  The incompressibility condition implies   $b(u_1,u_2, u_3)=-b(u_1,u_3,u_2)$  for $u_i \in V^1$, $i=1,2,3$. 
  There exists a continuous bilinear map $B:V^1\times V^1 \mapsto
  V^{-1}$ such that
  \[ \langle B(u_1,u_2), u_3\rangle = b(u_1,u_2,u_3), \quad \mbox{\rm for all } \; u_i\in V^1, \; i=1,2,3.\]
  Therefore, the map  $B$ satisfies the following antisymmetry relations:
  \begin{equation} \label{B}
  \langle B(u_1,u_2), u_3\rangle = - \langle B(u_1,u_3), u_2\rangle , \quad \langle B(u_1,u_2),  u_2\rangle = 0 
  \qquad \mbox {\rm for all } \quad u_i\in V^1.
  \end{equation}
  For $u,v\in V^1$,  we have $B(u,v):= \Pi \big( \big[ u\cdot \nabla\big]  v\big) $.
  
  Furthemore, since $D=[0,L]^2$ with periodic boundary conditions, we have (see e.g. \cite{Temam-84})
  \begin{equation}		\label{B(u)Au}
  \langle B(u,u), Au\rangle =0, \quad \forall  u\in V^2.
  \end{equation}
  Note that for $u\in V^1$ and $\theta_1, \theta_2\in H^1$, if $(u.\nabla)\theta = \sum_{i=1,2} u_i \partial_i \theta$, we have
  \begin{equation} 	\label{antisym-ut}
   \langle [u.\nabla]\theta_1\, , \, \theta_2 \rangle = - \langle [u.\nabla]\theta_2\, , \, \theta_1\rangle,
   \end{equation}
  so that $\langle [u.\nabla]\theta\, , \, \theta \rangle =0$ for $u\in V^1$ and $\theta\in H^1$.

  In dimension 2 the inclusions  $H^1\subset L^p$  and $V^1\subset \LL^p$  for $p\in [2,\infty)$ follow from the Sobolev embedding theorem. 
  More precisely 
  the following Gagliardo Nirenberg inequality is true for some constant $\bar{C}_p$
    \begin{equation} \label{GagNir}
  \|u\|_{\LL^p} \leq \bar{C}_p \;   \|A^{\frac{1}{2}}  u\|_{\LL^2}^{\alpha}  \|u\|_{\LL^2}^{1-\alpha}  \quad \mbox{\rm for} \quad
    \alpha = 1-\frac{2}{p}, \quad \forall u\in V^1.
  \end{equation} 

Finally, let us recall the following estimate of the bilinear terms $(u.\nabla)v$ and $(u.\nabla)\theta$.
\begin{lemma} \label{GiMi}
Let $\alpha,  \rho$ be positive numbers and $\delta \in [0,1)$ be such that  $\delta +\rho > \frac{1}{2}$ and $\alpha + \delta +\rho \geq 1$. 
 Let  $u\in V^\alpha$,  $v\in V^\rho$ and $\theta \in H^\rho$; then 
\begin{align} 	\label{GiMi-uv}
\| A^{-\delta} \Pi [(u.\nabla) v] \|_{V^0} &\leq C \| A^\alpha u\|_{V^0}\, \|A^\rho v\|_{V^0}, \\
\| \tilde{A}^{-\delta} [(u.\nabla) \theta]\|_{H^0} &\leq C \| A^\alpha u\|_{V^0}\, \|\tilde{A}^\rho \theta\|_{H^0}, \label{GiMi-ut}
\end{align} 
   for some positive constant $C:= C(\alpha, \delta, \rho)$.
\end{lemma}
\begin{proof}
The upper estimate \eqref{GiMi-uv} is Lemma 2.2 in \cite{GigMiy}. The argument, which is based on the Sobolev embeding theorem and
H\"older's inequality, clearly proves \eqref{GiMi-ut}. 
\end{proof}

\subsection{The stochastic perturbation}
Let $K$ (resp. $\tilde K$) be a Hilbert space 
 and let $(W(t), t\geq 0)$ (resp. $(\tilde{W}(t), t\geq 0)\, )$ be a $K$-valued (resp. $\tilde{K}$- valued) 
 Brownian motion with covariance $Q$ (resp. $\tilde{Q}$),
 that is a trace-class operator of $K$ (resp. $\tilde{K}$)  such that $Q \zeta_j = q_j \zeta_j$ (resp. $\tilde{Q} \tilde{\zeta}_j = \tilde{q}_j \tilde{\zeta}_j$), where
  $\{ \zeta_j\}_{j\geq 0}$ (resp. $\{ \tilde{\zeta}_j\}_{j\geq 0}$) is a complete orthonormal system of $K$ (resp. $\tilde{K}$), $q_j, \tilde{q}_j >0$,
  and ${\rm Tr}(Q)=\sum_{j\geq 0} q_j <\infty$
  (resp. ${\rm Tr}(\tilde{Q})=\sum_{j\geq 0} \tilde{q}_j <\infty$). Let $\{ \beta_j \}_{j\geq 0}$ (resp. $\{ \tilde{\beta}_j\}_{j\geq 0}$) be a sequence of independent
  one-dimensional Brownian motions on the same filtered probability space $(\Omega, {\mathcal F}, ({\mathcal F}_t, t\geq 0), \PP)$. Then
  \[ W(t)=\sum_{j\geq 0} \sqrt{q_j}\, \beta_j(t) \,\zeta_j, \quad \tilde{W}(t)=\sum_{j\geq 0} \sqrt{\tilde{q}_j }\, \tilde{\beta}_j \,\tilde{\zeta}_j.\] 
   For details concerning these Wiener processes  we refer  to \cite{DaPZab}.

 Projecting the velocity on divergence free fields,  we consider the following SPDEs for processes modeling  the velocity  $u(t)$ 
 and the temperature $\theta(t) $. The initial conditions $u_0$ and $\theta_0$ are ${\mathcal F}_0$-measurable, taking values in  
  $V^0$ and $H^0$ respectively, and
\begin{align}
\partial_t u(t) + \big[ \nu\, A u(t) + B(u(t),u(t)) \big] dt = &\, \Pi (\theta(t) v_2) + G(u(t)) \, dW(t), \label{def_u} \\
\partial_t \theta(t) + \big[\kappa \,\AA \theta(t) + (u(t).\nabla) \theta(t)\big] dt =&\, \tilde{G}(\theta(t)) \, d\WW(t), 	\label{def_t}
\end{align}
$\nu,\kappa$ are strictly positive constants, and  $v_2=(0,1)\in \RR^2$.

 We make the following classical linear growth and Lipschitz assumptions on the diffusion coefficients $G$ and $\GG$.
 For technical reasons, we will have to require $u_0\in V^1$, $\theta\in H^1$ and prove estimates similar to \eqref{mom_u_L2}--\eqref{mom_t_L2}
raising the space regularity of the processes by one step in the scale of Sobolev spaces. Therefore, we have to strengthen the regularity
 of the diffusion coefficients. \\
 {\bf Condition (C-u) (i)} Let $G:V^0\to {\mathcal L}(K;V^0)$ be such that
 \begin{align}  \|G(u)\|^2_{{\mathcal L}(K,V^0)} \leq&\,  K_0 + K_1 \|u\|^2_{V^0}, \quad \forall u\in V^0, 	\label{growthG-0}\\
  \|G(u_1)-G(u_2)\|^2_{{\mathcal L}(K,V^0)} \leq &\,  L_1  \|u_1-u_2|^2_{V^0} , \quad  \forall u_1,u_2\in V^0.	\label{LipG}
 \end{align}
{\bf  (ii)} Let also   $G:V^1\to {\mathcal L}(K;V^1)$ satisfy the growth condition
\begin{equation}		\label{growthG-1}
\| G(u)\|_{{\mathcal L}(K;V^1)}^2 \leq K_2 + K_3 \|u\|_{V^1}^2, \quad \forall u\in V^1.
\end{equation}

and\\

\noindent {\bf Condition  (C-$\theta$) (i)} Let $\GG:H^0\to {\mathcal L}(\KK;H^0)$ be such that
 \begin{align}  \|\GG(\theta)\|^2_{{\mathcal L}(\KK,H^0)} \leq&\,  \tilde{K}_0 + \tilde{K}_1 \|\theta\|^2_{H^0}, \quad \forall \theta \in H^0, 	
 \label{growthGG-0}\\
  \|\GG(\theta_1)-\GG(\theta_2)\|^2_{{\mathcal L}(\KK,H^0)} \leq &\,  \tilde{L}_1  \|\theta_1-\theta_2 \|^2_{H^0} , \quad \forall \theta_1, \theta_2\in H^0.	\label{LipGG}
 \end{align}
 {\bf (ii)} Let also $\GG:H^1\to {\mathcal L}(K;H^1)$
 satisfy the growth condition
 \begin{equation}		\label{growthGG-1}
\| \GG(\theta)\|_{{\mathcal L}(\tilde{K};H^1)}^2 \leq \KK_2 + \KK_3 \|\theta\|_{H^1}^2, \quad \forall \theta\in H^1.
\end{equation}

 \subsection{The fully implicit time Euler scheme} \label{ssEuler}
 Fix $N\in \{1,2, ...\}$, let $h:=\frac{T}{N}$ denote the time mesh,
 and for $j=0, 1, ..., N$ set $t_j:=j \frac{T}{N}$. 

The fully implicit time Euler scheme $\{ u^k ; k=0, 1, ...,N\}$  and $\{ \theta^k ; k=0, 1, ...,N\}$ is defined by $u^0=u_0$, $\theta^0=\theta_0$,
 and for $\varphi \in V^1$, $\psi\in H^1$ and $l=1, ...,N$, 
\begin{align}	\label{Euler_u}
\Big( u^l-u^{l-1} &+ h \nu A u^l + h B\big( u^l,u^l\big)   , \varphi\Big) = \big(\Pi \theta^{l-1} v_2, \varphi) h  \nonumber \\
& + \big( G(u^{l-1}) [W(t_l)-W(t_{l-1})]\, , \, \varphi),  \\
\Big( \theta^l-\theta^{l-1} &+ h \kappa \AA \theta^l +h [ u^{l-1}. \nabla]\theta^l , \psi\Big) = \big( \GG(\theta^{l-1}) [\WW(t_l)-\WW(t_{l-1})]\, , \, \varphi). 
\label{Euler_t}
\end{align}
\medskip

 \section{Main results} \label{main}
 In this section, we state the main  results about well-posedness of the solutions $(u,\theta)$, the scheme $\{ u^k ; k=0, 1, ...,N\}$, and 
 the rate of convergence of the scheme $\{ (u^k, \theta^k) ; k=0, 1, ...,N\}$  
 to $(u,\theta)$. 
 
 \subsection{Global well-posedness and moment estimates of $(u,\theta)$}	\label{gwp}
   The first results states the existence and uniqueness of   a   weak pathwise solution (that is strong probabilistic solution in the weak deterministic sense) 
  of \eqref{def_u}-\eqref{def_t}. 
It is proven in  \cite{ChuMil} (see also \cite{DuaMil}).
 \begin{theorem}		\label{th-gwp}
 Let $u_0\in L^{2p}(\Omega;V^0)$, $\theta_0 \in L^{2p}(\Omega;H^0)$ for $p=1$ or $p\in [2,\infty)$. Let the coefficients $G$ and $\GG$
 satisfy the conditions {\bf (C-u)(i)} and {\bf (C-$\theta$)(i)} respectively. Then equations \eqref{def_u}--\eqref{def_t} have a unique pathwise solution, i.e.,
 \begin{itemize}
 \item  $u $ (resp. $\theta$) is an adapted $V^0$-valued (resp. $H^0$-valued) process which belongs a.s. to $L^2(0,T ; V^1)$
 (resp. to $L^2(0,T;H^1)$),
 \item $\PP$ a.s. we have  $u\in C([0,T];V^0)$, $\theta \in C(0,T];H^0)$, and 
  \begin{align*}
  \big(u(t), \varphi\big) +& \nu \int_0^t \big( A^{\frac{1}{2}} u(s), A^{\frac{1}{2}} \varphi\big) ds 
  + \int_0^t \big\langle [u(s) \cdot \nabla]u(s), \varphi\big\rangle ds \\
 &    = \big( u_0, \varphi) + \int_0^t \big( \Pi \theta(t) v_2,\varphi) ds + \int_0^t \big( \varphi ,  G(u(s)) dW(s) \big),\\
  \big(\theta(t), \psi\big) +& \kappa \int_0^t \big( \AA^{\frac{1}{2}} \theta(s), \AA^{\frac{1}{2}} \psi\big) ds 
  + \int_0^t \big\langle [u(s) \cdot \nabla]\theta(s), \psi\big\rangle ds \\
 &    = \big( \theta_0, \psi) + \int_0^t \big( \phi ,  \GG(\theta(s)) d\WW(s) \big),
 \end{align*}
for every $t\in [0,T]$ and every $\varphi \in V^1$ and $\psi\in H^1$.
  \end{itemize}
Furthermore, 
 \begin{align}		\label{mom_u_L2}
 \EE\Big( \sup_{t\in [0,T]} \|u(t)\|_{V^0}^{2p} + \int_0^T \|A^{\frac{1}{2}} u(t)\|_{V^0}^2 \, \big[ 1+\|u(t)\|_{V^0}^{2(p-1)}\big] dt \Big) 
 &\leq C\big(1+\EE(\|u_0\|_{V^0}^{2p}\big), \\
 \EE\Big( \sup_{t\in [0,T]} \|\theta(t)\|_{H^0}^{2p} + \int_0^T \|\AA^{\frac{1}{2}} \theta(t)\|_{H^0}^2 \, \big[ 1+\|\theta(t)\|_{H^0}^{2(p-1)}\big] dt \Big) 
 &\leq C\big(1+\EE(\|\theta_0\|_{H^0}^{2p}\big).		\label{mom_t_L2}
 \end{align}
 \end{theorem}
 


The following result proves that if $u_0\in V^1$, the solution $u$ to  \eqref{def_u}--\eqref{def_t} is more regular.
\begin{prop}		\label{prop_u_V1}
Let $u_0\in L^{2p}(\Omega;V^1)$, $\theta_0\in L^{2p}(\Omega;H^0)$ for $p=1$ or some $p\in [2,\infty)$,  and let $G$ satisfy condition {\bf (C-u)} and
$\GG$ satisfy condition {\bf (C-$\theta$)}.
 Then the solution $u$ to  \eqref{def_u}--\eqref{def_t} belongs a.s. to $ C([0,T];V^1)  \cap L^2([0,T];V^2)$. Moreover, for some constant $C$
 \begin{equation}		\label{mom_u_V1}
\EE\Big( \sup_{t\in [0,T]} \|u(t)\|_{V^1}^{2p} + \int_0^T \!\! \| A u(t)\|_{V^0}^2 \,\big[ 1+  \| A^{\frac{1}{2}} u(t) \|_{V^0}^{2(p-1)}\big] dt \Big) 
 \leq C  \big[ 1+\EE\big( \|u_0\|_{V^1}^{2p} + \|\theta_0\|_{H^0}^{2p}\big)\big].
 \end{equation}
\end{prop}

The next result proves similar bounds for moments of the gradient of the temperature, uniformly in time.
\begin{prop}		\label{lem_mom_t_H1}
Let $u_0\in L^{8p+\epsilon}(\Omega;V^1)$ and $\theta_0\in L^{8p+\epsilon}(\Omega;H^1)$ for some $\epsilon >0$ and $p=1$ or $p\in [2,+\infty)$.
Suppose that the coefficients $G$ and $\GG$ satisfy the conditions {\bf (C-u)} and {\bf (C-$\theta$)}. 
There exists a constants $C$ 
such that 
\begin{equation} 		\label{mom_t_H1}
\EE\Big[ \sup_{t\leq T}  
\|\AA^{\frac{1}{2}}  \theta(t)\|_{H^0}^{2p} +  \int_0^T\| \AA \theta(s)\|_{H^0}^2 \| \AA^{\frac{1}{2}} \theta(s)\|_{H^0}^{2(p-1)} ds 
 \Big] \leq C. 
\end{equation}
\end{prop}

\subsection{Global well-posedness 
of the time Euler scheme}
The following proposition states the existence and uniqueness of the sequences $\{ u^k\}_{k=0, ...,N}$  and $\{ \theta^k\}_{k=0, ..., N}$.
\begin{prop} 		\label{prop_ul}		
Let  Condition {\bf (G-u)(i)} and {\bf (C-$\theta$)(i)} be satisfied, $u_0\in V^0$ and $\theta^0\in H^0$ a.s.
 The time fully  implicit scheme \eqref{Euler_u}--\eqref{Euler_t}  has a unique 
 solution $\{u^l\}_{l=1, ...,N} \in V^1$,  $\{\theta^l\}_{l=1, ...,N} \in H^1$.
 \end{prop}
\subsection{Rates of convergence in probability and in $L^2(\Omega)$}
The following theorem states that   the implicit time Euler scheme converges to the pair $(u,\theta)$ in probability with the ``optimal" 
rate 
``almost 1/2". It is the main result of the paper. \\
For $j=0, ...,N$ set $e_j:= u(t_j)-u^j$ and $\tilde{e}_j:= \theta(t_j)-\theta^j$; then $e_0=\tilde{e}_0=0$.  
\begin{theorem} 		\label{th_cv_proba}
Suppose that  the  conditions {\bf (C-u)} and {\bf (C-$\theta$)} hold. 
Let  $u_0\in L^{32+\epsilon}(\Omega ; V^1)$ and $\theta_0\in L^{32+\epsilon}(\Omega;H^1)$ for some $\epsilon>0$,
  $u,\theta$ be the solution to \eqref{def_u}--\eqref{def_t},
$\{u^j, \theta^j\}_{j=0, ..., N}$ be the solution to \eqref{Euler_u}--\eqref{Euler_t}. 
 Then for every $\eta \in (0,1)$ we have
\begin{equation}		\label{speed_proba} 
\lim_{N\to \infty} P\Big( \max_{1\leq J\leq N} \big[ \|e_J\|_{V^0}^2 + \|\ee_J\|_{H^0}^2 \big] + \frac{T}{N} \sum_{j=1}^N \big[ \|A^{\frac{1}{2}} e_j\|_{V^0}^2
+ \| \AA^{\frac{1}{2}} \ee_j \|_{H^0}^2 \big] \geq  N^{-\eta} \Big) = 0 .
\end{equation} 
\end{theorem} 

 We finally state  that
the strong (i.e. in $L^2(\Omega)$) rate of convergence  of the implicit time Euler scheme is some negative exponent of $\ln N$.
Note that if the initial conditions $u_0$ and $\theta_0$ are deterministic, or if their $V^1$ and $H^1$-norms have moments of all orders (for example
if $u_0$ and $\theta_0$ are Gaussian random variables), the strong rate of convergence is any   negative exponent of $\ln N$.
 More precisely we have the following
result.
\begin{theorem}		\label{th_strong_rate}
Suppose that  the  conditions {\bf (C-u)} and {\bf (C-$\theta$)(i)}  hold.  
Let $u_0\in L^{2^q+\epsilon}(\Omega;V^1)$ and $\theta_0\in L^{2^q+\epsilon}(\Omega;H^1)$ for $q\in [5,\infty)$ and some $\epsilon >0$. 
Then for  some constant $C$ such that  
\begin{align}	\label{strong_rate}
\EE\Big( \max_{1\leq J\leq N}& \big[ \|e_J\|_{V^0}^2 + \|\ee_J\|_{H^0}^2 \big] + \frac{T}{N} \sum_{j=1}^N \big[ \|A^{\frac{1}{2}} e_j\|_{V^0}^2
+ \| \AA^{\frac{1}{2}} \ee_j\|_{H^0}^2
\Big) 		
 \leq C  \big( \ln (N) \big) ^{ - (2^{q-1}+1)} 
\end{align}
for  $N$ large enough.
\end{theorem}

\section{More egularity of the solution} \label{s_regularity_ut}
{
\subsection{Moments of $u$ in $L^\infty(0,T;V^1)$ }
In this section, we prove that if $u_0\in V^1$ and $\theta_0\in H^0$, the $H^1$-norm of the velocity has bounded moments uniformly in time.

\noindent {\it Proof of Proposition  \ref{prop_u_V1}}
Apply the operator $A^{\frac{1}{2}}$ to \eqref{def_u} and use (formally) It\^o's formula for the square of the $\|.\|_{V^0}$-norm of $A^{\frac{1}{2}} u(t)$.
Then, using \eqref{B(u)Au}, we obtain 
\begin{align}		\label{Ito-u}
\|A^{\frac{1}{2}}u(t)\|_{V^0}^2 &\, + 2\nu \int_0^t \|A u(s)\|_{V^0}^2 \, ds = \| A^{\frac{1}{2}} u_0\|_{V^0}^2 + 2 \int_0^t \big( A^{\frac{1}{2}} \Pi \theta(s) v_2,
A^{\frac{1}{2}} u(s)\big) ds\\& + 2 \int_0^t \big( A^{\frac{1}{2}} G(u(s)) dW(s), A^{\frac{1}{2}} u(s)\big) 
+ \int_0^t \| A^{\frac{1}{2}} G(u(s))\|^2_{{\mathcal L}(K;V^0)} \, {\rm Tr}(Q)  ds. \nonumber 
\end{align}
Let $ \tau_M:= \inf\{ t : \|u(t)\|_{V^1} \geq M\} $; using \eqref{growthG-1}, integration by parts, the Cauchy-Schwarz and Young inequalities, we deduce 
for $M>0$ and $t\in [0,T]$
\begin{align*}
\EE\Big( \|A^{\frac{1}{2}} &u(t\wedge \tau_M )\|_{V^0}^2 +2\nu \int_0^{t\wedge \tau_M} \|A u(s)\|_{V^0}^2 ds \Big) \leq \EE\big( \|u_0\|_{V^0}^2 \big) \\
& +
2 \EE\Big( \int_0^{t\wedge \tau_M} \|\theta(s)\|_{H^0} \, \| A u(s)\|_{V^0} ds + {\rm Tr}(Q) \EE\Big( \int_0^{t\wedge \tau_N} \big[ K_2 + K_3 \|u(s)\|_{V^1}^2
\big] ds \Big) \\
\leq & \EE\Big( \|u_0\|_{V^0}^2   + \nu \int_0^{t\wedge \tau_M} \|A u(s)\|_{V^0}^2 ds \Big) 
+ \frac{1}{\nu} \EE\Big( \int_0^{t\wedge \tau_M} \|\theta(s)\|_{H^0}^2 ds \Big) + K_2 T \\
& + K_3 T \EE\Big( \sup_{t\in [0,T]} \|u(t)\|_{V^0}^2\Big)
+ K_3 \int_0^t \EE\Big( \| A^{\frac{1}{2}} u(s\wedge \tau_M)\|_{V^0}^2 \Big) ds.
\end{align*}
Indeed the stochastic integral in the right hand side of \eqref{Ito-u} is a square integrable - hence centered - martingale. 
Neglecting the time integral in the left hand side, using \eqref{mom_u_L2} and  the Gronwall lemma, we deduce
\begin{equation}		\label{mom_1_Au}
\sup_M  \sup_{t\in [0,T]} \EE\Big( A^{\frac{1}{2}} \|u(t\wedge \tau_M)\|_{V^0}^2 \Big) \leq C <\infty.
\end{equation}
As $M\to \infty$, this implies $\EE\big( \int_0^{T}  \|A u(s)\|_{V^0}^2 ds\big) <\infty$. 

Furthermore, the Davis inequality and Young's inequality imply
\begin{align*}
\EE\Big( \sup_{s\leq t} &\int_0^{s\wedge \tau_M} \big( A^{\frac{1}{2}} G(u(r)) dW(r), A^{\frac{1}{2}} u(r)\big) \Big) \\
&  \leq 3 \EE \Big( \Big\{ \int_0^{t} \|A^{\frac{1}{2}} u(r\wedge \tau_M)\|_{V^0}^2 \; {\rm Tr}(Q)
\,  \| A^{\frac{1}{2}} G(u(r\wedge \tau_M))\|_{{\mathcal L}(K;V^0)}^2 dr \Big\}^{\frac{1}{2}} \Big)\\
&\leq   3 \EE \Big(\sup_{s\leq t} \|A^{\frac{1}{2}} u(s\wedge \tau_M)\|_{V^0} \Big\{ {\rm Tr}(Q) \int_0^t
 [K_2+K_3 \|u(s\wedge \tau_M )\|_{V^1}^2] ds \Big\}^{\frac{1}{2}} \Big)\\
&\leq \frac{1}{2} \,  \EE \Big(\sup_{s\leq t} \|A^{\frac{1}{2}} u(s\wedge \tau_M)\|_{V^0}^2\Big) + 9 {\rm Tr}(Q) \EE\Big( \int_0^t [K_2+K_3 \|u(r\wedge \tau_M)\|_{V^1}^2 ds\Big).
\end{align*}
The upper estimates \eqref{Ito-u}, \eqref{mom_u_L2}, \eqref{mom_t_L2} and \eqref{mom_1_Au} imply for some constant $C$ depending on
$ \EE(\int_0^T \big[ \|u(t)\|_{V^0}^2 + \|A^{\frac{1}{2}} u(t)\|_{V^0}^2 + \|\theta(t)\|_{H^0}^2\big] ds \big)<\infty$, 
\begin{align*}
 \sup_M \EE\Big( \frac{1}{2}  \sup_{t\leq T} &  \|A^{\frac{1}{2}} u(t\wedge \tau_M) \|_{V^0}^2  + \int_0^{T\wedge \tau_M} \|Au(s)\|_{V^0}^2 ds \Big) \\
 & \leq C + 
C \EE\Big(\int_0^T \big[  \|A^{\frac{1}{2}} u(t)\|_{V^1}^2 + \|\theta(t)\|_{H^0}^2\big] ds \Big) <\infty .
\end{align*}
As $M\to \infty$,  we deduce
\[  \EE\Big(  \sup_{t\in [0,T]} \|A^{\frac{1}{2}} u(t)\|_{V^0}^2 \Big) + \EE \Big( \int_0^T \|A u(s)\|_{V^0}^2 ds\Big) \leq C <\infty. \]

this proves \eqref{mom_u_V1} for $p=1$.
Given $p\in [2,\infty)$ and using It\^o's formula for the map $x\mapsto x^p$ in \eqref{Ito-u}, we obtain
\begin{align} 		\label{Itou-p}
\| A^{\frac{1}{2}} &u(t\wedge \tau_M)\|_{V^0}^{2p} + 2p\nu \int_0^{t\wedge \tau_M}  \|A u(s)\|_{V^0}^2 \, \|A^{\frac{1}{2}} u(s)\|_{V^0}^{2(p-1)} ds
 = \|A^{\frac{1}{2} } u_0 \|_{V^0}^{2p} \nonumber \\
 &+ 2p \int_0^{t\wedge \tau_M} \big( A^{\frac{1}{2}} \Pi \theta(s) v_2, A^{\frac{1}{2}} u(s) \big) \, 
 \|A^{\frac{1}{2}} u(s)\|_{V^0}^{2(p-1)}  ds	\nonumber \\
 &+ 2p \int_0^{t\wedge \tau_M} \big( A^{\frac{1}{2}} G(u(s)) dW(s) , A^{\frac{1}{2}} u(s)\big) \, \|A^{\frac{1}{2}} u(s)\|_{V^0}^{2(p-1)}   	\nonumber \\
& + p {\rm Tr}(Q) \int_0^{t\wedge \tau_M} \| G(u(s)\|_{{\mathcal L}(K;V^1)}^2 \|A^{\frac{1}{2}} u(s)\|_{V^0}^{2(p-1)} ds	\nonumber \\
& + 2p(p-1) {\rm Tr}(Q) \int_0^{t\wedge \tau_M} \| \big(A^{\frac{1}{2}} G\big)^*(u(s)) \big(A^{\frac{1}{2}} u(s)\big)\|_K^2 
\| A^{\frac{1}{2}} u(s)\|_{V^0}^{2(p-2)} ds.
\end{align}
Integration by parts, the Cauchy-Schwarz, H\"older  and Young inequalities imply 
\begin{align}		\label{t-u-p}
\Big| \int_0^t& \!\!\big( A^{\frac{1}{2}} \Pi \theta(s) v_2, A^{\frac{1}{2}} u(s)\big) \|A^{\frac{1}{2}} u(s)\|_{V^0}^{2(p-1)} ds\Big| \leq 
\int_0^t\!\!  \|Au(s)\|_{V^0}   \|\theta(s)\|_{H^0}\, \| A^{\frac{1}{2}} u(s)\|_{V^0}^{2(p-1)} ds  \nonumber \\
&\leq \Big\{ \int_0^t  \|A u(s)\|_{V^0}^2  \|A^{\frac{1}{2}} u(s)\|_{V^0}^{ 2(p-1)} ds \Big\}^{\frac{1}{2}} 
\Big\{ \int_0^t \|\theta(s)\|_{H^0}^2 \|A^{\frac{1}{2}} u(s)\|_{V^0}^{2(p-1)} ds \Big\}^{\frac{1}{2}} \nonumber \\
&\leq \frac{p\nu}{2}   \int_0^t  \|A u(s)\|_{V^0}^2 \|A^{\frac{1}{2}} u(s)\|_{V^0}^{ 2(p-1)} ds  + 
\frac{1}{2p\nu}  \int_0^t \|\theta(s)\|_{H^0}^2 \|A^{\frac{1}{2}} u(s)\|_{V^0}^{2(p-1)} ds  \nonumber\\
& \leq \epsilon  \int_0^t \!  \|A u(s)\|_{V^0}^2 \|A^{\frac{1}{2}} u(s)\|_{V^0}^{ 2(p-1)} ds  +  C\! \int_0^t \! \|\theta(s)\|_{H^0}^{2p} ds +
C\! \int_0^t \! \| A^{\frac{1}{2}} u(s)\|_{V^0}^{2p} ds. 
\end{align}
Since $a^{p-1}\leq 1+a^p$ for any $a\geq 0$, the growth condition \eqref{growthG-1} implies
\begin{align}		\label{varq-u-p}
\int_0^t \|A^{\frac{1}{2}}&  G(u(s))\|_{{\mathcal L}(K,V^0)}^2 \|A^{\frac{1}{2}} u(s)\|_{V^0}^{2(p-1)} ds   \nonumber  \\
& \leq \int_0^t \big[ K_2+K_3 \|u(s)\|_{V^0}^2
+ K_3 \| A^{\frac{1}{2}} u(s)\|_{V^0}^2\big] \|A^{\frac{1}{2}} u(s)\|_{V^0}^{2(p-1)} ds  \nonumber \\
&\leq C\Big( T + \int_0^T \|u(s)\|_{V^0}^{2p} + \int_0^t \| A^{\frac{1}{2}} u(s)\|_{V^0}^{2p} ds\Big). 
\end{align}
Furthermore, since $\big( \|A^{\frac{1}{2}} G(u(s)) \big)^* A^{\frac{1}{2}} u(s) \|_{V^0}^2 \leq [K_2 + K_3 \|u(s)\|_{V^1}^2] \|A^{\frac{1}{2}} u(s)\|_{V^0}^2$,
the upper estimate of the corresponding integral is similar to that of \eqref{varq-u-p}. 
Since the stochastic integral $\int_0^{t\wedge \tau_M} \big( A^{\frac{1}{2}} G(u(s))  dW(s), A^{\frac{1}{2}} u(s) \big) \|u(s)\|_{V^0}^{2(p-1)}\big)$ 
is square integrable  it is centered. 
Therefore,\eqref{Itou-p} and the above upper estimates \eqref{t-u-p}--\eqref{varq-u-p}  imply
\begin{align*}
\sup_M \EE\Big( &\|A^{\frac{1}{2}} u(t\wedge \tau_M)\|_{V^0}^{2p} + p\nu \int_0^{t\wedge \tau_M}\|Au(s)\|_{V^0}^2 \|A^{\frac{1}{2}} u(s)\|_{V^0}^{2(p-1)} 
 \Big)  \\
 &\leq C\Big( T+\EE\Big( \int_0^t \big[ \|u(s)\|_{V^0}^{2p} + \|\theta(s)\|_{H^0}^{2p}\big] ds \Big) + 
 \int_0^t \EE\big( \|A^{\frac{1}{2}} u(s\wedge \tau_M)\|_{V^0}^{2p} \big) ds.
\end{align*} 
Gronwall's lemma implies 
\begin{align}
 \sup_M &\, \sup_{t\in [0,T]} \EE\big( \|A^{\frac{1}{2}} u(s\wedge \tau_M)\|_{V^0}^{2p}\big) =C<\infty, 	\label{mom-u-tau}\\\
\sup_M &\; \EE\Big( \int_0^{T\wedge \tau_M}
\|Au(s)\|_{V^0}^2 \|A^{\frac{1}{2}} u(s)\|_{V^0}^{2(p-1)} ds \Big) =C< \infty.		\label{mom-Au-tau}
\end{align}
Finally, the Davis inequality, then the H\"older and Young inequalities imply 
\begin{align}	\label{Davis-up}
\EE\Big(& \sup_{s\in [0,t]} 2p \Big| \int_0^{s\wedge \tau_M} \big( A^{\frac{1}{2}} G(u(r)) dW(r),A^{\frac{1}{2}} u(r) \| A^{\frac{1}{2}}u(r)\|_{V^0}^{2(p-1)} 
\Big|\Big)  \nonumber \\
&\leq 6p\,  \EE\Big( \Big\{ \int_0^{t\wedge \tau_M}  {\rm Tr}(Q) \|A^{\frac{1}{2}} G(u(s))\|_{{\mathcal L}(K;V^0)}^2 \|A^{\frac{1}{2}} u(s) \|_{V^0}^{4p-2}
ds \Big\}^{\frac{1}{2}} \Big) \nonumber \\
&\leq 6p\,  \big( { \rm Tr}(Q)\big)^{\frac{1}{2}} \EE\Big(  \sup_{s\leq t\wedge \tau_M} \|A^{\frac{1}{2}} u(s)\|_{V^0}^{p} \nonumber \\
&\qquad \qquad \times 
\Big\{ \int_0^t \|A^{\frac{1}{2}} G(u(s\wedge \tau_N))\|_{{\mathcal L}(K;V^0)}^2 
\|A^{\frac{1}{2}} u(s\wedge \tau_M) \|_{V^0}^{2p-2}
ds \Big\}^{\frac{1}{2}} \Big) \nonumber \\
&\leq \frac{1}{2} \EE\Big( \sup_{s\in [0,t\wedge \tau_M]}   \|A^{\frac{1}{2}} u(s)\|_{V^0}^{2p} \Big) +
C \EE\Big( 1+ \int_0^t \|u(s)\|_{V^0}^{2p} ds + \int_0^t \|A^{\frac{1}{2}} u(s)\|_{V^0}^{2p} ds\Big). 
\end{align}
The upper estimates \eqref{Itou-p}, \eqref{mom_u_L2} and \eqref{Davis-up} imply 
\[ \sup_M \EE\Big( \sup_{s\in [0, T\wedge \tau_M]} \|A^{\frac{1}{2}} u(s)\|_{V^0}^{2p} \Big) \leq C \Big[ 1+\sup_M  \EE\Big( \int_0^{T}  
\big[\|\theta(s\wedge \tau_M)\|_{H^0}^{2p} + \|u(s\wedge \tau_M)\|_{V^1}^{2p}\big]  ds\Big)\Big] <\infty. 
\]
As $M\to \infty$ in this inequality and in \eqref{mom-Au-tau}, the monotone convergence theorem  concludes the proof  of \eqref{mom_u_V1}. 
\hfill $\Box$

\subsection{Moment estimates of $\theta$ in $L^\infty(0,T;H^1)$}
We next give upper estimates for moments of $\sup_{t\in [0,T]} \|\AA^{\frac{1}{2}} \theta(t)\|_{H^0}$, i.e., prove Proposition \ref{lem_mom_t_H1}.\\
 However, since 
$\langle [u(s). \nabla] \theta(s),\AA \theta(s)\rangle \neq 0$, 
 unlike what we have in the proof of the previous result, we keep the bilinear term. 
This creates technical problems and we proceed in two steps. First, using the mild formulation of the weak solution $\theta$ of \eqref{def_t},
 we prove that the gradient of the temperature has finite moments. Then going back to the weak form, we prove the desired result.
 
Let  $\{S(t)\}_{t\geq 0}$ be the semi-group generated by $-\nu A$,  $\{\SS(t)\}_{t\geq 0}$ be the semi-group generated by $-\kappa \AA$, that
is $S(t)=\exp(-\nu t A)$ and $\SS(t)=\exp(-\kappa t\AA)$ for every $t\geq 0$. Note that  for every $\alpha >0$
\begin{align}
\| A^\alpha S(t) \|_{{\mathcal L}(V^0;V^0)} \leq C t^{-\alpha}, \quad\forall  t>0 	\label{AS}
 \\
\| A^{-\alpha} \big[ {\rm Id} - S(t)\big] \|_{{\mathcal L}(V^0;V^0)}
  \leq C t^{\alpha}, \quad \forall t>0.	\label{A(I-S)}
\end{align} 
 Similar upper estimates are valid when we replace $A$ by $\AA$, $S(t)$ by $\SS(t)$ and $V^0$ by $H^0$.
 
Note that  if $u_0\in L^2(\Omega;V^1)$ and $\theta_0\in L^2(\Omega;H^0)$,  $u\in L^{2}(\Omega ; C([0,T];V^0)\cap L^\infty( [0,T] ; V^1))$ 
and 
$\theta \in  L^{2}(\Omega ; C([0,T];H^0))\cap 
L^2(\Omega\times [0,T] ; H^1)$, we can write the solutions of \eqref{def_u}--\eqref{def_t} in the following mild form
\begin{align}
u(t) = &\, S(t) u_0 - \int_0^t S(t-s) B(u(s), u(s))\,  ds + \int_0^t S(t-s) \big(\Pi \theta(t) v_2\big) \,  ds  \nonumber \\
&+ \int_0^t S(t-s) G(u(s)) dW(s), 	\label{weak_u}\\
\theta(t) = &\, \tilde{S}(t)  \theta_0 - \int_0^t \SS(t-s) \big( [u(s) . \nabla] \theta(s)\big) \, ds + \int_0^t \SS(t-s) \GG(\theta(s)) d\WW(s), 	\label{weak_t}
\end{align}
where the first equality holds a.s. in $V^0$ and the second one in $H^0$.

Indeed, since $\|A^\alpha u\|_{V^0}\leq C  \|A^{\frac{1}{2}}  u\|_{V^0}^{2\alpha} \|u(s)\|_{V^0}^{1-2\alpha}$,
 the upper estimate \eqref{GiMi-uv} for $\delta +\rho >\frac{1}{2}$, $\delta+\alpha+\rho=1$ and the Minkowski inequality imply 
 \begin{align*}
\Big\| \int_0^t S(t-s)& B(u(s),u(s)) ds \Big\|_{V^0 }  \leq  \, \int_0^t \| A^{\delta}  A^{-\delta} B(u(s),u(s)) \|_{V^0} ds \\
& \leq C  \int_0^t (t-s)^{-\delta } \|A^\alpha u(s)\|_{V^0} \|A^\rho u(s)\|_{V^0} ds \\
& \leq C \,  \sup_{s\in [0,t]} \| u(s)\|_{V^1}^{2} \int_0^t (t-s)^{-\delta }  ds
 \end{align*}
Since $\|S(t)\|_{{\mathcal L}(V^0;V^0)} \leq 1$, it is easy to see that 
 \[ \Big\| \int_0^t S(t-s) \Pi \theta(t) v_2 ds\Big\|_{V^0} \leq C \int_0^t \|\theta(t)\|_{H^0} ds .\]
 Furthermore,
 \[ \EE\Big( \Big\| \int_0^t S(t-s) G(u(s)) dW(s) \Big\|_{V^0}^2 \Big) \leq {\rm Tr}(Q) \EE\Big( \int_0^t [K_0+K_1 \|u(t)\|_{V^0}^2\big] ds \Big) <\infty.
 \]
 Therefore, the stochastic integral  $\int_0^t S(t-s) G(u(s)) dW(s) \in V^0$ a.s., and the identity \eqref{weak_u} is true a.s. in $V^0$.
 
  A similar argument shows that \eqref{weak_t} holds a.s. in $H^{0}$. We only show that the convolution involving the bilinear term belongs to $H^0$.
  Using the Minkowski inequality and the upper estimate \eqref{GiMi-ut} with positive constants $\delta, \alpha, \rho $ such that 
  $\alpha, \rho\in (0,\frac{1}{2})$,   $\delta + \rho > \frac{1}{2}$ and $\delta + \alpha + \rho =1$, we obtain
  \begin{align*}
  \Big\|& \int_0^t \SS(t-s) [ (u(s). \nabla) \theta(s)] ds \Big\|_{H^0} \leq \int_0^t \| \AA^\delta  \SS(t-s) \; \AA^{-\delta} [ (u(s). \nabla) \theta(s)] \|_{H^0}\,  ds  \\
  &\leq C  \int_0^t (t-s)^{-\delta} \, \|A^\alpha u(s)\|_{V^0} \, \| \AA^\rho \theta(s)\|_{H^0} d\, s \\
  & \leq C  \sup_{s\in [0,t]} \|u(s)\|_{V^1}\sup_{s\in [0,t]} \|\theta(s)\|_{H^0}^{1-2\rho}  \Big( \int_0^t (t-s)^{-\frac{\delta }{1-\rho}} ds\Big)^{1-\rho}
 \Big(  \int_0^t \|\AA^{\frac{1}{2}} \theta(s)\|_{H^0}^2 ds \Big)^\rho <\infty,
  \end{align*}
  where the last upper estimate is deduced from H\"older's inequality and $\frac{\delta}{1-\rho}<1$.

The following result shows that for fixed $t$, the $L^2$-norm of the gradient of $\theta(t)$ has finite moments.
\begin{lemma}		\label{lem_sup_E_t}
Let  $p\in [0,+\infty)$, $u_0\in L^{4p + \epsilon}(\Omega;V^1)$ and $\theta_0 \in L^{4p+\epsilon}(\Omega;H^1)$ for some  $\epsilon \in (0,\frac{1}{2})$. Let 
the diffusion coefficient $G$ and $\GG$ satisfy the condition {\bf(C)} and {\bf ($\tilde{\rm \bf C}$)} respectively. 
 For every N, let  $\tilde{\tau}_N := \inf\{ t\geq 0 : \|\AA^{\frac{1}{2}} \theta(t)\|_{H^0} \geq N\} \wedge T$; then 
\begin{equation}		\label{sup_E_t}
\sup_{N>0} 
\sup_{t\in [0,T]} \EE\big(\| \AA^{\frac{1}{2}} \theta(t \wedge \tilde{\tau}_N
)\|_{H^0}^{2p}\big)  <\infty.
\end{equation} 
\end{lemma}
\begin{proof}
Write $\theta(t)$ using \eqref{weak_t}; then $\|\AA^{\frac{1}{2}} \theta(t)\|_{H^0} \leq \sum_{i=1}^3 T_i(t)$,  where
\begin{align*}
 T_1(t)&\, =\| \AA^{\frac{1}{2}} \SS(t) \theta_0\|_{H^0}\; , \quad T_2(t)=\Big\| \int_0^t \AA^{\frac{1}{2}} \SS(t-s) [ (u(s).\nabla) \theta(s)] ds\Big\|_{H^0}, \\
T_3(t)&\, = \Big\| \int_0^t \AA^{\frac{1}{2}} \SS(t-s) \GG(\theta(s)) d\WW(s)\Big\|_{H^0}. 
\end{align*}
The Minkowski inequality implies for $\beta \in (0,\frac{1}{2})$ 
\begin{align*}
 T_2(t)  \leq&\, \int_0^t \| \AA^{\frac{1}{2}} \SS(t-s) [(u(s) . \nabla) \theta(s)] \|_{H^0} ds\\
 & \leq  \int_0^t \|\AA^{1-\beta} \SS(t-s)\|_{{\mathcal L}(H^0;H^0)} \| \AA^{-(\frac{1}{2}-\beta)} [ (u(s).\nabla)\theta(s)\|_{H^0} ds. 
\end{align*}
Apply \eqref{GiMi-ut} with $\delta = \frac{1}{2}-\beta$, $\alpha = \frac{1}{2}$ and $\rho \in (\beta, \frac{1}{2})$. A simple computation
proves that $\|\AA^\rho f\|_{H^0} \leq \|\AA^{\frac{1}{2}} f\|_{H^0}^{2\rho} \|f\|_{H^0}^{1-2\rho}$ for any $f\in H^1$.  Therefore, 
\begin{align*} 
 \| \AA^{-(\frac{1}{2}-\beta)} [(u(s).\nabla)\theta(s)]\|_{H^0} &\leq C \|A^{\frac{1}{2}} u(s)\|_{V^0} \|\AA^\rho \theta(s)\|_{H^0}\\
 & \leq  C \|A^{\frac{1}{2}} u(s)\|_{V^0} 
\| \AA^{\frac{1}{2}} \theta(s) \|_{H^0}^{2\rho} \|\theta(s)\|_{H^0}^{1-2\rho}.
\end{align*}   
This upper estimate and \eqref{AS} imply
\[ T_2(t) \leq C \sup_{s\in [0,T]} \|A^{\frac{1}{2}} u(s)\|_{V^0}  \sup_{s\in [0,t]} \|\theta(s)\|_{H^0}^{1-2\rho} 
\int_0^t (t-s)^{-1+\beta} \|\AA^{\frac{1}{2}} 
\theta(s)\|_{H^0}^{2\rho} ds.\]
For any $p\in [1,\infty)$, H\"older's inequality  with respect to the finite measure $(t-s)^{-(1-\beta)} 1_{[0,t)}(s) ds$ implies
\begin{align*}
 T_2(t)^{2p} &\, \leq C \sup_{s\in [0,t]} \|A^{\frac{1}{2}} u(s)\|_{V^0}^{2p} \sup_{s\in [0,t]} \|\theta(s)\|_{H^0}^{2p(1-2\rho)} 
 \Big( \int_0^t (t-s)^{-(1-\beta)} ds\Big)^{2p-1}\\
 &\quad \times 
\Big( \int_0^t (t-s)^{-(1-\beta)} \|\AA^{\frac{1}{2}} \theta(s)\|_{H^0}^{4p\rho} ds \Big).
\end{align*}
Let $p_1=\frac{2(1-\rho)}{1-2\rho}$, $p_2=\frac{2(1-\rho)}{(1-2\rho)^2}$ and $p_3=\frac{1}{2\rho}$. Then $\frac{1}{p_1}+\frac{1}{p_2} + \frac{1}{p_3}=1$,
$4\rho p p_3=2p$ and $p p_1 = p(1-2\rho) p_2 := \tilde{p}$. Young's and H\"older's inequalities imply
\begin{align*}
 T_2(t)^{2p}   \leq &\, C \Big[ \frac{1}{p_1} \sup_{s\in [0,t]} \|A^{\frac{1}{2}} u(s)\|_{V^0}^{2\tilde{p}} 
 + \frac{1}{p_2} \sup_{s\in [0,t]} \|\theta(s)\|_{H^0}^{2\tilde{p}} \\
 &\quad 
+ \frac{1}{p_3} \Big( \int_0^t (t-s)^{-1+\beta} \|\AA^{\frac{1}{2}} \theta(s)\|_{H^0}^{2p} ds\Big) \Big( \int_0^t (t-s)^{-1+\beta} ds\Big)^{p_3-1} \Big] .
\end{align*}
Note that the continuous function $\rho\in (0,\frac{1}{2}) \mapsto  \frac{2(1-\rho)}{1-2\rho}$ is increasing with  $\lim_{\rho\to 0}  \frac{2(1-\rho)}{1-2\rho} =2$. 
Given $\epsilon>0$ choose 
$\rho \in (0,\frac{1}{2})$ close enough to 0 to have $2\tilde{p}=2p \frac{2(1-\rho)}{1-2\rho} = 4 p+\epsilon$, then choose $\beta \in (0,\rho)$. The above
computations yield
\begin{align} 	\label{T_2-convol}
T_2(t)^{2p} \leq C\Big[ \sup_{s\in [0,t]} \|A^{\frac{1}{2}} u(s)\|_{V^0}^{4p+\epsilon} + \sup_{s\in [0,t]} \|\theta(s)\|_{H^0}^{4p+\epsilon} \Big]
+ C \int_0^t (t-s)^{-1+\beta} \|\AA^{\frac{1}{2}} \theta(s)\|_{H^0}^{2p} ds.
\end{align}
Finally, Burhholder's inequality, the growth condition \eqref{growthGG-1}  and H\"older's inequality imply for $t\in [0,T]$ 
\begin{align}	\label{mom-convol}
\EE \Big( \Big\| \int_0^{t\wedge \tau_N} &\!\! \AA^{\frac{1}{2}} \SS(t-s) \GG(\theta(s)) d\WW(s) \Big\|_{H^0}^{2p} \Big) 
\leq \, C_p \big( {\rm Tr}(Q)\big)^p \EE\Big( \Big|
\int_0^{t\wedge \tau_N}\!\! \|\AA^{\frac{1}{2}} \GG(\theta(s))\|_{{\mathcal L}(\KK;H^0)}^2 ds \Big|^p\Big)  \nonumber \\
\leq &\, C_p \, \big( {\rm Tr}(Q)\big)^p \EE\Big( \Big| \int_0^{t\wedge \tau_N} [\tilde{K}_2 + \tilde{K}_3 \|\theta(s)\|_{H^0}^2 
 + \tilde{K}_3 \| \AA^{\frac{1}{2}} \theta(s)\|_{H^0}^2 \big] ds \Big|^p\Big)
\nonumber \\
\leq &\, C(p,\tilde{K}_2, \tilde{K}_3, {\rm Tr}(Q))  T^p \Big[ 1+ \EE\Big( \sup_{s\in [0,T]} \|\theta(s)\|_{H^0}^{2p} \Big]  \nonumber \\
&\; + C_p \big( {\rm Tr}(Q)\big)^p  \tilde{K_3}^p T^{p-1} 
\int_0^t \EE\big( \|\AA^{\frac{1}{2}} \theta(s\wedge \tau_N)\|_{H^0}^{2p}  \big) ds.
\end{align} 
The upper estimates 
\eqref{T_2-convol}, \eqref{mom-convol} and $T_1(t) \leq \|A^{\frac{1}{2}} \theta_0\|_{H^0} \leq \|\theta_0\|_{H^1}$ used with $t\wedge \tilde{\tau}_N$
instead of $t$ imply for every $t\in [0,T]$ 
\begin{align*}
 \EE\big( \|\AA^{\frac{1}{2}} \theta(t\wedge \tilde{\tau}_N)\|_{H^0}^{2p} \big) \leq&\,  C_p
\Big[1+  \EE\Big(\| \AA^{\frac{1}{2}} \theta_0\|_{H^0}^{2p} +  \sup_{s\in [0,T]} \|A^{\frac{1}{2}} u(s)\|_{V^0}^{4p+\epsilon} +
\sup_{s\in [0,T]} \|\theta(s)\|_{H^0}^{4p+\epsilon} \Big)\Big]\\
 &\, + C_p \int_0^t  \big[ (t-s)^{-1+\beta}  + \tilde{K}_3 T^{p-1} \big] 
 \EE\big( \| \AA^{\frac{1}{2}} \theta(s\wedge \tilde{\tau}_N)\|_{H^0}^{2p} \big) ds, 
\end{align*}
where the constant $C_p$ does not depend on $t$ and $N$. 
Theorem \ref{th-gwp}, Proposition \ref{prop_u_V1}, and the version of Gronwall's lemma proved in the following lemma \ref{Gronwall} imply
\eqref{sup_E_t} 
for some constant $C$ depending on $\EE(\|u_0\|_{V^1}^{4p+\epsilon})$ and $\EE(\|\theta_0\|_{H^0}^{4p+\epsilon})$.\\
 The proof of the Lemma 
is complete.
\end{proof}

The following lemma is an extension of Lemma 3.3, p. 316  in \cite{Walsh}.
 For the sake of completeness its prove is given at the end of this section.
\begin{lemma}		\label{Gronwall}
Let $\epsilon \in (0,1)$, $a,b,c$ be positive constants and $\varphi$ be a bounded non negative function such that 
\begin{equation}  	\label{maj_Gron}
\varphi(t) \leq a+\int_0^t  \big[ b+c(t-s)^{-1+\epsilon}\big] \,\varphi(s)\, ds, \quad \forall t\in [0,T].
\end{equation}
Then $\sup_{t\in [0,T]} \varphi(t)\leq  C $ 
 for some constant $C$  depending on $a,b,c,T$ and $\epsilon$.
\end{lemma}

\bigskip

\noindent{\it Proof of Proposition \ref{lem_mom_t_H1}}
We next prove that 
the gradient of the temperature
 has bounded moments uniformly in time. 

We only prove \eqref{mom_t_H1} for $p\in [2,+\infty)$; the other argument is similar and easier. 

Applying the operator $\AA^{\frac{1}{2}}$ to equation \eqref{def_t}, and writing It\^o's formula for the square of corresponding $H^0$-norm.
we obtain 
\begin{align*}		
\| \AA^{\frac{1}{2}} \theta(t) &\|_{H^0}^2 +2\kappa \int_0^t \|\AA \theta(s)\|_{H^0}^2 ds = \| \AA^{\frac{1}{2}} \theta_0\|_{H^0}^2 
- 2 \int_0^t \langle (u(s).\nabla) \theta(s), \AA\theta(s) \rangle ds  \\
& + 2 \int_0^t \big( \AA^{\frac{1}{2}} \GG(\theta(s)) d\WW(s) , \AA^{\frac{1}{2}} \theta(s)\big)
+ {\rm Tr}(Q) \int_0^t \| \AA^{\frac{1}{2}} \GG(\theta(s))\|_{H^0}^2 ds.
\end{align*}
Then apply It\^o's formula for the map $x\mapsto x^p$. This yields, using integration by parts, 
\begin{align} \label{Ito-t-H1}
\| \AA^{\frac{1}{2}}& \theta(t) \|_{H^0}^{2p} + 2p\kappa \int_0^t \|\AA \theta(s)\|_{H^0}^2 \|\AA^{\frac{1}{2}} \theta(s)\|_{H^0}^{2(p-1)} ds
= \| \AA^{\frac{1}{2}} \theta_0\|_{H^0}^{2p}  \nonumber \\
&- 2p \int_0^t \langle (u(s).\nabla) \theta(s), \AA\theta(s) \rangle \| \AA^{\frac{1}{2}} \theta(s)\|_{H^0}^{2(p-1)} ds  \nonumber \\
& + 2p \int_0^t \big( \AA^{\frac{1}{2}} \GG(\theta(s)) d\WW(s) , \AA^{\frac{1}{2}} \theta(s)\big)  \| \AA^{\frac{1}{2}} \theta(s)\|_{H^0}^{2(p-1)} \nonumber \\
&+ p {\rm Tr}(\QQ) \int_0^t \| \AA^{\frac{1}{2}} \GG(\theta(s))\|_{H^0}^2  \| \AA^{\frac{1}{2}} \theta(s)\|_{H^0}^{2(p-1)} ds\nonumber \\
&+ 2p(p-1) {\rm Tr}(\QQ)\int_0^{t} \| \big(\AA^{\frac{1}{2}} \GG\big)^*(\theta(s)) \big(\AA^{\frac{1}{2}} \theta(s)\big)\|_K^2 
\| \AA^{\frac{1}{2}} \theta(s)\|_{H^0}^{2(p-2)} ds.
\end{align}
The Gagliardo-Nirenberg inequality \eqref{GagNir} 
and the inclusion $V^1\subset \LL^4$ implies 
\begin{align*}
  \int_0^t \big|  \langle (u(s).\nabla) &\theta(s) , \tilde{A} \theta(s)\rangle \big| \,  \|\AA^{\frac{1}{2}} \theta(s)\|_{H^0}^{2(p-1)}  ds  \\
  & \leq  C  \int_0^t \|\tilde{A} \theta(s)\|_{H^0}  \|u(s)\|_{\LL^4} \|\AA^{\frac{1}{2}} \theta(s)\|_{L^4}  \|\AA^{\frac{1}{2}} \theta(s)\|_{H^0}^{2(p-1)}  ds  \\
 &\leq  C  \int_0^t \|\tilde{A} \theta(s)\|_{H^0}^{\frac{3}{2}}   \|u(s)\|_{V^1}  \|\AA^{\frac{1}{2}} \theta(s)\|_{H^0}^{2p-\frac{3}{2}}  ds .
\end{align*}
Then using the H\"older and Young inequalities, we deduce 
\begin{align}		\label{bilin-ut}
2p \int_0^{t} 
  &\big| \langle (u(s).\nabla) \theta(s), \AA\theta(s) \rangle \big| 
\| \AA^{\frac{1}{2}} \theta(s)\|_{H^0}^{2(p-1)} ds  	\nonumber \\
\leq & \; (2p-1)\, \kappa  \int_0^{t} 
 \|\AA(\theta(s))\|_{H^0}^{2} \| \AA^{\frac{1}{2}} \theta(s)\|_{H^0}^{2(p-1)} ds  \nonumber \\
 &\; +   C(\kappa,p) \sup_{s\in [0,T]}  \|u(s)\|_{V^1}^4 
 \int_0^t 
 \| \AA^{\frac{1}{2}} \theta(s)\|_{H^0}^{2p} ds .
\end{align}
The growth condition \eqref{growthGG-1}, H\"older's and Young inequalities imply
\begin{equation}		\label{maj_trace1}
 \int_0^{t} 
  \|\AA^{\frac{1}{2}} \GG(\theta(s)) \|_{H^0}^2 \|\AA^{\frac{1}{2}} \theta(s)\|_{H^0}^{2(p-1)} ds 
\leq C \int_0^{t} 
\big[ 1+ \|\theta(s)\|_{H^0}^{2p} +\|\AA^{\frac{1}{2}} \theta(s)\|_{H^0}^{2p} \big] ds,
\end{equation}
and a similar computation yields 
\begin{align}		\label{maj_trace2}
 \int_0^{t} 
  \big\| \big(\AA^{\frac{1}{2}} &\, \GG\big)^*(\theta(s)) \big(\AA^{\frac{1}{2}} \theta(s)\big)\|_K^2 
\| \AA^{\frac{1}{2}} \theta(s)\|_{H^0}^{2(p-2)} ds \nonumber \\
& \leq 
C \int_0^{t} 
\big[ 1+ \|\theta(s)\|_{H^0}^{2p} +\|\AA^{\frac{1}{2}} \theta(s)\|_{H^0}^{2p} \big] ds.
\end{align}
 Let $\tilde{\tau}_N:=\inf\{ t\geq 0: \|\AA^{\frac{1}{2}} \theta(t)\|_{H^0} \geq N\}$. The upper estimates \eqref{Ito-t-H1}-- \eqref{maj_trace2}
  written for $t\wedge \tilde{\tau}_N$ instead of $t$ imply 
\begin{align*}
\sup_{t\in [0,T]} &\|\AA^{\frac{1}{2}} \theta(t\wedge \tilde{\tau}_N) 
 \|_{H^0}^{2p}  + \kappa \int_0^{T\wedge \tilde{\tau}_N} \| \AA \theta(s)\|_{H^0}^2 \| \AA^{\frac{1}{2}} \theta(s)\|_{H^0}^{2(p-1)} ds 
  \leq \|\AA^{\frac{1}{2}} \theta_0\|_{H^0}^{2p}\\
&  
+ C \sup_{s\in [0,T]} \! \|u(s)\|_{V^1}^4 \!\! \int_0^{T\wedge \tilde{\tau_N}} \! \| \AA^{\frac{1}{2}} \theta(s)\|_{H_0}^{2p} ds  
+C \int_0^{T\wedge \tilde{\tau}_N} \!\! \big( 1+\|\theta(s)\|_{H^0}^{2p} + \| \AA^{\frac{1}{2}} \theta(s)\|_{H^0}^{2p}  \big)ds \\
&+ 2p \,
\sup_{t\in [0,T]} \int_0^{t\wedge \tilde{\tau}_N} 
 \big(
 \AA^{\frac{1}{2}} \GG(\theta(s)) d\WW(s), 
\AA^{\frac{1}{2}} \theta(s) \big) 
\| \AA^{\frac{1}{2}} \theta(s)\|_{H^0}^{2(p-1)}.
\end{align*}
Using the Cauchy-Schwarz inequality, Fubini's theorem, \eqref{mom_u_V1}  and \eqref{sup_E_t}, we deduce
\begin{align}		\label{maj_u_int-theta}
\EE\Big( \sup_{s\in [0,T]}& \| u(s)\|_{V^1}^4  \int_0^{T\wedge\tilde{\tau}_N}  \| \AA^{\frac{1}{2}} \theta(s)\|_{H^0}^{2p} ds \Big) \nonumber \\
&  \leq 
\Big\{ \EE\Big( \sup_{s\in [0,T]} \| u(s)\|_{V^1}^8  \Big) \Big\}^{\frac{1}{2}} 
\Big\{ \int_0^T \EE\big( \|\AA^{\frac{1}{2}} \theta(s\wedge \tilde{\tau}_N)\|_{H^0}^{4p} \big) ds \Big\}^{\frac{1}{2}} \leq C.
\end{align} 
The Davis inequality, the growth condition \eqref{growthGG-1},  the Cauchy-Schwarz, Young and H\"older inequalities imply
\begin{align*}
\EE \Big(& \sup_{t\in [0,T]} 
\Big|  \int_0^{t\wedge \tilde{\tau}_N} \big( 
\AA^{\frac{1}{2}} \GG(\theta(s)) d\WW(s), 
\AA^{\frac{1}{2}} \theta(s) \big)  \| \AA^{\frac{1}{2}} \theta(s)\|_{H^0}^{2(p-1)} \Big) 
\\
& \leq C \, \EE\Big( \Big\{ \int_0^{T} 
 {\rm Tr}(\QQ)
 \big[\tilde{K}_2  + \tilde{K}_3 \|\theta(s\wedge \tilde{\tau}_N)\|_{H^1}^2   \big] 
\|\AA^{\frac{1}{2}} \theta(s\wedge \tilde{\tau}_N) \|_{H^0}^{4p-2} ds \Big\}^{\frac{1}{2}} \Big)  \\
&\leq  C \, \EE\Big[ \big( \sup_{s\leq T} \big( 
 \|\AA^{\frac{1}{2}} \theta(s\wedge \tilde{\tau}_N)
\|_{H^0}^p\big)   ({\rm Tr}(\QQ))^{\frac{1}{2}}  \\
&\quad \times 
 \Big\{\int_0^T \big[ \tilde{K}_2  + \tilde{K}_3 \|\theta(s\wedge \tilde{\tau}_N)\|_{H^0}^2  + \tilde{K}_3
  \|\AA^{\frac{1}{2}} \theta(s\wedge \tilde{\tau}_N)\|_{H^0}^2 \big] 
\| \AA^{\frac{1}{2}} \theta(s\wedge \tilde{\tau}_N) 
\|_{H^0}^{2(p-1)} ds \Big\}^{\frac{1}{2}} \Big)\\
&\leq \frac{1}{4p} \EE\Big( \sup_{s\leq T} \big( 
\|\AA^{\frac{1}{2}}   \theta(s\wedge\tilde{\tau}_N) 
\|_{H^0}^{2p}\Big) 
+ C \EE\Big( \int_0^T \big[1+\| \theta(s\wedge\tilde{\tau}_N)\|_{H^0}^{2p} 
+ \| \AA^{\frac{1}{2}} \theta(s\wedge \tilde{\tau}_N)\|_{H^0}^{2p}  \big] ds\Big). 
\end{align*}
Therefore, the upper estimates \eqref{mom_t_L2}, \eqref{sup_E_t} and \eqref{maj_u_int-theta} imply
\[ 
\frac{1}{2}  \EE\Big( \sup_{s\leq T}  \|\AA^{\frac{1}{2}} \theta(s\wedge\tilde{\tau}_N)\|_{H^0}^{2p} \Big) 
+  \kappa\; \EE\Big(  \int_0^{T\wedge \tilde{\tau}_N} \| \AA \theta(s\wedge \tilde{\tau}_N)\|_{H^0}^2 \| \AA^{\frac{1}{2}} \theta(s)\|_{H^0}^{2(p-1)} ds \Big)
 \leq C
\] 
for some constant $C$ independent of $N$. 

 As $N\to +\infty$, we deduce \eqref{mom_t_H1}; 
this completes the proof of Proposition \ref{prop_ul}. 
\hfill $\Box$

We conclude this section with an extension of the Gronwall Lemma. 

 \noindent{\it Proof of Lemma \ref{Gronwall}  }
For $t\in [0,T]$, iterating \eqref{maj_Gron} and using the Fubini theorem, we obtain
\begin{align*}
\varphi(t) \leq &\, a+\int_0^t \big[ b +c (t-s)^{-1+\epsilon}] \Big[ a+\int_0^s \big( b+c(s-r)^{-1+\epsilon}\big)  \varphi(r) dr \Big] ds \\
\leq & \, a\Big( 1+\! \! \int_0^t [b+c(t-s)^{-1+\epsilon}] ds \Big) + \int_0^t \!\Big( \! \int_r^t \! \big[ b+c(t-s)^{-1+\epsilon}] [b+c(s-r)^{-1+\epsilon}] ds\Big) 
\varphi(r) dr\\
\leq &\, A_1
+ \int_0^t 
\Big[ b^2(t-r) + \frac{2bc}{\epsilon}  (t-r)^\epsilon + c^2\int_r^t (t-s)^{-1+\epsilon}(s-r)^{-1+\epsilon}ds \Big] \varphi(r) dr\\
\leq &  A_1
+ \int_0^t \Big[ B_1
+ C_1
 (t-r)^{-1+2\epsilon} \! 
\int_0^1\! \!\lambda^{-1+\epsilon}(1-\lambda)^{-1+\epsilon} d\lambda \Big] \varphi(r) dr,
\end{align*} 
for positive constants $A_1$ (depending on  $a,b,c,T,\epsilon$), $B_1$  (depending on $b,c,T,\epsilon$), and $C_1$
(depending on $c$ and $\epsilon$).  One easily proves by induction on $k$ that for every integer $k\geq 1$
\begin{align*} 
 \varphi(t) &\, \leq A_k
 + \! \int_0^t \Big[ B_k
 +  c\,  C_{k-1}
 \int_r^t (t-s)^{-1+k\epsilon} (s-r)^{-1+\epsilon} ds \Big]
  \varphi(r) dr\\
 &\, \leq A_k
 + \int_0^t \big[ B_k
 + C_k
 (t-r)^{-1+(k+1)\epsilon}\big] \varphi(r) dr,
\end{align*} 
for some positive constants $A_k $, $B_k$ 
and $C_k$ depending on  $a,b,c,T$ and $\epsilon$.
 Indeed, a change of variables implies 
\begin{align*}
 \int_r^t (t-s)^{-1+k\epsilon} (s-r)^{-1+\epsilon} ds&\, 
 = (t-r)^{-1+(k+1)\epsilon} \int_0^1 \lambda^{-1+k\epsilon} (1-\lambda)^{-1+\epsilon} d\lambda \\
&\, = \tilde{C}_k
(t-r)^{-1+(k+1)\epsilon}
\end{align*}
for some constant $\tilde{C}_k$ depending on $k$ and $\epsilon$. 

Let $k^*$ be the largest integer such that $k\epsilon <1$, that is $k^*\epsilon<1\leq (k^*+1)\epsilon$. 
Then since $(t-r)^{-1+(k^*+1)\epsilon}\leq T^{-1+(k^*+1)\epsilon}$, we deduce 
\[ \varphi(t) \leq A
+ \int_0^T\! B\, 
\varphi(r) dr,\]
for some positive constants $A$ and $B $ depending on the parameters $a,b,c,T$ and $\epsilon$. 
The classical Gronwall lemma concludes the proof of the Lemma. 
\hfill $\Box$


\section{ Moment estimates of time increments of the solution} \label{s-increments}
In this section we prove moment estimates for various norms of time increments of the solution to \eqref{def_u}--\eqref{def_t}. 
This will be crucial to deduce the speed of convergence of numerical schemes. 
We first prove the time regularity of the velocity and temperature in $L^2$.
\begin{prop}	\label{prop_regularity_L2} 
Let $u_0, \theta_0$ be ${\mathcal F}_0$-measurable, and suppose that $G$ and $\GG$ satisfy {\bf (C-u)} and  {\bf (C-$\theta$)}
respectively.  

(i) Let  $u_0\in L^{4p}(\Omega;V^1)$ and $\theta_0\in L^{2p}(\Omega;H^0)$. 
 For $0\leq \tau_1<\tau_2 \leq T$,
\begin{equation}		\label{increm_u_L2}
\EE\big( \|u(\tau_2) -u(\tau_1) \|_{V^0}^{2p}\big) \leq C \,   |\tau_2-\tau_1|^{p}.
\end{equation}

(ii) Let $u_0\in L^{8p+\epsilon}(\Omega;V^1)$ and $\theta_0\in L^{8p+\epsilon}(\Omega;H^1)$ for some  $\epsilon >0$. 
Then for  $0\leq \tau_1<\tau_2 \leq T$,
\begin{equation} 	\label{increm_t_L2}
\EE\big(  \|\theta(\tau_2) -\theta(\tau_1) \|_{H^0}^{2p} \big) \leq C \, |\tau_2-\tau_1|^{ p}.
\end{equation}

\end{prop}
\begin{proof}
Recall that  $S(t)=e^{-\nu t A}$ is the analytic semi group  generated by the Stokes operator $A$ multiplied by the viscosity $\nu$
and $\SS(t)=e^{-\kappa t \AA} $ is the semi group generated by $\AA=-\Delta$. We use the mild formulation of the solutions stated in
\eqref{weak_u} and \eqref{weak_t}. 

(i) Let $0\leq \tau_1<\tau_2\leq T$; then $u(\tau_2)-u(\tau_1)=\sum_{i=1}^4 T_i(\tau_1,\tau_2)$, where
\begin{align}		\label{decomp_increm_u}
T_1(\tau_1,\tau_2)&\, = S(\tau_2) u_0 - S(\tau_1) u_0 = \big[ S(\tau_2)-S(\tau_1)\big] S(\tau_1) u_0, \nonumber  \\
T_2(\tau_1,\tau_2)&\, = \int_0^{\tau_2} S(\tau_2-s) B(u(s),u(s)) ds - \int_0^{\tau_1} S(\tau_1-s) B(u(s),u(s)) ds, 		\nonumber \\
T_3(\tau_1,\tau_2)&\, = \int_0^{\tau_2} S(\tau_2-s) \Pi \theta(s) v_2 ds - \int_0^{\tau_1} S(\tau_1-s) \Pi \theta(s) v_2 ds  	\nonumber \\
T_4(\tau_1,\tau_2)&\, = \int_0^{\tau_2} S(\tau_2-s) G(u(s)) dW(s) - \int_0^{\tau_1} S(\tau_1-s) G(u(s)) dW(s).
\end{align}
The arguments used in the proof of Lemma 2.1 \cite{BeMi_additive}, using \eqref{GiMi-uv},  \eqref{AS}, \eqref{A(I-S)} and \eqref{mom_u_V1}  prove
\begin{equation}	\label{T1-T2}
\EE\big( \| T_1(\tau_1,\tau_2)\|_{V^0}^{2p} + \| T_2(\tau_1,\tau_2)\|_{V^0}^{2p}\big)  \leq C \big[1+ \EE(\|u_0\|_{V^1}^{4p})] |\tau_2-\tau_1|^p.
\end{equation}
Let $ T_3(\tau_1,\tau_2) = T_{3,1}(\tau_1,\tau_2)+T_{3,2}(\tau_1,\tau_2)$, where
\begin{align*} T_{3,1}(\tau_1,\tau_2)=&\, \int_0^{\tau_1}  [S(\tau_2-\tau_1)-{\rm Id}] S(\tau_1-s) \big[ \Pi \theta(s) v_2\big] ds, \\
T_{3,2}(\tau_1,\tau_2)=&\, \int_{\tau_1}^{\tau_2} \! S(\tau_2-s) \big[ \Pi \theta(s) v_2\big] ds.
\end{align*}
Since the family of sets $\{ A(t,M)\}_t$ is decreasing, the Minkowski inequality,  \eqref{AS} and \eqref{A(I-S)}  imply
\begin{align*}
 \|  T_{3,1}(\tau_1,\tau_2)\|_{V^0} \leq& \int_{0}^{\tau_1} \!\! \|A^{\frac{1}{2}} S(\tau_1-s)\|_{{\mathcal L}(V^0;V^0)} \, 
 \| A^{- \frac{1}{2}} [ S(\tau_2-\tau_1)-{\rm Id}] \|_{{\mathcal L}(V^0;V^0)}  \|\Pi \theta(s) v_2
\|_{V^0}  ds  \\
\leq &\; C \big| \tau_2-\tau_1\big|^{\frac{1}{2}}  \;  \sup_{s\in [0,T]}    \| \theta(s)\|_{H^0} , 
\end{align*}  
and
\[ \|T_{3,2}(\tau_1,\tau_2)\|_{V^0} \leq \int_{\tau_1}^{\tau_2} \| S(\tau - s) \big[ \Pi \theta(s) v_2\big] \|_{V^0}ds \leq 
\big| \tau_2-\tau_1\big|  \sup_{s\in [0,T] } \| \theta(s)\|_{H^0}.\]
The inequality \eqref{mom_t_L2}  implies
\begin{equation}		\label{mom_u_T3}
\EE\Big(\|T_3(\tau_1,\tau_2)\|_{V^0}^{2p} \Big) \leq C\,  \big| \tau_2-\tau_1|^{p} \, \EE(\|\theta_0\|_{H^0}^{2p}).
\end{equation} 
Finally, decompose the stochastic integral as follows:
\[  T_{4,1}(\tau_1,\tau_2)=\int_0^{\tau_1}\!\!  [S(\tau_2-\tau_1)-{\rm Id}] S(\tau_1-s) G(s) dW(s), \, 
T_{4,2}(\tau_1,\tau_2)\!=\!\int_{\tau_1}^{\tau_2} \! S(\tau_2-s) G(s) dW(s).\]
The Burkholder inequality, \eqref{A(I-S)}, H\"older's inequality and the growth condition \eqref{growthG-1} imply 
\begin{align}		\label{mom_u_T41}
\EE\Big( \|T_{4,1}\|_{V^0}^{2p}\Big)  \leq &\, C_p \EE \Big( \Big| \int_0^{\tau_1} \| [S(\tau_2-\tau_1) - {\rm Id}] S(\tau_1-s) G(u(s))\|_{V^0}^2
{\rm Tr}(Q) ds\Big|^p\Big)  \nonumber \\
\leq & \, C ({\rm Tr}(Q))^p \EE\Big( \Big| \int_0^{\tau_1}  \|A^{-\frac{1}{2}} [S(\tau_2-\tau_1) - {\rm Id}] \|_{{\mathcal L}(V^0;V^0)}^2
\| A^{\frac{1}{2}} G(u(s))\|_{V^0}^2 ds \Big|^p \Big) \nonumber \\
\leq & \, C \EE\Big( \Big| \int_0^{\tau_1} \big| \tau_2-\tau_1\big| \;  \big[ K_2+K_3\|u(s)\|_{V^1}^2\big] ds \Big|^p\Big)\nonumber  \\
\leq & \, C \big[ 1+\EE(\|u_0\|_{V^1}^{2p}) \big] |\tau_2-\tau_1|^p,
\end{align}
where the last upper estimate is a consequence of \eqref{mom_u_L2} and \eqref{mom_u_V1}. 
A similar easier argument implies
\begin{align}	\label{mom_u_T42}
 \EE\Big( \|T_{4,2}\|_{V^0}^{2p}\Big)  \leq &\, C_p \EE\Big( \Big| \int_{\tau_1}^{\tau_2} \| S(\tau_2-s) G(u(s))\|_{V^0}^2 {\rm Tr}(Q) ds \Big|^p \Big)
 \nonumber \\
 \leq & \, C \big[ 1+ \EE(\|u_0\|_{V^0}^{2p}) \big] \,  \big|\tau_2-\tau_1\big|^p.
\end{align}
The inequalities \eqref{T1-T2}--\eqref{mom_u_T42} complete the proof of \eqref{increm_u_L2}. 

(ii) As in the proof of (i), for $0\leq \tau_1<\tau_2\leq T$, let $\theta(\tau_2)-\theta(\tau_1)=\sum_{i=1}^3 \tilde{T}_i(\tau_1,\tau_2)$, where
\begin{align}		\label{decom_increm_t} 
\tilde{T}_1(\tau_1,\tau_2)=&\, \big[ \SS(\tau_2-\tau_1)-{\rm Id}\big] \SS(\tau_1) \theta_0, 	\nonumber \\
\tilde{T}_2(\tau_1,\tau_2)=&\, -\int_0^{\tau_2} \SS(\tau_2-s) \big( [u(s).\nabla] \theta(s)\big) ds +
\int_0^{\tau_1} \SS(\tau_1-s) \big( [u(s).\nabla] \theta(s)\big) ds 		\nonumber \\
\tilde{T}_3(\tau_1,\tau_2)=&\, \int_0^{\tau_2} \SS(\tau_2-s)\GG(\theta(s)) d\WW(s) -  
\int_0^{\tau_1} \SS(\tau_1-s) \GG(\theta(s)) d\WW(s).  
\end{align}
The inequality \eqref{A(I-S)} implies 
\begin{align}		\label{t_TT1}
\| \tilde{T}_1(\tau_1,\tau_2)\|_{H^0} = &\, \| \AA^{-\frac{1}{2}}  \big[ \SS(\tau_2-\tau_1)-{\rm Id}\big] \SS(\tau_1)  \AA^{\frac{1}{2}} \theta_0\|_{H^0} 
\nonumber \\
\leq & \, C \big| \tau_2-\tau_1|^{\frac{1}{2}} \|\theta_0\|_{H^1}.
\end{align}
Decompose $\tilde{T}_2(\tau_1,\tau_2)=\tilde{T}_{2,1}(\tau_1,\tau_2) + \tilde{T}_{2,2}(\tau_1,\tau_2)$, where
\begin{align*}
 \tilde{T}_{2,1}(\tau_1,\tau_2) = & \, -\int_0^{\tau_1} \big[ \SS(\tau_2-\tau_1)-{\rm Id}]\,  \SS(\tau_1-s)\,  \big( [u(s).\nabla]\theta(s)\big] ds,\\
  \tilde{T}_{2,2}(\tau_1,\tau_2) =&\,  - \int_{\tau_1}^{\tau_2} \SS(\tau_2-s) \,\big( [u(s).\nabla] \theta(s)\big) ds.
\end{align*} 
Let  $\delta \in (0,\frac{1}{2})$; the Minkowski inequality, \eqref{AS}, \eqref{A(I-S)},
 and \eqref{GiMi-ut} applied with $\alpha = \rho = \frac{1}{2}$ imply
\begin{align*}
\|\tilde{T}_{2,1}&(\tau_1,\tau_2) \|_{H^0} \leq  \int_0^{\tau_1} \| \SS(\tau_1-s)\, \big[ \SS(\tau_2-\tau_1)-{\rm Id} \big] \big( [u(s).\nabla]\theta(s)\big)
\|_{H^0} ds \\
\leq & \int_0^{\tau_1} \|\AA^{\frac{1}{2}+ \delta } \SS(\tau_1-s)\|_{{\mathcal L}(H^0;H^0)} \|\AA^{-\frac{1}{2}}[ \SS(\tau_2-\tau_1)-{\rm Id}]\|_{{\mathcal L}(H^0;H^0)}
\|\AA^{-\delta} \big( [u(s).\nabla] \theta(s)\|_{H^0} ds\\
\leq & \, C \int_0^{\tau_1} (\tau_1-s)^{-(\frac{1}{2} + \delta )} \, |\tau_2-\tau_1|^{\frac{1}{2}} 
\, \|A^{\frac{1}{2}} u(s)\|_{V^0} \, \| \AA^{\frac{1}{2}} \theta(s)\|_{H^0} ds\\
\leq & \, C |\tau_2-\tau_1|^{\frac{1}{2}} \, \sup_{s\in [0,T]} \|A^{\frac{1}{2}}u(s) \|_{V^0} \,
\int_0^{\tau_1}  (\tau_1-s)^{-(\frac{1}{2} + \delta )} \,   
\|\AA^{\frac{1}{2}}\theta(s)\|_{H^0}\, ds . 
\end{align*}
Let $p_1\in \big( 2 , 2+\frac{\epsilon}{4p} \big)$, let $\delta \in \big( 0, \frac{1}{2} -\frac{1}{p_1}\big)$. Let $p_2$ be the  conjugate exponent
of $p_1$; we have   $(\frac{1}{2}+\delta)p_2 <1$. Thus, H\"older's inequality for the finite measure 
$(\tau_1-s)^{-(\frac{1}{2}+\delta)} 1_{[0, \tau_1)}(s) ds$
with exponents $2p$ and $\frac{2p}{2p-1}$ and then for $ds$ with conjugate exponents $p_1$ and $p_2$ imply 
\begin{align*}
\|\tilde{T}_{2,1}(\tau_1,\tau_2) \|_{H^0}^{2p}  \leq  &\, 
C \big| \tau_2-\tau_1\big|^{p}  \; \sup_{s\in [0,T]} \| A^{\frac{1}{2}} u(s)\|_{V^0}^{2p} \;  \int_0^{\tau_1} (\tau_1-s)^{-(\frac{1}{2}+\delta)} \| \AA^{\frac{1}{2}}   
\theta(s)\|_{H^0}^{2p} ds \\
& \qquad \times  \Big\{\int_0^{\tau_1} (\tau_1-s)^{-(\frac{1}{2}+\delta)}  ds \Big\}^{2p-1} \\
\leq & \, C\, \big| \tau_2-\tau_1\big|^{p}  \; \sup_{s\in [0,T]} \| A^{\frac{1}{2}} u(s)\|_{V^0}^{2p} \; 
\Big\{ \int_0^{\tau_1} \| \AA^{\frac{1}{2}} \theta(s)\|_{H^0}^{2pp_1} ds \Big\}^{\frac{1}{p_1}}  \\
&\qquad \times \Big\{ \int_0^{\tau_1} (\tau_1-s)^{- (\frac{1}{2} + \delta) p_2} ds \Big\}^{\frac{1}{p_2}}. 
\end{align*}
Since $2pp_1 < 4p+\frac{\epsilon}{2}$ and $2pp_2<4p$,  H\"older's inequality,  Fubini's theorem 
 together with  the upper estimates \eqref{mom_u_V1} 
and   \eqref{sup_E_t}  imply 
\begin{align}		\label{mom_t_TT21}
\EE \big(  &\|\tilde{T}_{2,1}(\tau_1,\tau_2) \|_{H^0}^{2p}\big)  \leq C \big| \tau_2-\tau_1|^{ p} 
\Big\{ \EE\Big(\sup_{s\in [0,T]} \|A^{\frac{1}{2}}u(s) \|_{V^0}^{2p p_2} \Big) \Big\}^{\frac{1}{p_2}} \nonumber \\
& \times   \Big\{  \int_0^{\tau_1} \EE\ \big( \|\AA^{\frac{1}{2}} 
\theta(s)\|_{H^0}^{2pp_1} \big) ds  \Big\}^{\frac{1}{p_1}}  \leq C\;  \big| \tau_2-\tau_1|^{ p} .
\end{align}
A similar argument proves for $\eta \in (0,1)$
\begin{align*}
\|\tilde{T}_{2,2}(\tau_1,\tau_2) \|_{H^0} \leq &\,  \int_{\tau_1}^{\tau_2} \| \AA^{1-\eta} \SS(\tau_2-s) \|_{\mathcal{L}(H^0;H^0)}
\|\AA^{-(1-\eta)} \big( [u(s).\nabla]\theta(s)\big) \|_{H^0} ds \\
\leq &\, C  \int_{\tau_1}^{\tau_2} (\tau_2-s)^{-1+\eta} \, \|A^{\frac{1}{2}} u(s)\|_{V^0} \, \| \AA^{\frac{1}{2}} \theta(s)\|_{H^0} ds \\
\leq & \, C\, \big| \tau_2-\tau_1|^{\eta} \, \sup_{s\in [0,T]} \|A^{\frac{1}{2}} u(s)\|_{V^0} \; \int_{\tau_1}^{\tau_2} (\tau_2-s)^{-1+\eta} \, 
 \|\AA^{\frac{1}{2}} \theta(s)\|_{H^0}
ds . 
\end{align*}
Let $\eta \in \big( \frac{p}{2p-1},1\big)$; for $\epsilon>0$ let  $p_1,p_2\in (1,+\infty)$ be conjugate exponents such that
 $\big(\frac{1}{\eta}\big) \vee \big( \frac{8p+\epsilon}{4p+\epsilon}\big)<p_1<2$; then $(1-\eta)p_2<1$.
H\"older's inequality implies
\begin{align*}
 \|\tilde{T}_{2,2}(\tau_1,\tau_2) \|_{H^0}^{2p}  \leq &\,C\,  \big| \tau_2-\tau_1|^{(2p-1)\eta} \, \sup_{s\in [0,T]} \|A^{\frac{1}{2}} u(s)\|_{V^0}^{2p}  \; 
 \int_{\tau_1}^{\tau_2} (\tau_2-s)^{-1+\eta} \,   \|\AA^{\frac{1}{2}} \theta(s)\|_{H^0}^{2p} ds \\
 \leq & \, C\,  \big| \tau_2-\tau_1|^{(2p-1)\eta} \, \sup_{s\in [0,T]} \|A^{\frac{1}{2}} u(s)\|_{V^0}^{2p}  \;  \Big\{ \int_{\tau_1}^{\tau_2} 
 (\tau_2-s)^{-(1-\eta) p_2}  ds \Big\}^{\frac{1}{p_2}} \\
 &\qquad \times \Big\{ \int_{\tau_1}^{\tau_2} \|\AA^{\frac{1}{2}} \theta(s)\|_{H_0}^{2pp_1} ds \Big\}^{\frac{1}{p_1}}. 
\end{align*} 
Since  $(2p-1) \eta >p$, $\frac{1}{\eta}<2$; furthermore, $2p p_2<4p+\frac{\epsilon}{2}$ and $2pp_1\leq 4p$. 
 H\"older's inequality together with  the upper estimates \eqref{mom_u_V1} and \eqref{mom_t_H1} imply 
\begin{align}		\label{mom_t_TT22}
\EE \big(  \|\tilde{T}_{2,2}&(\tau_1,\tau_2) \|_{H^0}^{2p}\big)  \leq C  \, \big| \tau_2-\tau_1|^{ p} \, 
\Big\{ \EE\Big( \sup_{s\in [0,T]} \|A^{\frac{1}{2}} u(s)\|_{V^0}^{2pp_2} \Big) \Big\}^{\frac{1}{p_2}} 		\nonumber  \\
 & \quad \times \Big\{ \int_{\tau_1}^{\tau_2}\EE\big(  \| \AA^{\frac{1}{2}} \theta(s)\|_{H^0}^{2pp_1} \big) ds \Big\}^{\frac{1}{p_1}}
 \leq C\,  \big| \tau_2-\tau_1|^{ p} .
\end{align} 
This inequality and \eqref{mom_t_TT21} imply 
\begin{equation} 	\label{mom_t_TT2}
\EE\big( \|\tilde{T}_{2}(\tau_1,\tau_2) \|_{H^0}^{2p}\big) \leq C\,  \big| \tau_2-\tau_1|^{p}.
\end{equation}
Finally, an argument similar to that used to prove \eqref{mom_u_T41} and \eqref{mom_u_T42}, using  the growth condition
\eqref{growthGG-1} and \eqref{mom_t_L2},  implies, 
\begin{equation}	\label{mom_t_TT3}
\EE\big(  \|\tilde{T}_{3}(\tau_1,\tau_2) \|_{H^0}^{2p}\big) \leq C \, \big| \tau_2-\tau_1\big|^p.
\end{equation} 
The upper estimates \eqref{t_TT1}, \eqref{mom_t_TT2} and \eqref{mom_t_TT3}  complete the proof of 
\eqref{increm_t_L2}. 
\end{proof}

We next prove some time regularity for the gradient of the velocity and the temperature. 

\begin{prop}			\label{prop_regularity_H1}
 Let $N\geq 1$ be an integer  and for $k=0, \cdots, N$ set $t_k=\frac{kT}{N}$, $G$ and $\GG$ satisfy conditions
 {\bf (C-u)} and {\bf (C-$\theta$)} respectively, and let $\eta \in (0,1)$. 

(i) Let $p\in [2,\infty)$,  $u_0\in L^{4p}(\Omega;V^1)$ and $\theta_0\in L^{2p}(\Omega;H^0)$. 
Then
there exists  a positive constant $C$ (independent of $N$) such that 
\begin{align}			\label{mom_increm_u_1}
\EE\Big(  \Big|  \sum_{j=1}^N \int_{t_{j-1}}^{t_j}\!\! & \big[ 
\| u(s)-u(t_j)\|_{V^1}^2 + \|u(s)-u(t_{j-1})\|_{V^1}^2 \big] ds \Big|^p \Big)
\leq C \Big( \frac{T}{N}\Big)^{\eta p} .
\end{align} 

(ii) Let $p\in [2,\infty)$, $u_0\in L^{16p+\epsilon}(\Omega;V^1)$ and $\theta_0\in L^{16p+\epsilon}(\Omega;H^0)$ for some $\epsilon >0$. Then
 \begin{align}			\label{mom_increm_t_1}
\EE\Big(  \Big|  \sum_{j=1}^N \int_{t_{j-1}}^{t_j}\!\! & \big[ 
\| \theta(s)-\theta(t_j)\|_{H^1}^2 + \| \theta(s)-\theta(t_{j-1})\|_{H^1}^2 
\big] ds \Big|^p \Big)
\leq C \Big( \frac{T}{N}\Big)^{\eta p} .
\end{align} 
\end{prop}
\begin{proof}
(i) For $j=1, ...,N$, write  the decomposition \eqref{decomp_increm_u} of $u(t_j)-u(s)$ used in the proof of Lemma \ref{prop_regularity_L2} 
 (that is $\tau_1=s$, $\tau_2=t_j$),
 and apply $A^{\frac{1}{2}}$. 
The upper estimates of the sum of terms $A^{\frac{1}{2}} T_1(s,t_j)$,  $A^{\frac{1}{2}} T_2(s,t_j)$ obtained in the proof of Lemma 2.2 in
 \cite{BeMi_additive} imply for $\eta \in (0,1)$ 
 \begin{equation}		\label{mom_increm_T1-2_V1}
 \EE\Big( \Big| \sum_{j=1}^N \int_{t_{j-1}}^{t_{j}} \big[ \|A^{\frac{1}{2}} T_1(s,t_j)\|_{V^0}^2 + \|A^{\frac{1}{2}} T_2(s,t_j)\|_{V^0}^2\big] ds \Big|^p\Big)
 \leq C(\EE(\|u_0\|_{V_1}^{4p})) \, \Big(\frac{T}{N}\Big)^{\eta p}.
 \end{equation} 
The Minkowski inequality and the upper estimates \eqref{AS} and \eqref{A(I-S)} imply for $\delta \in (0,\frac{1}{2})$ 
\begin{align*}
\|A^{\frac{1}{2}} T_{3,1}(s,t_j) \|_{V^0} \leq & \int_0^{t_j} \| A^{\frac{1}{2} + \delta } S(t_j-s)\|_{{\mathcal L}(V^0;V^0)}  
\, \| A^{-\delta} \big[ S(t_j-s) - \mbox{\rm Id } \big] \|_{{\mathcal L}(V^0;V^0)} \\
&\quad \times \|  \Pi \theta(s) v_2\|_{V^0} \; ds\\
\leq & \; C |t_j-s|^{\delta} \, \sup_{s\in [0,t_j]} \|\theta(s)\|_{H^0}  \int_0^{t_j} (t_1-s)^{-(\frac{1}{2}+\delta)}  ds, 
\end{align*}
Hence we deduce
\begin{align*}
\Big| \sum_{j=1}^N &\int_{t_j{}-1}^{t_{j}} \|A^{\frac{1}{2}} T_{3,1}(s,t_j)\|_{V^0}^2 \, ds \Big|^p \leq  C \Big(\frac{T}{N}\Big)^{2p\delta} \sup_{s\in [0,T]} 
\|\theta(s)\|_{H^0}^{2p}  	\\
&\quad \times \Big| \sum_{j=1}^N \int_{t_{j-1}}^{t_{j}} \Big| \int_0^s (s-r)^{-(\frac{1}{2}+ \delta)} dr \Big|^2 ds \Big|^p\\
\leq &\,  C \Big(\frac{T}{N}\Big)^{2p\delta} \sup_{s\in [0,T]}  \|\theta(s)\|_{H^0}^{2p} \Big| \int_0^T s^{1-2\delta} ds \Big|^p 
\leq C \Big(\frac{T}{N}\Big)^{2p\delta} \sup_{s\in [0,T]}  \|\theta(s)\|_{H^0}^{2p}.
\end{align*}
Using once more the Minkowski inequality and \eqref{AS}, we obtain
\begin{align*}
\Big|  \sum_{j=1}^N \! \int_{t_{j-1}}^{t_{j}}\!\!  &\|A^{\frac{1}{2}} T_{3,2}(s,t_j)\|_{V^0}^2 \, ds \Big|^p \leq   \Big| \sum_{j=1}^N  \! \int_{t_j}^{t_{j+1}} 
\Big| \int_s^{t_j} \|A^{\frac{1}{2}} S(t_j-r) \Pi \theta(r) v_2\|_{V^0} dr\Big|^2  ds \Big|^p \\
\leq & \, C \sup_{r\in [0,T]} \|\theta(r)\|_{H^0}^{2p} \; \Big| \sum_{j=1}^N \int_{t_{j-1}}^{t_{j}} \Big| \int_s^{t_j} (t_j-s)^{-\frac{1}{2}} dr \Big|^2 ds \Big|^p \\
\leq & \, C  \sup_{r\in [0,T]} \|\theta(r)\|_{H^0}^{2p}   \Big(\frac{T}{N}\Big)^{p} .
\end{align*}
The above estimates of $T_{3,1}$ and $T_{3,2}$ together with \eqref{mom_t_L2} imply for $\eta \in (0,1)$
\begin{equation}		\label{mom_increm_T3_V1}
\EE\Big( \Big| \sum_{j=1}^N \int_{t_{j-1}}^{t_{j}} \|A^{\frac{1}{2}} T_{3}(s,t_j)\|_{V^0}^2 \, ds \Big|^p\Big) \leq  C  \Big(\frac{T}{N}\Big)^{\eta p} .
\end{equation}
We next study the stochastic integrals. \\
Using twice H\"older's inequality, the Burkholder inequality, \eqref{AS}, \eqref{A(I-S)} and the growth condition \eqref{growthG-1}, we obtain 
for $\delta \in (0, \frac{1}{2})$ 
\begin{align}		\label{mom_increm_T41_V1}
\EE\Big( \Big|& \sum_{j=1}^N \| A^{\frac{1}{2}} T_{4,1}(s,t_j)\|_{V^0}^2 ds \Big|^p\Big) \leq N^{p-1} \sum_{j=1}^N \EE\Big( \Big| \int_{t_j{}-1}^{t_{j}} \|
A^{\frac{1}{2}} T_{4,1}(s,t_j)\|_{V^0}^2
ds \Big|^p \Big)		\nonumber \\
\leq & N^{p-1} \Big( \frac{T}{N}\Big)^{p-1} \sum_{j=1}^N \! \int_{t_{j-1}}^{t_{j}} \EE\Big( \Big\| \int_0^s A^\delta S(s-r) A^{-\delta} [S(t_j-s)-{\rm Id}] 
A^{\frac{1}{2}} G(u(r)) dW(r) \Big\|_{V^0}^{2p}\Big) ds		\nonumber \\
\leq &\,  C_p T^{p-1} \sum_{j=1}^N \int_{t_{j-1}}^{t_j} \EE\Big( \Big|\int_0^s (s-r)^{-2\delta} (t_j-s)^{2\delta} \|A^{\frac{1}{2}} G(u(r)) \|_{{\mathcal L}(K,V^0)}^2
{\rm Tr}(Q) dr \Big|^p \Big) ds 		\nonumber \\
\leq & \, C \Big( \frac{T}{N}\Big)^{2\delta p} \int_0^T \EE\Big(\int_0^s  (s-r)^{-2\delta} \big[ K_2+K_3 \|u(r)\|_{V^1}^2 \big] dr \Big|^p \Big) ds 		\nonumber \\
\leq &\,  C \, \Big( \frac{T}{N}\Big)^{2\delta p} \Big[ 1+\EE\Big( \int_0^T  \|u(r)\|_{V^1}^{2p} \Big( \int_r^T  (s-r)^{-2\delta}  ds \Big) dr \Big) \Big] \leq C \, \Big( \frac{T}{N}\Big)^{2\delta p} ,
\end{align}
where the last upper estimates are deduced from the Fubini theorem, from \eqref{mom_u_L2} and \eqref{mom_u_V1}. 

A similar argument proves
\begin{align}	\label{mom_increm_T42_V1}	
\EE\Big( \Big|& \sum_{j=1}^N \| A^{\frac{1}{2}} T_{4,2}(s,t_j)\|_{V^0}^2 ds \Big|^p\Big) 
\nonumber \\
 \leq & \, 
T^{p-1}  \sum_{j=1}^N \int_{t_{j-1}}^{t_j} \EE\Big( \Big\| \int_s^{t_j} S(t_j-s) A^{\frac{1}{2}} G(u(r)) dW(r)\Big\|_{V^0}^{2p} \Big) ds 	
		\nonumber\\
\leq & \, C_p \, T^{p-1}  {\rm Tr}(Q)^p \sum_{j=1}^N \int_{t_{j-1}}^{t_j} \EE\Big( \Big| \int_s^{t_j} \big[ K_2 + K_3 \|u(r)\|_{V^1}^2 \big]  dr \Big|^p \Big) ds
	\nonumber \\
\leq & \, C \sum_{j=1}^N \int_{t_{j-1}}^{t_j} \Big( \frac{T}{N}\Big)^{p-1} \int_s^{t_j} \big[ K_2^p+K_3^p  \EE(\|A^{\frac{1}{2}} u(r)\|_{V^0}^{2p})\big]  dr   ds 
\nonumber \\
\leq & \, C \Big( \frac{T}{N}\Big)^{p-1} \sum_{j=1}^N \int_{t_{j-1}}^{t_j}  \big[ K_2^p+K_3^p  \EE(\|A^{\frac{1}{2}} u(r)\|_{V^0}^{2p})\big] 
\Big(\int_r^{t_j} ds \Big) dr 		\nonumber \\
\leq &\,  C  \Big( \frac{T}{N}\Big)^{ p} \Big[ 1+ \int_0^T \EE( \|A^{\frac{1}{2}} u(s)\|_{V^0}^{2p})  ds \Big] \leq C \Big( \frac{T}{N}\Big)^{ p}.
\end{align}  
The inequalities \eqref{mom_increm_T41_V1} and \eqref{mom_increm_T42_V1} imply for $\eta \in (0,1)$,
\begin{equation} 	\label{mom_increm_T4_V1}
\EE\Big( \Big| \sum_{j=1}^N\int_{t_{j-1}}^{t_j}  \| A^{\frac{1}{2}} T_{4}(s,t_j)\|_{V^0}^2 ds \Big|^p\Big)  \leq C \Big( \frac{T}{N}\Big)^{\eta p}.
\end{equation} 
The above arguments \eqref{mom_increm_T1-2_V1}, \eqref{mom_increm_T3_V1} and \eqref{mom_increm_T4_V1}   prove similar inequalities 
when replacing $T_i(s,t_j)$ by $T_i(t_{j-1},s)$ for $i=1, ...,4$ and $j=1, ...,N$.
Using \eqref{increm_u_L2}, this concludes the proof of \eqref{mom_increm_u_1}.  

(ii) As above, we apply $\AA^{\frac{1}{2}}$ to the terms $\tilde{T}_i(s,t_j), i=1,2,3$ of the decomposition \eqref{decom_increm_t} of $\theta(t_j)-\theta(s)$
introduced in the proof of  Proposition \ref{prop_regularity_L2} (ii). 
For $\delta \in (0,\frac{1}{2})$, the inequalities \eqref{AS} and \eqref{A(I-S)} imply 
\begin{align*}	
\Big| \sum_{j=1}^N \int_{t_{j-1}}^{t_j} \!\! \| \AA^{\frac{1}{2}} \SS(s) \big[ \SS(\tau_j-s) - {\rm Id} \big] \theta_0\|_{H^0}^2 \Big|^p 
\leq &\, \Big| \sum_{j=1}^N \int_{t_{j-1}}^{t_j}\!\!  \| \AA^\delta \SS(s) \AA^{-\delta} \big[ \SS(\tau_j-s) - {\rm Id} \big] 
\AA^{\frac{1}{2}} \theta_0\|_{H^0}^2 \Big|^p 	 \\
\leq & \, C \Big| \sum_{j=1}^N \int_{t_{j-1}}^{t_j} \! s^{-2\delta} \Big( \frac{T}{N}\Big)^{2\delta} \|\AA^{\frac{1}{2}} \theta_0\|_{H^0}^2 \, ds \Big|^p	 \\
\leq & \, C \Big( \frac{T}{N}\Big)^{2\delta p}  \, \|\AA^{\frac{1}{2}} \theta_0\|_{H^0}^{2p} \Big| \int_0^T s^{-2\delta} ds \Big|^p .
\end{align*}
Hence for $\eta \in (0,1)$, 
\begin{equation}\label{incremH1_TT1_t}
\EE\Big( \Big| \sum_{j=1}^N \int_{t_{j-1}}^{t_j} \| \AA^{\frac{1}{2}} \SS(s) \big[ \SS(\tau_j-s) - {\rm Id} \big] \theta_0\|_{H^0}^2 \Big|^p \Big)
\leq C \EE\big( \|\theta_0\|_{H^0}^{2p}\big)\, \Big( \frac{T}{N}\Big)^{\eta p}.
\end{equation} 
Let $\beta\in (0,\frac{1}{2})$ and $\delta \in (0, \frac{1}{2}-\delta)$. The Minkowski inequality, \eqref{AS}, \eqref{A(I-S)} and \eqref{GiMi-ut} applied 
with $\alpha=\rho=\frac{1}{2}$ imply for $s\in [t_{j-1},t_j]$
\begin{align*}
\Big\| \int_0^s \AA^{\frac{1}{2}}& \SS(s-r) \big[ \SS(t_j-s)-{\rm Id}\big] \big( [u(r).\nabla] \theta(r) \big) dr \Big\|_{H^0}  \\
\leq & \, \int_0^s \|\AA^{\frac{1}{2} + \beta +\delta} \SS(s-r) \, \AA^{-\beta} \big[ \SS(t_j-s)-{\rm Id}\big]  \,
 \AA^{-\delta} \big( [u(r).\nabla] \theta(r) \big) \, \|_{H^0} dr \\
 \leq & \, C \int_0^s (s-r)^{-(\frac{1}{2}+\beta+\delta)} \Big( \frac{T}{N}\big)^\beta \|A^{\frac{1}{2}}  u(r)\|_{V^0} \,
   \| \AA^{\frac{1}{2}} \theta(r) \|_{H^0}\,  dr.
\end{align*}
Therefore, the Cauchy-Schwarz inequality and Fubini's theorem imply
\begin{align*}
\EE\Big(& \Big| \sum_{j=1}^N \int_{t_{j-1}}^{t_j} \|\AA^{\frac{1}{2}} \tilde{T}_{2,1}(s,t_j)\|_{H^0}^2 ds \Big|^p\Big) \\
\leq & \, C \Big( \frac{T}{N}\Big)^{2\beta p} \; \EE\Big[ \sup_{s\in [0,T]} \|A^{\frac{1}{2}} u(s)\|_{V^0}^{2p} \Big| \sum_{j=1}^N \int_{t_j}^{t_{j+1}}
\Big( \int_0^s 
(s-r)^{-(\frac{1}{2} + \beta + \delta)} \|\AA^{\frac{1}{2}} \theta(r)\|_{H^0}^2 \Big)  \\
&\qquad \times \Big( \int_0^s  (s-r)^{-(\frac{1}{2} + \beta + \delta)}  dr\Big) ds\Big|^p \Big]\\
\leq & \, C \Big( \frac{T}{N}\Big)^{2\beta p} \; \EE\Big[ \sup_{s\in [0,T]} \|A^{\frac{1}{2}}  u(s)\|_{V^0}^{2p} 
\Big|\Big( \int_0^T \!\! \| \AA^{\frac{1}{2}} \theta(r)\|_{H^0}^2 ds \Big)
\Big( \int_r^T (s-r)^{-(\frac{1}{2}+\beta+\delta)} ds \Big) dr \Big|^p \Big] \\
\leq & \, C \Big( \frac{T}{N}\Big)^{2\beta p}\, \Big\{ \EE\Big( \sup_{s\in [0,T]} \|A^{\frac{1}{2}}  u(s)\|_{V^0}^{4p} \Big)\Big\}^{\frac{1}{2}}
\Big\{ \int_0^T \EE\big( \|\AA^{\frac{1}{2}} \theta(r)\|_{H^0}^{4p} \big) dr \Big\}^{\frac{1}{2}}
\end{align*} 
The upper estimates \eqref{mom_u_V1} and \eqref{sup_E_t} imply for $\eta\in (0,1)$ 
\begin{equation}	\label{mom_increm_tH1_T21}
\EE\Big( \Big| \sum_{j=1}^N \int_{t_{j-1}}^{t_j} \|\AA^{\frac{1}{2}} \tilde{T}_{2,1}(s,t_j)\|_{H^0}^2 ds \Big|^p\Big) \leq C \Big( \frac{T}{N}\Big)^{\eta p}.
\end{equation}
Using the Minkowski inequality, \eqref{AS} and \eqref{GiMi-ut} with $\alpha=\rho=\frac{1}{2}$, and Fubini's theorem, 
we obtain for $\delta \in (0,\frac{1}{2})$
\begin{align*}
\Big| \sum_{j=1}^N &\int_{t_{j-1}}^{t_j} \!\! \!\|\AA^{\frac{1}{2}} \tilde{T}_{2,2}(s,t_j)\|_{H^0}^2 ds \Big|^p \leq  \Big|\sum_{j=1}^N  \int_{t_{j-1}}^{t_j} 
\! \Big| \int_s^{t_j} (t_j-r)^{-(\frac{1}{2}+\delta)} \|A^{\frac{1}{2}} u(r)\|_{V^0} \|\AA^{\frac{1}{2}} \theta(r)\|_{H^0} dr \Big|^2 ds \Big|^p\\
\leq & \, C \sup_{r\in [0,T]} \|A^{\frac{1}{2}} u(r)\|_{V^0}^{2p} \Big|\sum_{j=1}^N  \int_{t_{j-1}}^{t_j} \!\! 
\Big( \int_{s}^{t_j} (t_j-r)^{-(\frac{1}{2}+\delta)} \|\AA^{\frac{1}{2}}
\theta(s)\|_{H^0}^2 dr\Big)\\
&\quad \times  \Big( \int_{s}^{t_j} (t_j-r)^{-(\frac{1}{2}+\delta)} dr\Big) ds \Big|^p \\
\leq & \, C \sup_{r\in [0,T]} \|A^{\frac{1}{2}} u(r)\|_{V^0}^{2p} \Big|\sum_{j=1}^N  \int_{t_{j-1}}^{t_j} \|\AA^{\frac{1}{2}} \theta(r)\|_{H^0}^2 
\Big(\int_r^{t_j} ds \Big) dr \Big|^p \\
\leq & \, C  \sup_{r\in [0,T]} \|A^{\frac{1}{2}} u(r)\|_{V^0}^{2p} \Big( \frac{T}{N}\Big)^{ p} \int_0^T \| \AA^{\frac{1}{2}} \theta(s)\|_{H^0}^{2p} dr
\end{align*}
The Cauchy-Schwarz inequality, \eqref{mom_u_V1} and \eqref{sup_E_t} imply
\begin{equation}	\label{mom_increm_tH1_TT22}
\EE\Big( \Big| \sum_{j=1}^N \int_{t_{j-1}}^{t_j} \|\AA^{\frac{1}{2}} \tilde{T}_{2,2}(s,t_j)\|_{H^0}^2 ds \Big|^p\Big) \leq C \Big( \frac{T}{N}\Big)^{p}.
\end{equation}
Finally, arguments similar to that used to prove \eqref{mom_increm_T4_V1}  imply for $\eta \in (0,1)$
\begin{equation}		\label{mom_increm_TT3_V1} 
\EE\Big( \Big| \sum_{j=1}^N  \int_{t_{j-1}}^{t_j} \|\AA^{\frac{1}{2}} \tilde{T}_{3}(s,t_j)\|_{V^0}^2 ds \Big|^p\Big)  \leq C \Big( \frac{T}{N}\Big)^{\eta p}.
\end{equation} 
The upper estimates \eqref{incremH1_TT1_t}--\eqref{mom_increm_TT3_V1} conclude the proof of
\[ \EE\Big( \Big| \sum_{j=1}^N  \int_{t_{j-1}}^{t_j} \|\AA^{\frac{1}{2}} \big[ \theta(t_j)-\theta(s)\big]\|_{H^0}^2 ds \Big|^p \Big) \leq 
C \Big( \frac{T}{N}\Big)^{\eta p}, \quad \eta \in (0,1). 
\]
Using \eqref{increm_t_L2},  similar argument  completes the proof of  \eqref{mom_increm_t_1}. 
\end{proof}

\section{ The  implicit time  Euler scheme} \label{sEuler}
We  first prove the existence of the fully time implicit time Euler scheme  $\{ u^k ; k=0, 1, ...,N\}$  and $\{ \theta^k ; k=0, 1, ...,N\}$ defined by
\eqref{Euler_u}--\eqref{Euler_t}. Set $\Delta_l W:= W(t_l)-W(t_{l-1})$ and $\Delta_l \WW=\WW(t_l)-\WW(t_{l-1})$, $l=1, ...,N$.

\subsection{Existence of the scheme} 

\noindent {\it Proof of Proposition \ref{prop_ul}}
 The proof is divided in two steps.\\
{\bf Step 1} For technical reasons we consider a Galerkin approximation.
Let $\{e_l\}_l$ denote an orthonormal basis of $V^0$ made of elements of $V^2$ which are orthogonal in $V^1$ (resp. 
 $\{\tilde{e}_l\}_l$ denote an orthonormal basis of $H^0$ made of elements of $H^2$ which are orthogonal in $H^1$). \\
 For $m=1,2, ...$ let
 $V_m=\mbox{ \rm span }(e_1, ..., e_m) \subset V^2$  and let $P_m:V^0\to V_m$ denote the projection from 
$V^0$ to $V_m$. Similarly, let $\tilde{H}_m=\mbox{ \rm span }(\tilde{e}_1, ...,\tilde{e}_m) \subset H^2$  and let
 $\tilde{P}_m:H^0\to \tilde{H}_m$ denote the projection from 
$H^0$ to $\tilde{H}_m$. 

 In order to find a solution to  \eqref{Euler_u}--\eqref{Euler_t}  we project these equations on $V_m$ and $\tilde{H}_m $ respectively, that
is we define by induction  $\{u^k(m)\}_{k=0, ...,N} \in V_m$ and $\{\theta^k(m)\}_{k=0, ...,N} \in \tilde{H}_m$ such that 
$u^0(m)= P_m(u_0)$, $\theta^0(m)=\tilde{P}_m(\theta_0)$, and for $k=1, ...,N$, $\varphi \in V_m$ and $\psi\in \tilde{H}_m$
\begin{align}		\label{def-ukm}
\big( u^k(m)-u^{k-1}(m), \varphi\big)  &+ h \Big[ \nu \big( A^{\frac{1}{2}} u^k(m), A^{\frac{1}{2}}  \varphi)  + 
 \big\langle B\big( u^k(m), u^k(m)\big), \varphi \big\rangle  \nonumber  \\
& = h \big( \Pi \theta^{k-1} v_2 , \varphi\big) +  \big( G(u^{k-1}(m)) \Delta_k W \, , \, \varphi\big) \\
\big( \theta^k(m)-\theta^{k-1}(m), \psi \big)  &+ h \Big[ \kappa \big( \AA^{\frac{1}{2}} \theta^k(m), \AA^{\frac{1}{2}}  \psi)  + 
 \big\langle [ u^{k-1}k(m) . \nabla] \theta^k(m) \big), \psi \big\rangle  \nonumber  \\
 &= \big( \GG(\theta^{k-1}(m)) \Delta_k \WW \, , \, \psi\big)			\label{def-tkm}
\end{align}
For almost every $\omega$ set $R(0,\omega):= \|u_0(\omega)\|_{V^0}$ and $R(0,\omega):= \|\theta_0(\omega)\|_{H^0}$. 
Fix  $k=1, ...,N$ and suppose that  for $j=0, ..., k-1$ the ${\mathcal F}_{t_j}$-
measurable random variables 
$u^j(m)$and $\theta^j(m)$  have been defined, and that  
\[ R(j,\omega):=\sup_{m\geq 1}
\|u^j(m,\omega)\|_{\LL^2}<\infty\quad   {\rm and }\quad  \tilde{R}(j,\omega):=\sup_{m\geq 1}
\|\theta^j(m,\omega)\|_{\LL^2}<\infty \;\]
 for almost every  $\omega$. We prove that $u^k(m)$ and $\theta^k(m)$ exists and satisfy  $ \sup_{m\geq 1} \|u^k(m,\omega)\|_{V^0}<\infty$
 and $ \sup_{m\geq 1} \|\theta^k(m,\omega)\|_{H^0}<\infty$ a.s. 

For $\omega \in \Omega$ let $\Phi^k_{m,\omega}:V_m\to V_m$ (resp. $\tilde{\Phi}^k_{m,  \omega}$) be defined for $f\in V_m$  
(resp. for $\tilde{f}\in \tilde{H}_m$) as the solution of
\begin{align*}
\big(\Phi^k_{m,\omega}(f) &, \varphi\big) = \; \big( f-u^{k-1}(m,\omega), \varphi\big)  + h \Big[ \nu \big( A^{\frac{1}{2}}  f, A^{\frac{1}{2}} \varphi\big) 
+ \big\langle P_m B(f,f),\varphi
\big\rangle  \\
&\quad  - \big( \Pi \theta^{k-1}(m) v_2,  \varphi\big)\Big] 
- \big( P_m G(u^{k-1}(m,\omega)) \Delta_kW(\omega), \varphi \big), \qquad \forall \varphi \in V_m\\
\big(\tilde{\Phi}^k_{m,\omega}(\tilde{f}) , \psi\big) = &\; \big( \tilde{f}-\theta^{k-1}(m,\omega), \psi\big)  + h \Big[ \kappa \big( \AA^{\frac{1}{2}}  \tilde{f},
\AA^{\frac{1}{2}}  \psi\big)  + \big\langle [u^{k-1}(m).\AA^{\frac{1}{2}} ] \tilde{f}], \psi\big\rangle \\
&\quad - \big( \tilde{P}_m \GG(\theta^{k-1}(m,\omega)) \Delta_k \WW(\omega), \psi \big), \qquad \forall \psi \in \tilde{H}_m. 
\end{align*}
Then the Cauchy-Schwarz and Young inequalities imply 
\begin{align*}
 \big| \big( u^{k-1}(m,\omega),f\big)\big| \leq&\,  \frac{1}{4} \|f\|_{V^0}^2 + \|u^{k-1}(m,\omega)\|_{V^0}^2,  \\
 \big| \big( \theta^{k-1}(m,\omega),\tilde{f}\big)\big| \leq &\, \frac{1}{4} \|\tilde{f} \|_{H^0}^2 + \|\theta^{k-1}(m,\omega)\|_{H^0}^2, \\
 \big| \big( \Pi \theta^{k-1}(m,\omega) , f\big) \big| \leq &\, \frac{1}{4} \| {f} \|_{V^0}^2 + \|\theta^{k-1}(m,\omega)\|_{H^0}^2, \\
 \big| \big( G(u^{k-1}(m,\omega)) \Delta_k W , f\big) \leq &\,  \frac{1}{4} \| {f} \|_{V^0}^2 + \|G(u^{k-1}(m,\omega))\|_{{\mathcal L}(K,V^0)}^2
 \| \Delta_k W\|_K^2\\
\leq  &\,  \frac{1}{4} \| {f} \|_{V^0}^2 + \big[ K_0+K_1\|u^{k-1}(m,\omega)\|_{V^0}^2\big] \|\Delta_k W\|_K^2, \\
\big| \big( \GG(\theta^{k-1}(m,\omega)) \Delta_k \WW , \tilde{f}\big) 
 \leq &\,  
\frac{1}{4} \| \tilde{f} \|_{H^0}^2 + \|\GG(\theta^{k-1}(m,\omega)) \|_{{\mathcal L}(\KK,H^0)}^2
 \| \Delta_k \WW\|_K^2\\
\leq  &\,  \frac{1}{4} \| \tilde{f} \|_{H^0}^2 + \big[ \tilde{K}_0+\tilde{K}_1\|u^{k-1}(m,\omega)\|_{H^0}^2\big] \|\Delta_k \WW\|_{\KK}^2. 
\end{align*} 
If 
\begin{align*}
 \|f\|_{V^0}^2 =R^2(k,\omega) :=  &\, 4\Big[  R^2(k-1,\omega) +  \big( h \tilde{R}(k-1,\omega)\big)^2  \\
 &\qquad 
 + \big[K_0 + K_1\ R^2(k-1,\omega) \big] \|\Delta_k W(\omega)\|^2_{K}
\Big) \Big], \\
\|\tilde{f}\|_{H^0}^2 = \tilde{R}^2(k,\omega) := &\,  2\Big[ \tilde{R}^2(k-1,\omega) + \big( \tilde{K}_0 + \tilde{K}_1 \tilde{R}^2(k-1,\omega)\big)
\| \Delta_k \WW(\omega)\|_{\KK}^2\Big], 
\end{align*}
we deduce 
\begin{align*}
 \big( \Phi^k_{m,\omega}(f) ,f\big) \geq &\frac{1}{4} \|f\|_{\LL^2}^2 -   \|u^{k-1}(m,\omega)\|_{\LL^2}^2  + h \nu \|A^{\frac{1}{2}} f \|_{V^0}^2 
 - h^2 \| \theta^{k-1}(m,\omega)\|_{H^0}^2  \\
 & \quad - \big[ K_0 + K_1  \|u^{k-1}(m,\omega)\|_{V^0}^2 \big] \|\Delta_k W(\omega)\|^2_K  \geq 0 \\
 \big( \tilde{\Phi}^k_{m,\omega}(f) ,\tilde{f}\big) \geq &\frac{1}{2} \|\tilde{f} \|_{H^0}^2 -   \|\theta^{k-1}(m,\omega)\|_{H^0}^2  
 + h\|\AA^{\frac{1}{2}} \tilde{f}\|_{H^0}^2 \\
 &\quad - \big[ \tilde{K}_0 + \tilde{K}_1  \|\theta^{k-1}(m,\omega)\|_{H^0}^2 \big] \|\Delta_k \WW(\omega)\|^2_{\tilde{K}}
   \geq 0 .
 \end{align*}
 Using \cite[Cor 1.1]{GirRav} page 279, which can be deduced from Brouwer's theorem,  we deduce the existence of an element $u^k(m,\omega) \in V(m)$
(resp. $\theta^k(m,\omega)\in \tilde{H}(m)$), such that $\Phi^k(m,\omega)(u^k(m,\omega))=0$ (resp. $\tilde{\Phi}^k(m,\omega)(\theta^k(m,\omega))=0$)
and $\| u^k(m,\omega)\|_{V^0}^2 \leq R^2(k,\omega)$ (resp. $\| \theta^k(m,\omega)\|_{H^0} \leq \tilde{R}^2(k,\omega)$) a.s. 
Note that these elements $u^k(m,\omega)$ and $\theta^k(m,\omega)$ need not be unique. 
Furthermore, the random variables $u^k(m)$ and $\theta^k(\omega)$ are ${\mathcal F}_{t_k}$-measurable. 

The definition of $u^k(m)$  (resp. $\theta^k(m)$) implies that it is a solution to \eqref{def-ukm} (resp. \eqref{def-tkm}).
 Taking $\varphi = u^k(m)$ in \eqref{def-ukm}, using the antisymmetry  property \eqref{B}  and the Young
inequality, we obtain 
\begin{align*}
\| u^k(m)&\|_{V^0}^2 + h\, \nu  \|A^{\frac{1}{2}}  u^k(m)\|_{V^0}^2 =  \big( u^{k-1}(m), u^k(m)\big)  + h\, \big( \Pi \theta^{k-1}(m) v_2, u^k(m)\big)  \\
&+  \big( G(u^{k-1}(m) \Delta_k W, u^k(m)\big)
\\
&\leq \frac{3}{4} \|u^k(m)\|_{V^0}^2 + \|u^{k-1}(m)\|_{V^0}^2  
+ \| \theta^{k-1}(m)\|_{H^0}^2 + \big[ K_0+K_1 \|u^{k-1}(m)\|_{V^0}^2\big]
\|\Delta_k W\|_{K}^2.
\end{align*} 
Hence a.s. 
\begin{align*}
 \sup_{m\geq 1} \Big[ \frac{1}{4} \|u^k(m,\omega)\|_{V^0}^2 &+ h\, \nu \|A^{\frac{1}{2}}  u^k(m,\omega)\|_{V^0}^2\Big]  \leq  R^2(k-1,\omega) 
 +  \tilde{R}^2(k-1,\omega)
\\
& + \big[ K_0 + K_1  R^2(k-1,\omega)\big]  \| \Delta_k W(\omega)\|_{K}^2 .
\end{align*}  
A similar computation using $\psi=\theta^k(m)$ in \eqref{def-tkm} implies
\begin{align*}
\sup_{m\geq 1} \Big[ \frac{1}{2} \|\theta^k(m,\omega)\|_{H^0}^2 &+ h\, \kappa \|\AA^{\frac{1}{2}}  \theta^k(m,\omega)\|_{H^0}^2 \Big] \leq
\tilde{R}^2(k-1) + \big[ \tilde{K}_0 + \tilde{K}_1  \tilde{R}^2(k-1)\big] \| \Delta_k \WW\|_{\KK}. 
\end{align*} 
Therefore, for fixed $k$ and almost every $\omega$, the sequence $\{u^k(m,\omega)\}_m$ is bounded in $V^1$; it has a
subsequence (still denoted $\{u^k(m,\omega)\}_m$) which converges weakly in $V^1$ to $\phi_k(\omega)$. The random variable
$\phi_k$ is ${\mathcal F}_{t_k}$-measurable. Similarly,  for fixed $k$ and almost every $\omega$, the sequence
 $\{\theta^k(m,\omega)\}_m$ is bounded in $H^1$; it has a
subsequence (still denoted $\{\theta^k(m,\omega)\}_m$) which converges weakly in $H^1$ to $\tilde{\phi}_k(\omega)$ which is 
${\mathcal F}_{t_k}$-measurable.

Since $D$ is bounded, the embedding of  $V^1$  in $V^0$ (rsp. of $H^1$ in $H^0$) is compact; hence the subsequence 
$\{u^k(m,\omega)\}_m$ converges strongly to  $\phi_k(\omega)$ in $V^0$ (resp. $\{ \theta^k(m,\omega)\}_m$ converges strongly to
$\tilde{\phi}_k(\omega)$ in $H^0$). \\
 {\bf Step 2} We next prove that the pair $(\phi_k, \tilde{\phi}_k)$ is a solution to \eqref{Euler_u}--\eqref{Euler_t}.
  By definition $u^0(m)$ converges strongly to $u_0$ in $V^0$, and $\theta^0(m)$ converges strongly to $\theta_0$ in $H^0$. 
  We next prove by induction on $k$ that the pair $(\phi^k, \tilde{\phi}^k)$ solves \eqref{Euler_u}--\eqref{Euler_t}.
  Fix a positive integer $m_0$ and consider the equation \eqref{def-ukm} for $k=1, ..., N$, 
   $\varphi \in V_{m_0}$,  and $m\geq m_0$.
  As $m\to \infty$ we have a.s. 
  \begin{align*}
  & \big(u^k(m)-u^{k-1}(m) , \varphi) \to \big(\phi^k-\phi^{k-1} , \varphi), \quad \big( A^{\frac{1}{2}} u^k(m), A^{\frac{1}{2}} \phi\big) \to  
  \big( A^{\frac{1}{2}} \phi^k, A^{\frac{1}{2}} \phi\big), \\
  &\big( \Pi \theta^{k-1}(m) v_2, \varphi\big ) = \big(  \theta^{k-1}(m) v_2, \varphi\big )  \to \big( \tilde{\phi}^k v_2, \varphi). 
  \end{align*}
  Furthermore, the antisymmetry of $B$ \eqref{B} and the Gagliardo-Nirenberg inequality \eqref{GagNir} yield a.s. 
  \begin{align*}
  \big| \big\langle B\big(&u^k(m),u^k(m)\big) - B(\phi^k,\phi^k), \varphi \big\rangle \big| \\
   &\leq \; \big| \big\langle B\big(u^k(m)-\phi^k, \varphi \big), u^k(m)  
  \big\rangle \big| +\big| \big\langle B\big(\phi^k, \varphi \big), u^k(m)  -\phi^k \big\rangle \big| \\
   &\leq \; \|A^{\frac{1}{2}}  \varphi\|_{V^0} \|u^k(m)-\phi^k\|_{\LL^4} \, \big[ \|u^k(m)\|_{\LL^4} + \| \phi^k\|_{\LL^4}\big] \\
   &\leq \; C \; \|\varphi\|_{V^0}  \big[ \max_{m}\|u^k(m)\|_{V^1} + \|\phi^k\|_{V^1}\big]  ] \| A^{\frac{1}{2}} u^k(m)-\phi^k\|_{V^0}^{\frac{1}{2}}
  \| u^k(m)-\phi^k\|_{V^0}^{\frac{1}{2}}  \to 0
  \end{align*} 
 as $m\to \infty$.
Finally, the Cauchy-Schwarz inequality and the Lipschitz condition \eqref{LipG} imply
\begin{align*}
\big| \big(  \big[ G\big( u^{k-1}(m)\big)  - G\big(  \phi^{k-1} \big)  \big] \Delta_kW, \varphi\big) \big|  &
\leq \|\varphi\|_{V^0}
 \| G(u^{k-1}(m)- G(\phi^{k-1})  \|_{{\mathcal L}(K;V^0)}  \| \Delta_kW\|_K \\
 &\leq  \sqrt{L_1}\, \|\varphi\|_{\LL^2}\, \|u^{k-1}(m)-\phi^{k-1}\|_{\LL^2} \, \| \Delta_kW\|_K  \to 0
\end{align*}
as $m\to \infty$. Therefore, letting $m\to \infty$ in \eqref{def-ukm}, we deduce 
\begin{align*}
\Big( \phi^k -\phi^{k-1}  &+ h \nu A \phi^k + h B\big( \phi^k ,\phi^k \big)  , \varphi\Big) 
= \big( \Pi \theta^{k-1} v_2, \varphi\big) + \big( G(\phi^{k-1} ) \Delta_k W \, , \, \varphi), \quad \forall \varphi \in V_{m_0}. 
\end{align*} 
Since $\cup_{m_0} V_{m_0}$ is dense in $V$, we deduce that $\{\phi^k\}_{k=0, ...,N}$ is a solution to \eqref{Euler_u}. 

A similar argument proves that $\tilde{\phi}^k$ is a solution to \eqref{Euler_t}. This concludes the proof.
\hfill $\Box$
\subsection{Moments of the Euler scheme}
We next prove upper bounds of moments of $u^k$ and $\theta^k$ uniformly in $k=1, ..., N$. 
\begin{prop}		\label{prop_mom_scheme}
Let $G$ and $\tilde{G}$ satisfy the condition {\bf (C-u)(i)} and {\bf (C-$\theta$)(i)} respectively. Let $K\geq 1$ be an integer, and let
 $u_0\in L^{2^K}(\Omega;V^0) $ and
$\theta_0\in L^{2^K}(\Omega;H^0)$ respectively. Let $\{u^k\}_{k=0, ...,N}$ and $\{\theta_k\}_{k=0, ..., N}$ be solution of \eqref{Euler_u} and
\eqref{Euler_t} respectively. Then
\begin{align}		\label{mom_scheme_0}
\sup_{N\geq 1} \EE\Big(   \max_{0\leq L\leq N} \|u^L\|_{V^0}^{2^K} + \max_{0\leq L \leq N} \|\theta^L\|_{H^0}^{2^K} \Big)  <\infty \\
\sup_{N\geq 1} \EE\Big( h \sum_{l=1}^N \|A^{\frac{1}{2}} u^l\|_{V^0}^2 \|u^l\|_{V^0}^{2^K-2} +  h \sum_{l=1}^N \|\AA^{\frac{1}{2}} \theta^l\|_{H^0}^2 
\|\theta^l\|_{H^0}^{2^K-2}  <\infty, 	\label{mom_scheme_1} 
\end{align} 
\end{prop}
\begin{proof} 
 Write \eqref{Euler_u} with $\varphi=u^l$, \eqref{Euler_t} with $\psi=\theta^l$,  and use the identity 
 $(f,f-g) = \frac{1}{2} \big[ \|f\|_{\LL^2 }- \|g\|_{\LL^2}^2 
+ \| f-g\|_{\LL^2}^2\big]$. Using the Cauchy-Schwarz and Young inequalities, the antisymmetry \eqref{B} 
and the growth condition \eqref{growthG-0}, this yields for $l=1, ..., N$  
\begin{align}
\frac{1}{2} \big[ \|u^l\|_{V^0}^2 - \|u^{l-1}\|_{V^0}^2 + \| u^l-u^{l-1}\|_{V^0}^2 \big] + h \nu \|A^{\frac{1}{2}} u^l\|_{V^0}^2 =&
h ( \Pi \theta^{l-1} e_2, u^l) + \big( G(u^{l-1}) \Delta_l W, u^l), \label{dif_u}\\
\frac{1}{2} \big[ \|\theta^l\|_{H^0}^2 - \|\theta^{l-1}\|_{H^0}^2 + \| \theta^l-\theta^{l-1}\|_{H^0}^2 \big] + h \kappa \|\AA^{\frac{1}{2}} \theta^l\|_{H^0}^2 =&
 \big( \GG(\theta^{l-1}) \Delta_l \WW, \theta^l). \label{dif_theta}
\end{align} 
Fix $L=1, ...,N$ and add both equalities for $l=1, ...,L$; this yields
\begin{align}		\label{u^L-t^L}
\frac{1}{2} &\big[ \|u^L\|_{V^0}^2 - \|u_0\|_{V^0}^2 + \|\theta^L\|_{H^0}^2 - \|\theta_0\|_{H^0}^2 \big] + \frac{1}{2} \Big[ \sum_{l=1}^L \|u^l-u^{l-1}\|_{V^0}^2
+ \sum_{l=1}^L \| \theta^l-\theta^{l-1}\|_{H^0}^2\Big] 		\nonumber \\
&\; + h \sum_{l=1}^L \big[ \nu \|A^{\frac{1}{2}} u^l\|_{V^0}^2 + \kappa \|\AA^{\frac{1}{2}} \theta^l\|_{H^0}^2\big] \leq \frac{h}{2} \sum_{l=0}^{L-1}
\|\theta^l\|_{H_0}^2 + h \sum_{l=1}^{L-1} \|u^l\|_{V^0}^2 	+ h \|u^L\|_{V^0}^2 	\nonumber \\
&\; + \sum_{l=1}^L \big[ \|G(u^{l-1})\|_{{\mathcal L}(K;V^0)}^2 \|\Delta_l W\|_K^2 +  \frac{1}{4} \|u^l-u^{l-1}\|_{V^0}^2  \big] +
\sum_{l=1}^L \big( G(u^{l-1}) \Delta_l W , u^{l-1}\big) 		\nonumber \\
& \; + \sum_{l=1}^L \big[ \|\GG(\theta^{l-1})\|_{{\mathcal L}(\tilde{K};H^0)}^2 \|\Delta_l \WW\|_{\tilde K}^2 + \frac{1}{4}  \|\theta^l-\theta^{l-1}\|_{H^0}^2
  \big] +
\sum_{l=1}^L \big( \GG(\theta^{l-1}) \Delta_l \WW , \theta^{l-1}\big) 	.
\end{align}
Let $N$ be large enough to have $h=\frac{T}{N} \leq \frac{1}{8}$. Taking expected values, we deduce
\begin{align*}
\EE\Big( &\|u^L\|_{V^0}^2 + \|\theta^L\|_{H^0}^2 +\frac{1}{2}  \sum_{l=1}^L \big[ \|u^l-u^{l-1}\|_{V^0}^2 + \|\theta^l-\theta^{l-1}\|_{V^0}^2 \big]\\
& +2  h \sum_{l=1}^N \big[ \nu \|A^{\frac{1}{2}} u^l\|_{V^0}^2 + \kappa \|\AA^{\frac{1}{2}} \theta^l\|_{H^0}^2\big] \Big) \leq \EE\big( \|u^0\|_{V^0}^2
+ \|\theta^0\|_{H^0}^2\big)  + 2T \big[ K_0 {\rm Tr}(Q)  + \tilde{K}_0 {\rm Tr}(\QQ)\big]  \\
& + h \big[ 4+2\max(K_1 {\rm Tr}(Q) , \tilde{K}_1 {\rm Tr}(\QQ) \big] 
\sum_{l=0}^{L-1} \EE\big( \|u^l\|_{V^0}^2 + \|\theta^l\|_{H^0}^2\big) .
\end{align*} 
Neglecting both sums in the left hand side and using the discrete Gronwall lemma, we deduce 
\begin{equation}		\label{mom_K=1}
\sup_{1\leq L\leq N} \EE\big( \|u^L\|_{V^0}^2 + \|\theta^L\|_{H^0}^2 \big) \leq C,
\end{equation}
where 
\[ C= \Big( 2\EE\big( \|u^0\|_{V^0}^2
+ \|\theta^0\|_{H^0}^2\big)  + 2T \big[ K_0 {\rm Tr}(Q)  + \tilde{K}_0 {\rm Tr}(\QQ)\big] \Big) 
 e^{T\big[ 4 +2\max(K_1 {\rm Tr}(Q) , \tilde{K}_1 {\rm Tr}(\QQ)\big] }\]
 is independent of $N$. 
This implies
\[ \sup_{N\geq 1} \EE\Big( \sum_{l=1}^N \big[ \| A u^l\|_{V^0}^2 + \|\AA \theta^l\|_{H^0}^2 \big]  + \|u^l-u^{l-1}\|_{V^0}^2 + \|\theta^l-\theta^{l-1}\|_{H^0}^2;
\Big) <\infty\]
this proves \eqref{mom_scheme_1} 
for $K=1$. 
For $s\in [t_j, t_{j+1})$, $j=0, ..., N-1$, set $\underline{s}=t_j$. The Davis inequality, and then the Cauchy-Schwarz and Young inequality imply for
any $\epsilon >0$
\begin{align}		\label{Davis-u^l-t^l}
\EE\Big( &\max_{1\leq L\leq N} \sum_{l=1}^L \big[ \big( G(u^{l-1})) \Delta_l W, u^{l-1}\big) + \big( \GG(\theta^{l-1} \Delta_l\WW, \theta^{l-1}\big) \big] \Big)
\nonumber \\
\leq & \; \EE\Big( \sup_{t\in [0,T]} \int_0^t \big( G(u^{\underline{s}}) dW(s) , u^{\underline{s}}\big) \Big) +
 \EE\Big( \sup_{t\in [0,T]} \int_0^t \big( \GG(\theta^{\underline{s}}) d\WW(s) , \theta^{\underline{s}}\big) \Big) 	\nonumber \\
 \leq & \; 3\EE\Big( \Big| \int_0^T \|G(u^{\underline{s}})\|_{{\mathcal L}(K;V^0)}^2 \|u^{\underline{s}}\|_{V^0}^2 {\rm Tr}(Q) ds \Big|^{\frac{1}{2}} \Big)
 +  3\EE\Big( \Big| \int_0^T \|\GG(\theta^{\underline{s}})\|_{{\mathcal L}(\tilde{K};H^0)}^2 \|\theta^{\underline{s}}\|_{H^0}^2 
 {\rm Tr}(\QQ) ds \Big|^{\frac{1}{2}} \Big)		\nonumber \\
 \leq &\;  3 {\rm Tr}(Q)^{\frac{1}{2}} \EE\Big[ \max_{1\leq l\leq N} \|u^l\|_{V^0} \Big(h\sum_{l=0}^N [K_0+K_1 \|u^{l-1}\|^2_{V^0} \big] \Big)^{\frac{1}{2}} \Big] 
 \nonumber \\
 &\quad + 3 {\rm Tr}(\QQ)^{\frac{1}{2}} \EE\Big[ \max_{1\leq l\leq N} \|\theta^l\|_{H^0} \Big(h\sum_{l=0}^N [\tilde{K}_0+\tilde{K}_1 
 \|\theta^{l-1}\|^2_{H^0} \big] \Big)^{\frac{1}{2}} \Big] 		\nonumber \\
 \leq & \; \epsilon \EE\Big( \max_{1\leq l\leq N} \|u^l\|_{V^0}^2\Big) + \EE\big( \|u^0\|_{V^0}^2 ) + \frac{9}{4\epsilon} {\rm Tr}(Q) \, h \sum_{l=1}^N
 \big[ K_0+K_1 \EE(\|u^{l-1}\|_{V^0}^2    )\big] 		\nonumber \\
& \quad  +   \epsilon \EE\Big( \max_{1\leq l\leq N} \|\theta^l\|_{H^0}^2\Big) + \EE\big( \|\theta ^0\|_{H^0}^2 ) 
+ \frac{9}{4\epsilon} {\rm Tr}(\QQ) \, h \sum_{l=1}^N \big[ \tilde{K}_0+\tilde{K}_1  \EE(\|\theta^{l-1} \|_{H^0}^2  )\big]
\end{align}
Taking the maximum over $L$ in \eqref{u^L-t^L} and using \eqref{Davis-u^l-t^l}, we deduce 
\begin{align*}
\EE\Big(& \max_{1\leq L\leq N} \big[ \|u^l\|_{V^0}^2 + \theta^l\|_{H^0}^2 \big] \Big) \leq 2\EE\big( \|u\|_{V^0}^2 + \| \theta^0\|_{H^0}^2 \big)
+ h \sum_{l=1}^N \big( \|\theta^{l-1}\|_{H^0}^2 + \|u^{l-1}\|_{V^0}^2  \big) \\
& + 2\epsilon \EE\Big( \max_{1\leq L\leq N} \big[ \|u^l\|_{V^0}^2 + \|\theta^L\|_{H^0}^2\big]\Big)
+  \frac{9}{4\epsilon}  {\rm Tr}(Q) \, h\, \sum_{l=1}^N \big[ K_0+K_1\EE( \|u^{l-1}\|_{V^0}^2 ) \big] \\
& +  \frac{9}{4\epsilon}  {\rm Tr}(\QQ) \, h\, \sum_{l=1}^N \big[ \tilde{K}_0+\tilde{K}_1\EE( \|\theta^{l-1}\|_{H^0}^2 ) \big]. 
\end{align*} 
For $\epsilon = \frac{1}{4}$,  \eqref{mom_K=1} proves
\[ \sup_{N\geq 1}\Big[  \EE\Big( \sup_{1\leq L\leq N} \|u^L\|_{V^0}^2\Big) + \EE\Big( \sup_{1\leq L\leq N} \|\theta^L\|_{H^0}^2 \Big)\Big] <\infty,\]
which proves \eqref{mom_scheme_0} for $K=1$.

We next prove \eqref{mom_scheme_0}--\eqref{mom_scheme_1} by induction on $K$. 

Multiply \eqref{dif_u} by $\|u^l\|_{V^0}^2$ and \eqref{dif_theta} by  $\|\theta^l\|_{H^0}^2$.  
Using the identity $a(a-b)= \frac{1}{2} \big[ a^2-b^2+|a-b|^2\big]$ for $a=\|u^l\|^2_{V^0}$ (resp.
$a=\|\theta^l\|_{H^0}^2$) and $b=\|u^{l-1}\|_{V^0}^2$ (resp. $b=\|\theta^{k-1}\|_{H^0}^2$), we deduce
for $l=1, ..., N$
\begin{align}		\label{u-t_K=2}
\frac{1}{4} \Big[& \|u^l\|_{V^0}^4 - \|u^{l-1}\|_{V^0}^4 + \big| \|u^l\|_{V^0}^2 - \|u^{l-1}\|_{V^0}^2 \big|^2 + 
\|\theta^l\|_{H^0}^4 - \|\theta^{l-1}\|_{H^0}^4 + \big| \|\theta^l\|_{H^0}^2 - \|\theta^{l-1}\|_{H^0}^2 \big|^2 \Big]		\nonumber \\
& + \frac{1}{2} \big[ \|u^l-u^{l-1}\|_{V_0}^2 \|u^l\|_{V^0}^2 + \| \theta^l-\theta^{l-1}\|_{H^0}^2 \|\theta^l\|_{H^0}^2 \big] + h \nu \|A^{\frac{1}{2}} u^l\|_{V^0}^2
\|u^l\|_{V^0}^2 		\nonumber \\
& + h \kappa \|\AA^{\frac{1}{2}} \theta^l\|_{H^0}^2 \|\theta^l\|_{H^0}^2  =  h \big( \Pi \theta^{l-1} v_2 , u^l\big) \|u^{l-1}\|_{V^0}^2 + \sum_{i=1}^4 T_i(l),
\end{align}
where
\begin{align*} 
 T_1(l)= & \big( G(u^{l-1}) \Delta_lW ,
u^{l-1}\big) \|u^l\|_{V^0}^2 , & 
T_2(l) = \big( G(u^{l-1})\Delta_l W, u^l-u^{l-1} \big) \|u^l\|_{V^0}^2 , \\
 T_3(l)= & \big( \GG(\theta^{l-1}) \Delta_l \WW , \theta^{l-1}\big) \|\theta^{l-1} \|_{H^0}^2, &	
T_4(l)=   \big( \GG(\theta^{l-1})\Delta_l \WW, \theta^l-\theta^{l-1} \big) \|\theta^l \|_{H^0}^2.
\end{align*} 
The Cauchy-Schwarz and Young inequalities imply
\begin{equation}	\label{T0l}
\big( \Pi \theta^{l-1} v_2, u^l\big) \|u^l\|_{V^0}^2 \leq \|\theta^{l-1}\|_{H^0} \|u^l\|_{V^0}^3 \leq \frac{1}{4} \|\theta^{l-1}\|_{H^0}^4 
+ \frac{3}{4} \|u^l\|_{V^0}^4.
\end{equation}
The Cauchy-Schwarz and Young inequalities imply for $\epsilon, \bar{\epsilon}>0$, 
\begin{align}		\label{T2l}
 |T_2(l)| &\leq  \|G(u^{l-1})\|_{{\mathcal L}(K;V^0)} \|u^l-u^{l-1}\|_{V^0}^2 \|u^l\|_{V^0}^2 	\nonumber \\
 & \leq \epsilon \|u^l-u^{l-1}\|_{V^0}^2 \|u^l\|_{V^0}^2 \nonumber \\
 &\quad + \frac{1}{4\epsilon} \|G(u^{l-1})\|_{{\mathcal L}(K;V^0)}^2 \|\Delta_l W\|_K^2 
\big[ \|u^{l-1}\|_{V^0}^2 + \big( \|u^l\|_{V^0}^2 - \|u^{l-1}\|_{V^0}^2 \big) \big]		\nonumber \\
&\leq \epsilon \|u^l-u^{l-1}\|_{V^0}^2 \|u^l\|_{V^0}^2 + \frac{1}{4\epsilon} \|G(u^{l-1})\|_{{\mathcal L}(K;V^0)}^2 \|\Delta_l W\|_K^2 
 \|u^{l-1}\|_{V^0}^2 		\nonumber \\
& \quad + \bar{\epsilon} \big| \|u^l\|_{V^0}^2 - \|u^{l-1}\|_{V^0}^2 \big|^2 + \frac{1}{16 \epsilon^2} \frac{1}{4\bar{\epsilon}} \|G(u^{l-1})\|_{{\mathcal L}(K;V^0)}^4
 \|\Delta_l W\|_K^4. 
\end{align}
A similar argument proves for $\epsilon, \bar{\epsilon}>0$ 
\begin{align}		\label{T4l}
 |T_4(l)| \leq  &\,  \epsilon \|\theta^l-\theta^{l-1}\|_{H^0}^2 \|\theta^l\|_{H^0}^2 + \frac{1}{4\epsilon} \|\GG(\theta^{l-1})\|_{{\mathcal L}(\tilde{K};H^0)}^2
 \|\Delta_l \WW\|_{\tilde{K}}^2   \|\theta^{l-1}\|_{H^0}^2 		\nonumber \\
& \quad + \bar{\epsilon} \big| \|\theta^l\|_{H^0}^2 - \|\theta^{l-1}\|_{H^0}^2 \big|^2 + \frac{1}{16 \epsilon^2} \frac{1}{4\bar{\epsilon}} 
\|\GG(\theta^{l-1})\|_{{\mathcal L}(\tilde{K};H^0)}^4
 \|\Delta_l \WW\|_{\tilde{K}}^4. 
\end{align}
A similar argument shows for $\bar{\epsilon}>0$ 
\begin{align}		\label{T1l}
|T_1(l)|\leq & \, \| G(u^{l-1})\Delta_l W\|_{V^0}   \|u^{l-1}\|_{V^0}^3 + \| G(u^{l-1})\Delta_l W\|_{V^0} \|u^{l-1}\|_{V^0} 
\big[ \|u^l\|_{V^0}^2 - \|u^{l-1}\|_{V^0}^2 \big]  		\nonumber \\
\leq & \frac{1}{4} \|G(u^{l-1})\|_{{\mathcal L}(K;V^0)}^4 \|\Delta_l W\|_K^4 + \frac{3}{4} \|u^{l-1}\|_{V^0}^4 
+ \bar{\epsilon} \big| \|u^l\|_{V^0}^2 - \|u^{l-1}\|_{V^0}\|^2 \big|^2 		\nonumber \\
&\quad + \frac{1}{4\bar{\epsilon}} \|G(u^{l-1})\|_{{\mathcal L}(K;V^0)}^2 \|\Delta_l W\|_K^2
\|u^{l-1}\|_{V^0}^2, 
\end{align} 
and
\begin{align}		\label{T3l}
|T_3(l)|\leq 
 & \frac{1}{4} \|\GG(\theta^{l-1})\|_{{\mathcal L}(\tilde{K}; H^0)}^4 \|\Delta_l \WW\|_K^4 + \frac{3}{4} \|\theta^{l-1}\|_{H^0}^4 
+ \bar{\epsilon} \big| \|\theta^l\|_{H^0}^2 - \|\theta^{l-1}\|_{H^0}\|^2 \big|^2 		\nonumber \\
&\quad + \frac{1}{4\bar{\epsilon}} \|\GG(\theta^{l-1})\|_{{\mathcal L}(\tilde{K};H^0)}^2 \|\Delta_l \WW\|_{\tilde{K}}^2
\|\theta^{l-1}\|_{H^0}^2.
\end{align} 
Add the inequalities \eqref{u-t_K=2}--\eqref{T3l} for $l=1$ to $L\leq N$,  choose $\epsilon=\frac{1}{4}$ and $\bar{\epsilon}=\frac{1}{16}$,  and
use the growth conditions \eqref{growthG-0} and \eqref{growthGG-0}. This yields
\begin{align}		\label{u^L-t^L-4}
\|u^L&\|_{V^0}^4 + \|\theta^L\|_{H^0}^4 + \frac{1}{2} \sum_{l=1}^L \big[  \big| \| u^l\|_{V^0}^2 -\|u^{l-1}\|_{V^0}^2\big|^2 +
 \|u^l-u^{l-1}\|_{V^0}^2 \|u^l\|_{V^0}^2  + \big| \|\theta^l\|_{H^0}^2 - \theta^{l-1}\|_{H^0}^2\big|^2 	\nonumber  \\
 &\; + \|\theta^l-\theta^{l-1} \|_{H^0}^2 \|\theta^l\|_{H^0}^2 \big] 
 +4 h\sum_{l=1}^L \big[ \nu \|A^{\frac{1}{2}} u^l\|_{V^0}^2 \|u^l\|_{V^0}^2 + \kappa \| \AA^{\frac{1}{2}} \theta^l\|_{H^0}^2 \|\theta^l\|_{H^0}^2 \big]  \nonumber \\
 \leq & \|u_0\|_{V^0}^4 + \|\theta_0\|_{H^0}^4 + \frac{1}{4} h \sum_{l=0}^{L-1} \|\theta^l\|_{H^0}^4 + \frac{3}{4} h  \sum_{l=1}^L \|u^l\|_{V^0}^4 \nonumber \\
 &\;
 + C\sum_{l=0}^L \Big(  \big[ K_0+K_1 \|u^{l-1}\|_{V^0}^2\big] \|u^{l-1}\|_{V^0}^2  \|\Delta_l W\|_K^2 +  \big[ K_0+K_1 \|u^{l-1}\|_{V^0}^2\big]^2  \|\Delta_l W\|_K^4
 \Big) 	\nonumber \\
 &\;  + C \sum_{l=0}^L \Big(  \big[ \tilde{K}_0+\tilde{K}_1 \|\theta ^{l-1}\|_{H^0}^2\big] \|\theta^{l-1}\|_{H^0}^2  \|\Delta_l \WW\|_{\tilde{K}}^2 
 +  \big[ \tilde{K}_0+\tilde{K}_1 \|\theta^{l-1}\|_{H^0}^2\big]^2  \|\Delta_l \WW\|_{\tilde{K}}^4
 \Big) . 
\end{align} 
Taking expected values, we deduce for every $L=1, ..., N$  and $h=\frac{T}{N}\leq 1$
\begin{align*}
\EE\Big( &\|u^L\|_{V^0}^4 + \|\theta^L\|_{H^0}^4 + \frac{1}{2} \sum_{l=1}^L \big[  \big| \| u^l\|_{V^0}^2 -\|u^{l-1}\|_{V^0}^2\big|^2 +
 \|u^l-u^{l-1}\|_{V^0}^2 \|u^l\|_{V^0}^2 \\ 
 &\quad  + \big| \|\theta^l\|_{H^0}^2 - \theta^{l-1}\|_{H^0}^2\big|^2 
 +  \|\theta^l-\theta^{l-1} \|_{H^0}^2 \|\theta^l\|_{H^0}^2 \big] \Big)\\
 &\quad  + \EE\Big( 
 4 h\sum_{l=1}^L \big[ \nu \|A^{\frac{1}{2}} u^l\|_{V^0}^2 \|u^l\|_{V^0}^2 + \kappa \| \AA^{\frac{1}{2}} \theta^l\|_{H^0}^2 \|\theta^l\|_{H^0}^2 \big] \Big) \\
 \leq & \; \EE(\|u_0\|_{V^0}^4 + \|\theta_0\|_{H^0}^4\big) +3 \,  h\,  \EE(\|u^L\|_{V^0}^4) + C+ C\, h\, \sum_{l=0}^{L-1} \big( \|u^l\|_{V^0}^4 +
  \|\theta^l\|_{H^0}^4\big)
\end{align*}
for some constant $C$ depending on $K_i, \tilde{K}_i, {\rm Tr}(Q), {\rm Tr}(\QQ)$ and $T$, and which does not depend on $N$. 
Let $N$ be large enough to have $3\, h<\frac{1}{2}$. Neglecting the sums in the left hand side and using the discrete Gronwall lemma, we deduce
for $\EE\big(\|u_0\|_{V^0}^4 + \|\theta_0\|_{H^0}^4\big) <\infty $
\begin{equation}		\label{max_Eut-4}
\sup_{N\geq 1} \max_{0\leq L\leq N}  \EE\Big( \|u^L\|_{V^0}^4 + \|\theta^L\|_{H^0}^4 \Big) <\infty, 
\end{equation}
which implies
\begin{equation}		\label{max_sum_ut-4}
\sup_{N\geq 1} \EE\Big( h\sum_{l=1}^N \big[ \| A^{\frac{1}{2}} u^l\|_{V^0}^2 \|u^l\|_{V^0}^2 + \|\AA^{\frac{1}{2}} \theta^l\|_{H_0}^2 \|\theta^l\|_{H^0}^2
\big] \Big) <\infty,
\end{equation}
that is \eqref{mom_scheme_1} for $K=2$. The argument used to prove \eqref{Davis-u^l-t^l} implies 
\begin{align*}
\EE\Big( \max_{1\leq L\leq N}& \sum_{l=1}^L \big( G(u^{l-1}) \Delta_l W , u^{l-1}\big) \|u^{l-1}\|_{V^0}^2 \Big) \\
 \leq &
\; \epsilon \EE\Big( \max_{1\leq L\leq N} \|u^L\|_{V^0}^4 \Big) + C(\epsilon) \Big[ 1+\max_{1\leq L\leq N} \EE(\|u^L\|_{V^0}^4) \Big]
\end{align*} 
and
\begin{align*}
\EE\Big( \max_{1\leq L\leq N}& \sum_{l=1}^L \big( \GG(\theta^{l-1}) \Delta_l \WW , \theta^{l-1}\big) \|u\theta^{l-1}\|_{H^0}^2 \Big) \\
 \leq &
\; \epsilon \EE\Big( \max_{1\leq L\leq N} \|\theta^L\|_{H^0}^4 \Big) + C(\epsilon) \Big[ 1+\max_{1\leq L\leq N} \EE(\|\theta^L\|_{H^0}^4) \Big]
\end{align*} 
Taking the maximum for $L=1, ...,N$ and using  \eqref{max_Eut-4}, we deduce \eqref{mom_scheme_0} for $K=2$. 
The details of the induction step, similar to the proof in the case $K=2$, are left to the reader.
\end{proof}

\section{Strong convergence of the localized implicit time Euler scheme}		\label{s-loc-converg}
Due to the bilinear terms $[u.\nabla]u$ and $[u.\nabla]\theta$, we first prove an $L^2(\Omega)$-convergence of the $\LL^2(D)$-norm of
the error,  uniformly on the time grid,  restricted to the set $\Omega_M(N)$ defined below for some $M>0$
\begin{equation} 		\label{def-A(M)}
\Omega_M(j):= \Big\{ \sup_{s\in [0,t_j]} \|A^{\frac{1}{2}} u(s)\|_{V^0}^2 \leq M\Big\} \cap \Big\{ \sup_{s\in [0,t_j]} \|\AA^{\frac{1}{2}} \theta(s)\|_{H^0}^2 \leq M\Big\},
\quad \forall j=0, ..., N,
\end{equation} 
and let $\Omega_M:= \Omega_M(N)$. 
Recall that for $j=0, ...,N$ set $e_j:= u(t_j)-u^j$ and $\tilde{e}_j:= \theta(t_j)-\theta^j$; then $e_0=\tilde{e}_0=0$. 
Using \eqref{def_u}, \eqref{def_t}, \eqref{Euler_u} and \eqref{Euler_t} we deduce for $j=1, ...,N$, $\phi\in V^1$ and $\psi\in H^1$
\begin{align}	\label{def_error_u}
\big( &e_j-e_{j-1}\, , \, \varphi \big) + \nu \!\int_{t_{j-1}}^{t_j} \! \!\big(  A^{\frac{1}{2}}  [ u(s) -   u^j ] ,  A^{\frac{1}{2}}  \varphi\big)  ds 
+ \int_{t_{j-1}}^{t_j} \!\! \big\langle B(u(s),u(s)) - B(u^j,u^j) ,  \varphi\big\rangle ds \nonumber \\
& =  \int_{t_j}^{t_{j+1}} \big( \Pi [\theta(s)-\theta^{j-1} ] v_2, \varphi\big) ds + 
 \int_{t_{j-1}}^{t_j} \big( [G(u(s))-G(u^{j-1}) ] dW(s) \, , \, \varphi\big),  
\end{align}
and
\begin{align}		\label{def_error_t}
\big( &\ee_j-\ee_{j-1}\, , \, \psi \big) + \kappa \!\int_{t_{j-1}}^{t_j} \! \! \big( \AA^{\frac{1}{2}}  [ \theta(s) -   \theta^j ] ,  \AA^{\frac{1}{2}}  \psi\big)  ds 
+ \int_{t_{j-1}}^{t_j} \!\! \big\langle [u(s).\nabla]\theta(s)  - [u^.\nabla]\theta^j] ,  \psi\big\rangle ds \nonumber \\
& =   \int_{t_{j-1}}^{t_j} \big( [\GG(\theta(s))-\GG(\theta^{j-1}) ] d\WW(s) \, , \, \psi\big).  
\end{align}
In this section, we will suppose that $N$ is large enough to have $h:=\frac{T}{N} \in (0,1)$. The following result is  a crucial step towards
the rate of convergence of the implicit time Euler scheme. 
\begin{prop}	\label{th_loc_cv}
Suppose that  the  conditions {\bf (C-u)} and {\bf (C-$\theta$)} hold. 
Let  $u_0\in L^{32+\epsilon}(\Omega ; V^1)$ and $\theta_0\in L^{32+\epsilon}(\Omega;H^1)$ for some $\epsilon>0$,
  $u,\theta$ be the solution to \eqref{def_u}--\eqref{def_t},
$\{u^j, \theta^j\}_{j=0, ..., N}$ be the solution to \eqref{Euler_u}--\eqref{Euler_t}. Fix $M>0$ and let $\Omega_M=\Omega_M(N)$ 
be defined by \eqref{def-A(M)}. Then 
 for $\eta\in (0,1)$, there exists a  positive constant  $C$,
 independent of $N$, such that  for $N$ large enough 
\begin{align} 		\label{loc_cv}
\EE\Big( &1_{\Omega_M} \Big[ \max_{1\leq j\leq N} \big( \|u(t_j)-u^j\|_{V^0}^2 + \|\theta(t_j)-\theta^j\|_{H^0}^2 \big)  + \frac{T}{N} \sum_{j=1}^N 
\big[ \|A^{\frac{1}{2}} [ u(t_j) - u^j] \|_{V^0}^2 \nonumber \\
&\quad + \|\AA^{\frac{1}{2}} [ \theta(t_j) - \theta^j] \|_{H^0}^2  \Big] \Big) 
\leq  C (1+M)   e^{{\mathcal C}(M) T} \Big( \frac{T}{N}\Big)^{\eta},
\end{align}
where 
\[  {\mathcal C}(M) =  \frac{9(1+\gamma) \bar{C}_4^2}{8}  \max\Big(\frac{5}{ \nu}  \, , 
 \, \frac{1}{\kappa}\Big) \, M\]  
for some $\gamma >0$, and $\bar{C}_4$ is the constant in the 
right hand side of the Gagliardo-Nirenberg inequality \eqref{GagNir}. 
\end{prop}
\begin{proof} 
Write \eqref{def_error_u} with $\varphi = e_j$ and \eqref{def_error_t} with $\psi=\theta^j$; 
using the equality $(f,f-g)=\frac{1}{2}\big[ \|f\|_{\LL^2}^2 - \|g\|_{\LL^2}^2 + \|f-g\|_{\LL^2}^2\big]$,  we have for $j=1, ..., N$
\begin{align}	\label{increm_error_u}
\frac{1}{2} \big( \|e_j\|_{V^0}^2 - \|e_{j-1}\|_{V^0}^2\big) &+ \frac{1}{2} \|e_j-e_{j-1}\|_{V^0}^2 + \nu h \|A^{\frac{1}{2}}  e_j\|_{V^0}^2  
 \leq \sum_{l=1}^7 T_{j,l},  \\
\frac{1}{2} \big( \|\ee_j\|_{H^0}^2 - \|\ee_{j-1}\|_{H^0}^2\big) &+ \frac{1}{2} \|\ee_j-\ee_{j-1}\|_{H^0}^2 + 
\kappa h \|\AA^{\frac{1}{2}}  \ee_j\|_{H^0}^2   \leq \sum_{l=1}^6 \tilde{T}_{j,l},		\label{incrfem_error_t}
\end{align}
where by the antisymmetry property \eqref{B} we have
\begin{align*}
T_{j,1}=&-\int_{t_{j-1}}^{t_j} \!\! \big\langle B\big(e_j, u^j\big) \, , \, e_j\big\rangle ds = 
-\int _{t_{j-1}}^{t_j} \!\! \big\langle B\big(e_j, u(t_j)\big) \, , \, e_j\big\rangle ds , \\
T_{j,2}=&-\int_{t_{j-1}}^{t_j}\!\!  \big\langle B\big(u(s)-u(t_j)\, , u(t_j)\big) \, e_j\big\rangle ds, \\
T_{j,3}=&-  \int_{t_{j-1}}^{t_j}\!\!  \big\langle B\big( u(s), u(s)-u(t_j)\big) \, , \, e_j\big\rangle ds = 
\int_{t_{j-1}}^{t_j} \!\! \big\langle B\big( u(s), e_j\big)  \, , \, u(s)-u(t_j) \big\rangle ds, \\
T_{j,4}=&-\nu \!\int_{t_{j-1}}^{t_j}\!\! \! \big( A^{\frac{1}{2}}  ( u(s)-u(t_j)) ,  A^{\frac{1}{2}}  e_j\big) ds, \quad 
T_{j,5}= \int_{t_{j-1}}^{t_j} \!\! \big( \Pi [\theta(s) - \theta^{j-1}] v_2 \, , \, e_j\big) ds, \\
T_{j,6}=& \int_{t_{j-1}}^{t_j}\!\! \! \big([G(u(s))-G(u^{j-1})\big] dW(s)\, ,\,  e_j-e_{j-1}\big), \\
T_{j,7}=& \int_{t_{j-1}}^{t_j}\!\!  \big([G(u(s))-G(u^{j-1}) \big] dW(s), e_{j-1}\big),
\end{align*}
and
\begin{align*}
\tilde{T}_{j,1}=&-\int_{t_{j-1}}^{t_j} \!\! \big\langle [e_{j-1}.\nabla] \theta^j  \, , \, \ee_j\big\rangle ds = 
-\int _{t_{j-1}}^{t_j} \!\! \big\langle  [e_{j-1}.\nabla] \theta(t_j) \, , \, \ee_j\big\rangle ds , \\
\tilde{T}_{j,2}=&-\int_{t_{j-1}}^{t_j}\!\!  \big\langle [(u(s)-u(t_{j-1}).\nabla ] \theta(t_j)   \, ,\, \ee_j\big\rangle ds,\\
\tilde{T}_{j,3}=&-  \int_{t_{j-1}}^{t_j}\!\!  \big\langle [u(s).\nabla] (\theta(s)-\theta(t_j)  \, , \, \ee_j\big\rangle ds = 
\int_{t_{j-1}}^{t_j} \!\! \big\langle  [u(s).\nabla] \ee_j \, , \, (\theta(s)-\theta(t_j)  \big\rangle ds, \\
\tilde{T}_{j,4}=&-\nu \!\int_{t_{j-1}}^{t_j}\!\! \! \big( \AA^{\frac{1}{2}}  ( \theta(s)-\theta(t_j)) ,  \AA^{\frac{1}{2}}  \ee_j\big) ds, \\
\tilde{T}_{j,5}=& \int_{t_{j-1}}^{t_j}\!\! \! \big([\GG(\theta(s))-\GG(\theta^{j-1}) d\WW(s)\, ,\,  \ee_j-\ee_{j-1}\big), \\
\tilde{T}_{j,6}=& \int_{t_{j-1}}^{t_j}\!\!  \big([G(u(s))-G(u^{j-1}) \big] dW(s), e_{j-1}\big),
\end{align*}
We next prove upper estimates of  the terms $T_{j,l}$ for $l=1, ...,5$, $\tilde{T}_{j,l}$ for $l=1, ..., 4$, and of the expected value of $T_{j,6}$, $T_{j,7}$
$\tilde{T}_{j,5}$ and $\tilde{T}_{j,6}$.

The H\"older and Young inequalities and the Gagliardo Nirenberg inequality \eqref{GagNir},
imply  for $\delta_1>0$
 \begin{align}		\label{maj_Tj1}
 |T_{j,1}|\leq & \, \bar{C}_4 \, h\, \|e_j\|_{V^0} \|A^{\frac{1}{2}} e_j\|_{V^0} \|A^{\frac{1}{2}} u(t_j)\|_{V^0}  \nonumber \\
 \leq & \, \delta_1\, \nu\,  h\,  \|A^{\frac{1}{2}} e_j\|_{V^0}^2 +\frac{ \bar{C}_4^2}{4\delta_1 \nu} \, h\,  \|A^{\frac{1}{2}} u(t_j)\|_{V^0}^2 \|e_j\|_{V^0}^2, 
  \end{align}
  and for $\tilde{\delta}_1, \delta_2>0$, 
  \begin{align}		\label{maj_TTj1}
 |\tilde{T}_{j,1}|\leq & \, \bar{C}_4  \, h\, \|A^{\frac{1}{2}} e_{j-1}\|_{V^0}^{\frac{1}{2}}  \| e_{j-1} \|_{V^0}^{\frac{1}{2}} 
  \| \AA^{\frac{1}{2}} \ee_j\|_{H^0}^{\frac{1}{2}}  \| \ee_j\|_{H^0}^{\frac{1}{2}}   \|\AA^{\frac{1}{2}} \theta(t_j)\|_{H^0}  \nonumber \\
 \leq & \, \delta_2\, \nu h \|A^{\frac{1}{2}} e_{j-1}\|_{V^0}^2 + \tilde{\delta}_1 h \kappa \|\AA^{\frac{1}{2}} \ee_j\|_{H^0}^2  \nonumber \\
 &\quad + 
 \frac{ \bar{C}_4^2}{16\delta_2 \nu} \, h\,  \|\AA^{\frac{1}{2}} \theta(t_j)\|_{H^0}^2 \|e_{j-1}\|_{V^0}^2
 + \frac{ \bar{C}_4^2}{16\tilde{\delta}_1 \kappa}\,  h\,  \|\AA^{\frac{1}{2}} \theta(t_j)\|_{H^0}^2 \|\ee_{j}\|_{H^0}^2.
  \end{align}
 H\"older's inequality and  the Sobolev embedding $V^1\subset \LL^4$ imply for $\delta_3 >0$
 \begin{align}		\label{maj_Tj2}
 | T_{j,2}|& \leq  C \int_{t_{j-1}}^{t_j} \|u(s)-u(t_j)\|_{V^1} \|A^{\frac{1}{2}} u(t_j)\|_{V^0} \|A^{\frac{1}{2}} e_j\|_{V^0}^{\frac{1}{2}} \|e_j\|_{V^0}^{\frac{1}{2}} ds
 \nonumber \\
 \leq & \; \delta_3 \nu \, h\, \|A^{\frac{1}{2}} e_j\|_{V^0}^2 +  h\, \|e_j\|_{V^0}^2 + \frac{C}{\sqrt{\nu \delta_3 } }
  \| A^{\frac{1}{2}} u(t_j)\|_{V^0}^2 \!  \int_{t_{j-1}}^{t_j}\!\!  \|u(t_j)-u(s)\|_{V^1}^2 ds,
 \end{align} 
 while for $\tilde{\delta}_2 >0$,  
 \begin{align}		\label{maj_TTj2} 
 | \tilde{T}_{j,2}|& \leq  
\; \tilde{\delta}_2 \kappa \, h\, \|\AA^{\frac{1}{2}} \ee_j\|_{H^0}^2 + \, h\, \|\ee_j\|_{H^0}^2 + \frac{C}{\kappa \tilde{\delta}_2 }
\int_{t_{j-1}}^{t_j}  \|A^{\frac{1}{2}}\big[ u(s)-u(t_{j-1})\big]\|_{V^0}^2 ds  \nonumber \\
&\quad +   \frac{C}{\kappa \tilde{\delta}_2 } \| \AA^{\frac{1}{2}} \theta(t_j)\|_{H^0}^4 \!  \int_{t_{j-1}}^{t_j}\!\!  \|u(s)-u(t_{j-1})\|_{V^0}^2 ds.
 \end{align} 
Similar arguments prove for $\delta_4, \tilde{\delta}_3>0$
\begin{align} 		\label{maj_Tj3}
|T_{j,3}|&\leq \delta_4 \nu h \|A^{\frac{1}{2}} e_j\|_{V^0}^2 + \frac{C}{\nu \delta_4} \sup_{s\in [0,T]} \|u(s)\|_{V^1}^2\int_{t_{j-1}}^{t_j} 
\|u(s)-u(t_j)\|_{V^1}^2 ds,\\
|\tilde{T}_{j,3}| & \leq \tilde{\delta}_3\, \kappa \, h\, \|\AA^{\frac{1}{2}} \ee_j\|_{H^0}^2 + \frac{C}{\kappa \tilde{\delta}_3} 
\sup_{s\in [0,T]} \|u(s)\|_{V^1}^2 \int_{t_{j-1}}^{t_j}  \|\theta(s)-\theta(t_j)\|_{H^1}^2 ds.			\label{maj_TTj3} 
\end{align} 
 The Cauchy-Schwarz  and Young inequalities imply for $\delta_5, \tilde{\delta}_4>0$
 \begin{align} 		\label{maj_Tj4}
|T_{j,4}|&\leq  \delta_5 \nu \, h\, \|A^{\frac{1}{2}} e_j\|_{V^0}^2 + \frac{\nu}{4\delta_5} \int_{t_{j-1}}^{t_j} \| A^{\frac{1}{2}} [ u(s)-u(t_j)]\|_{V^0}^2 ds,
 \\
|\tilde{T}_{j,4}|&\leq  \tilde{\delta}_4 \, \kappa \, h\, \|\AA^{\frac{1}{2}} \ee_j\|_{H^0}^2 + \frac{\kappa}{4 \tilde{\delta}_4}
 \int_{t_{j-1}}^{t_j} \| \AA^{\frac{1}{2}} [ \theta(s)-\theta(t_j)]\|_{H^0}^2 ds.
\label{maj_TTj4} 
 \end{align}
 Using once more the Cauchy-Schwarz and Young inequalities, we deduce
 \begin{align}		\label{maj_Tj5}
 |T_{j,5} | &\leq \int_{t_{j-1}}^{t_j} \big[ \|\theta(s)-\theta(t_{j-1}) \|_{H^0} + \| \ee_{j-1}\|_{H^0}\big] \, \|e_j\|_{V^0} ds  \nonumber \\
 &\leq \frac{h}{2}  \, \|e_j\|_{V^0}^2 + \frac{h}{2} \, \| \ee_{j-1}\|_{H^0}^2 + \frac{1}{2}  \int_{t_{j-1}}^{t_j}   \|\theta(s)-\theta(t_{j-1})\|_{H^0}^2  ds.
 \end{align}
 Note that the sequence of subsets $\{ \Omega_M(j)\}_{0\leq j\leq N}$ is decreasing. Therefore, since $e_0=\ee_0=0$, given $L=1, ..., N$, 
 we deduce 
 \begin{align*}
 \max_{1\leq J\leq L} &\sum_{j=1}^J   1_{\Omega_M({j-1})}  \big[ \|e_j\|_{V^0}^2 - \| e_{j-1}\|_{V^0}^2 + \|\ee_j\|_{H^0}^2  - \| \ee_{j-1}\|_{H^0}^2 \big]  \\
&\;  = \max_{1\leq J\leq L}  \sum_{j=2}^L \Big( 1_{\Omega_M({J-1})} \big[ \|e_J\|_{V^0}^2 + \|\ee_j\|_{H^0}^2 \big] \Big) \\
&\qquad  + \sum_{j=2}^L \big( 1_{\Omega_M(j-2)} - 1_{\Omega_M({j-1})}\big) \big[ \|e_{j-1}\|_{V^0}^2 + \|\ee_{j-1} \|_{H^0}^2 \big] \\
&\; \geq \max_{1\leq J\leq L}  \sum_{j=2}^L \Big( 1_{\Omega_M({J-1})} \big[ \|e_J\|_{V^0}^2 + \|\ee_j\|_{H^0}^2 \big] \Big).
 \end{align*} 
 Hence, for $\sum_{j=1}^5 \delta_j \leq \frac{1}{3}$ and $\sum_{j=1}^4 \tilde{\delta}_j<\frac{1}{3}$, using Young's inequality, we deduce
 for every $\alpha >0$
 \begin{align}		\label{maj_error-1}
 \frac{1}{6}& \max_{1\leq J\leq L}  \Big( 1_{\Omega_M({J-1})} \big[ \|e_J\|_{V^0}^2 + \|\ee_j\|_{H^0}^2 \big] \Big) +
 \frac{1}{6} \sum_{j=1}^L 1_{\Omega_M(j-1)} \big( \|e_j - e_{j-1}\|_{V^0}^2 + \|\ee_j-\ee_{j-1}\|_{H^0}^2 \big) 		\nonumber \\
 &\quad +  \sum_{j=1}^L 1_{\Omega_M(j-1)}\, h\,  \Big[ \nu \Big( \frac{1}{3}-\sum_{i=1}^5 \delta_i \Big) \|A^{\frac{1}{2}} e_j\|_{V^0}^2 
 + \kappa \Big( \frac{1}{3}-\sum_{i=1}^4 \tilde{\delta}_i \Big) \|\AA^{\frac{1}{2}} \ee_j\|_{H^0}^2  \Big]  		\nonumber \\
 & \leq \, h \sum_{j=1}^L 1_{\Omega_M(j-1)} \|e_j\|_{V^0}^2 \Big( \frac{(1+\alpha) \bar{C}_4^2 }{4\delta_1 \nu} \|A^{\frac{1}{2}} u(t_{j-1})\|_{V^0}^2 
  + \frac{(1+\alpha)\bar{C}_4^2}{16 \delta_2 \nu}  \| \AA^{\frac{1}{2}} \theta(t_{j-1})\|_{H^0}^2 +\frac{3}{2} \Big)		\nonumber \\
 &\qquad + h \sum_{j=1}^L 1_{\Omega_M(j-1)} \| \ee_j\|_{H^0}^2 \Big( \frac{(1+\alpha) \bar{C}_4^2}{16 \tilde{\delta}_1 \kappa} 
 \| \AA^{\frac{1}{2}} \theta(t_{j-1})\|_{H^0}^2 +\frac{3}{2}  \Big)  + Z_L 	\nonumber \\
 &\qquad + \max_{1\leq J\leq L} \sum_{j=1}^J 1_{\Omega_M(j-1)} \big[ T_{j,6}+\tilde{T}_{j,5}\big]  
 + \max_{1\leq J\leq L} \sum_{j=1}^J 1_{\Omega_M(j-1)} \big[ T_{j,7}+\tilde{T}_{j,6}\big],
 \end{align} 
 where 
 \begin{align}	\label{def_ZL}
 Z_L=&\; C\, h\, \sum_{j=1}^L \|e_j\|_{V^0}^2 \big( \|A^{\frac{1}{2}} [ u(t_j)-u(t_{j-1})] \|_{V^0}^2 + 
  \|\AA^{\frac{1}{2}} [ \theta(t_j)-\theta(t_{j-1})] \|_{H^0}^2 \big)	\nonumber \\
 &\quad + C\, h\, \sum_{j=1}^L \|\ee_j\|_{V^0}^2 \|\AA^{\frac{1}{2}} [ \theta(t_j)-\theta(t_{j-1})] \|_{H^0}^2  \nonumber \\
  &\quad + C \sum_{j=1}^L 
 \Big( \sup_{s\in [0,T]} \|u(s)\|_{V^1}^2 +1\Big) \int_{t_{j-1}}^{t_j}\big[  \|u(t_j)-u(s)\|_{V^1}^2 + \|u(s)-u(t_{j-1}\|_{V^1}^2\big] ds \nonumber \\
 &\quad +  C \sum_{j=1}^L 
  \|\AA^{\frac{1}{2}}\theta(t_{j-1})\|_{H^0}^4 
 \int_{t_{j-1}}^{t_j} \|u(s)-u(t_{j-1})\|_{V^0}^2 ds \nonumber  \\
 &\quad + C  \sum_{j=1}^L 
 \Big( \sup_{s\in [0,T]} \|u(s)\|_{V^1}^2 +1  \Big) 
 \int_{t_{j-1}}^{t_j} \|\theta(s)-\theta(t_j)\|_{H^1}^2 ds . 
 \end{align} 
 The Cauchy-Schwarz and Young inequalities imply 
 \begin{align}		\label{maj_Tj6}
 \sum_{j=1}^L &1_{\Omega_M(j-1)} |T_{j,6}|  \leq \frac{1}{6} \sum_{j=1}^L 1_{\Omega_M(j-1)}  \|e_j-e_{j-1}\|_{V^0}^2 \nonumber \\
 &+ \frac{3}{2} \sum_{j=1}^L 1_{\Omega_M(j-1)} \Big\|\int_{t_{j-1}}^{t_j} \big[ G(u(s))-G(u^{j-1})\big] dW(s)\Big\|_{V^0}^2 \\
  \sum_{j=1}^L &1_{\Omega_M(j-1)} |\tilde{T}_{j,5}|  \leq \frac{1}{6} \sum_{j=1}^L 1_{\Omega_M(j-1)}  \|\ee_j-\ee_{j-1}\|_{H^0}^2 \nonumber \\
 &+  \frac{3}{2}  \sum_{j=1}^L 1_{\Omega_M(j-1)} \Big\|\int_{t_{j-1}}^{t_j} \big[ \GG(\theta(s))-\GG(\theta^{j-1})\big] d\WW(s)\Big\|_{H^0}^2. 
 	\label{maj_TTj5}
 \end{align}
 Using the upper estimates \eqref{maj_error-1}--\eqref{maj_TTj5}, taking expected values, using the Cauchy-Schwarz and Young inequalities,
  and the inequalities \eqref{mom_u_L2}, \eqref{mom_t_L2},  \eqref{sup_E_t}, \eqref{mom_scheme_0},  \eqref{increm_u_L2}, 
 \eqref{mom_increm_u_1} and \eqref{mom_increm_t_1}, we deduce that for $\eta \in (0,1)$ and  every $L=1, ..., N$
 \begin{align}		\label{mom_ZL}
 \EE(&Z_L) \leq  \, C\Big\{ \EE\Big( \sup_{s\in [0,T]} \|u(s)\|_{V^0}^4 + \max_{0\leq j\leq N} \|u^j\|_{V^0}^4\Big) \Big\}^{\frac{1}{2}} \nonumber \\
 &\qquad \times 
 \Big\{ \EE \Big( \Big| h \sum_{j=1}^N \big( \|A^{\frac{1}{2}} \big[ u(t_j)-u(t_{j-1})\big] \|_{V^0}^2 
 + \|\AA^{\frac{1}{2}} \big[ \theta(t_j)-\theta(t_{j-1})\big] \|_{V^0}^2\big) \Big|^2 \Big) \Big\}^{\frac{1}{2}} 		\nonumber \\
 &\, + C\Big\{ \EE\Big( \sup_{s\in [0,T]} \|\theta(s)\|_{H^0}^4 + \max_{0\leq j\leq N} \|\theta^j\|_{H^0}^4\Big) \Big\}^{\frac{1}{2}} 
  \Big\{ \EE \Big( \Big| h \sum_{j=1}^N \|\AA^{\frac{1}{2}} \big[ \theta(t_j)-\theta(t_{j-1})\big] \|_{V^0}^2\big) \Big|^2 \Big\}^{\frac{1}{2}} 	
 \nonumber \\
&\, +C \Big\{ \EE\Big(1+  \sup_{0\leq s\leq T} \|u(s)\|_{V^1}^4 \Big) \Big\}^{\frac{1}{2}} \Big\{ \EE\Big( \Big| \sum_{j=1}^N \int_{t_{j-1}}^{t_j} 
\big[ \|u(s)-u(t_j)\|_{V^0}^2 \nonumber \\
&\qquad \qquad  + \|u(s)-u(t_{j-1} )\|_{V^0}^2 + \|\theta(s)-\theta(t_j)\|_{H^0}^2 \big] \Big|^2 \Big) \Big\}^{\frac{1}{2}} 	\nonumber \\
&\, +C \sum_{j=1}^N \int_{t_{j-1}}^{t_j} \big\{ \EE\big( \|\AA^{\frac{1}{2}} \theta(t_j)\|_{H^0}^8\big) \big\}^{\frac{1}{2}} 
\Big\{ \EE \big( \|u(s)-u(t_{j-1})\|_{V^0}^4 \big)\Big\}^{\frac{1}{2}}  ds \;
\leq \; C\,  h^\eta, 
 \end{align}
 for some constant $C$ independent of $L$ and $N$. 
 Furthermore, the Lipschitz  conditions \eqref{LipG} and \eqref{LipGG}, the inclusion $\Omega_M(j-1)\subset \Omega_M(j-2)$ for $j=2, ...,N$
 and the upper estimates \eqref{increm_u_L2} and \eqref{increm_t_L2} imply 
 \begin{align}		\label{mom_part_sumTj6}
 \EE\Big( & \sum_{j=1}^L 1_{\Omega_M(j-1)} \Big\|\int_{t_{j-1}}^{t_j} \big[ G(u(s))-G(u^{j-1})\big] dW(s)\Big\|_{V^0}^2 \big)	\nonumber \\
 &\;  \leq 
  \sum_{j=1}^L \EE\Big( \int_{t_{j-1}}^{t_j} 1_{\Omega_M(j-1)} L_1 \| u(s)-u^{j-1}\|_{V^0}^2  {\rm Tr}(Q) ds \Big)	\nonumber \\
  &\; \leq 2 L_1{\rm Tr}(Q)\, h\,  \sum_{j=2}^L \EE( 1_{\Omega_M(j-2)}\|e_{j-1} \|_{V^0}^2  \big)
  + C \sum_{j=1}^L \EE\Big( \int_{t_{j-1}}^{t_j} \|u(s)-u(t_{j-1})\|_{V^0}^2 ds\Big)\nonumber \\
  &\; \leq 2 L_1 {\rm Tr}(Q) \, h\,  \sum_{j=2}^L \EE( 1_{\Omega_M(j-2)}\|e_{j-1} \|_{V^0}^2  \big) + C h, \\
  \EE\Big(  \sum_{j=1}^L &1_{\Omega_M(j-1)} \Big\|\int_{t_{j-1}}^{t_j} \big[ \GG(\theta(s))-\GG(\theta^{j-1})\big] d\WW(s)\Big\|_{H^0}^2 \big)	\nonumber \\
 &\;  \leq 2 \tilde{L}_1 {\rm Tr}(\QQ)\, h\,  \sum_{j=2}^L \EE( 1_{\Omega_M(j-2)}\|\ee_{j-1} \|_{H^0}^2  \big) + C h.	\label{mom_part_sumTTj5}
 \end{align}
 Finally, the Davis inequality, the inclusion $\Omega_M(J-1)\subset \Omega_M(j-1)$ for $j\leq J$, the local property of stochastic integrals,
 the Lipschitz condition \eqref{LipG}, the Cauchy-Schwarz and Young inequalities, and the inequality \eqref{increm_u_L2} imply for $\lambda >0$
 \begin{align} 	\label{mom_max_Tj7}
 \EE\Big( & \max_{1\leq J\leq L} 1_{\Omega_M(J-1)} \sum_{j=1}^J T_{j7} \Big) 		\nonumber \\
 &\; \leq 3 \sum_{j=1}^L \EE\Big( \Big\{ 1_{\Omega_M(j-1)} \int_{t_{j-1}}^{t_j} 
  \|G(u(s))-G(u^{j-1})\|_{{\mathcal L}(K;V^0)}^2  {\rm Tr}(Q) \|e_{j-1}\|_{V^0}^2 ds \Big\}^{\frac{1}{2}} \Big)		\nonumber \\
 &\; \leq 3 \,\sum_{j=1}^L \EE\Big[ \Big(  \max_{1\leq j\leq L} 1_{\Omega_M(j-1)} \|e_{j-1}\|_{V^0} \Big) \Big\{ \int_{t_{j-1}}^{t_j} L_1
  {\rm Tr}(Q)  \|u(s)-u^{j-1}\|_{V^0}^2  ds
 \Big\}^{\frac{1}{2}} \Big) 		\nonumber \\
 &\; \leq  \lambda\EE \Big(  \max_{1\leq j\leq L} 1_{\Omega_M(j-1)} \|e_{j-1}\|_{V^0}^2 \Big) + 
 C \EE\Big( \sum_{j=1}^L \int_{t_{j-1}}^{t_j} L_1 {\rm Tr}(Q)  \|u(s)-u^{j-1}\|_{V^0}^2  ds \Big) 		\nonumber \\
 &\; \leq  \lambda\EE \Big(  \max_{1\leq j\leq L} 1_{\Omega_M(j-2)} \|e_{j-1}\|_{V^0}^2 \Big) +  C h \sum_{j=1}^L \EE(\|e_{j-1}\|_{V^0}^2) 
 + C\, h.
 \end{align} 
 A similar argument, using the Lipschitz condition \eqref{LipGG} and  \eqref{increm_t_L2},  yields for $\lambda >0$
 \begin{align} 	\label{mom_max_TTj6}
 \EE\Big( & \max_{1\leq J\leq L} 1_{\Omega_M(J-1)} \sum_{j=1}^J \tilde{T}_{j7} \Big) 	
 \leq  \lambda\EE \Big(  \max_{1\leq j\leq L} 1_{\Omega_M(j-2)} \|\ee_{j-1}\|_{H^0}^2 \Big) + C h \sum_{j=1}^L \EE(\|\tilde{e}_{j-1}\|_{V^0}^2)  
 + C h.
 \end{align}
 Collecting the upper estimates \eqref{maj_Tj1}--\eqref{mom_max_TTj6}  we obtain for $\sum_{i=1}^5 \delta_i<\frac{1}{3}$, 
 $\sum_{i=1}^4 \tilde{\delta}_i<\frac{1}{3}$, $\eta\in (0,1)$ 
 and $\alpha,\lambda >0$ 
 \begin{align}		\label{mom_max_L2-1} 
 \EE\Big( &\max_{1\leq J\leq N} 1_{\Omega_M(j-1)} \big[ \|e_j\|_{V^0}^2 + \| \ee_j\|_{H^0}^2 \big] \Big) 	\nonumber \\
 &\quad + \EE\Big( \sum_{j=1}^N 1_{\Omega_M(j-1)}
 \Big[ \nu\Big( 2-6\sum_{i=1}^5 \delta_i\Big)  \|A^{\frac{1}{2}} e_j\|_{V^0}^2
  + \kappa \Big( 2-6\sum_{i=1}^4 \tilde{\delta}_i\Big)  \|\AA^{\frac{1}{2}} \ee_j\|_{V^0}^2 \Big] \Big)		\nonumber \\
  \leq & \, h\, \sum_{j=1}^{N-1} \EE\Big( 1_{\Omega_M(j-1)} \|e_j\|_{V^0}^2 
 \Big[ \frac{3(1+\alpha) \bar{C}_4^2 }{2\nu} \Big( \frac{1}{\delta_1} + \frac{1}{4\delta_2}\Big) M
  +C  \Big]  \Big)		\nonumber \\
  &\quad + h\, \sum_{j=1}^{N-1} \EE\Big( 1_{\Omega_M(j-1)} \|\ee_j\|_{H^0}^2 
  \Big[ \frac{3(1+\alpha) \bar{C}_4^2 }{8 \tilde{\delta_1} \kappa}  M  +C  \Big]  \Big)		\nonumber \\
  &\quad + C(1+M) h \EE\Big( \sup_{t\in [0,T]} \big[ \|u(t)\|_{V^0}^2 + \|\theta(t)\|_{H^0}^2 \big] 
  + \max_{1\leq j\leq N} \big[ \|u^j\|_{V^0}^2 + \|\theta^j\|_{H^0}^2\big)  \big] \Big) 	\nonumber \\
&\quad   + 12 \lambda \EE\Big( \max_{1\leq j\leq N} 1_{\Omega_M(j-1)} \big[ \|e_{j-1}\|_{V^0}^2 + \| \ee_j\|_{H^0}^2 \big] \Big) + C h^\eta.
   \end{align} 
 Therefore, given $\gamma \in (0,1)$, choosing $\lambda \in (0, \frac{1}{12})$  and $\alpha >0$ such that $\frac{1+\alpha}{1-12\lambda} < 1+\gamma$, 
 neglecting the sum in the left hand side and using the discrete Gronwall lemma, we deduce for $\eta \in (0,1)$ 
 \begin{align}		\label{mom_loc_maxut_L2}
  \EE\Big( &\max_{1\leq J\leq N} 1_{\Omega_M(j-1)} \big[ \|e_j\|_{V^0}^2 + \| \ee_j\|_{H^0}^2 \big] \Big)  \leq C(1+M) e^{T {\mathcal C}(M)} h^\eta ,
 \end{align} 
 where 
 \[ {\mathcal C}(M):= \frac{3(1+\gamma) \bar{C}_4^2}{2}  \max\Big(\frac{1}{\delta_1 \nu} + \frac{1}{4\delta_2\nu } \, , 
 \, \frac{1}{4\tilde{\delta}_1 \kappa}\Big) \, M,  \]
  for $\sum_{i=1}^2 \delta_i <\frac{1}{3}$ and $\tilde{\delta}_1<\frac{1}{3}$ (and choosing $\delta_i, i=3, 4,5$ and $\tilde{\delta}_i, i=2,3,4$ such
 that $\sum_{i=1}^5 \delta_i < \frac{1}{3}$ and $\sum_{i=1}^4 \tilde{\delta}_i < \frac{1}{3}$).  Let $\delta_2<\frac{1}{15}$ and $\delta_1=4\delta_2$. 
 Then for some  $\gamma >0$, we have
 \[  {\mathcal C}(M) =  \frac{9(1+\gamma) \bar{C}_4^2}{8}  \max\Big(\frac{5}{ \nu}  \, , 
 \, \frac{1}{\kappa}\Big) \, M.  \]
Plugging the upper estimate \eqref{mom_loc_maxut_L2} in \eqref{mom_max_L2-1}, we conclude the proof of \eqref{loc_cv}. 
\end{proof}

\section{Rate of convergence in probability and in $L^2(\Omega)$} 		\label{sec_speed} 
In this section, we deduce from Proposition \ref{th_loc_cv} the convergence in probability of the implicit time Euler scheme with the ``optimal" 
rate of convergence
``almost 1/2" and a logarithmic speed of convergence in $L^2(\Omega)$. This notion has been introduced in \cite{Pri}. 
The presence of the bilinear term in the It\^o formula for
 $\|\AA^{\frac{1}{2}} \theta(t)\|_{H^0}^2$  does not enable us to prove exponential moments for this norm, which prevents from using
 the general framework presented in \cite{BeMi_FEM}  to prove a polynomial rate  for the strong convergence.
 \subsection{Rate of convergence in probability}
 In this section, we deduce the rate of convergence in probability from Propositions \ref{prop_u_V1}, \ref{lem_mom_t_H1}, \ref{prop_mom_scheme}
  and \ref{th_loc_cv}.  
 
 \noindent{\it Proof of Theorem \ref{th_cv_proba}}
For $N\geq 1$ and $\eta\in (0,1)$, let 
\[ A(N,\eta):=\Big\{ \max_{1\leq J\leq N} \big[ \|e_J\|_{V^0}^2 + \|\ee_J\|_{H^0}^2 \big] + \frac{T}{N} \sum_{j=1}^N \big[ \|A^{\frac{1}{2}} e_j\|_{V^0}^2
+ \| \AA^{\frac{1}{2}} \ee_j \|_{H^0}^2 \big] \geq  N^{-\eta}\Big\}.
\]
Let $\tilde{\eta}\in  (\eta,1)$, $M(N)=\ln(\ln N)$ for $N\geq 3$.  
Then 
\[ P\big(A(N,\eta)\big) \leq   P\big( A(N,\eta) \cap \Omega_{M(N)} \big) + P \big( (\Omega_{M(N)})^c\big),\]
where $\Omega_{M(N)}= \Omega_{M(N)}(N)$ is defined in Proposition \ref{th_loc_cv}. 
The inequality \eqref{loc_cv}  implies 
\begin{align*}
 P\big( &A(N,\eta) \cap  \Omega_{M(N)} \big) \\ 
 &    \leq N^{\eta}  \;  \EE\Big( 1_{\Omega_{M(N)}}
 \Big[   \max_{1\leq J\leq N} \big[ \|e_J\|_{V^0}^2 + \|\ee_J\|_{H^0}^2 \big] + \frac{T}{N} \sum_{j=1}^N \big[ \|A^{\frac{1}{2}} e_j\|_{V^0}^2
+ \| \AA^{\frac{1}{2}} \ee_j \|_{H^0}^2 \big]  \Big]\Big) \\
&\leq N^{\eta} \; C \big[1+\ln(\ln N) \big] e^{T \tilde{C}  \ln(\ln N)} \Big( \frac{T}{N}\Big)^{\tilde{\eta}  } \\
&\leq C \big[1+\ln(\ln N) \big] \big(  \ln N\big)^{\tilde{C} T} N^{-\tilde{\eta}+\eta} \to 0 \quad {\rm as}\; N\to \infty.
\end{align*} 
The inequalities  \eqref{mom_u_V1}, \eqref{mom_t_L2} and \eqref{mom_t_H1}  imply
 \begin{align*}
P\big(( \Omega_{M(N)})^c\big) \leq \frac{1}{M(N)} \EE \Big( \sup_{t\in [0,T]} \|u(t)\|_{V^1}^2 + \sup_{t\in [0,T]} \|\theta(t)\|_{H^1}^2 \Big) 
\to 0 \quad {\rm as}\; N\to \infty.
\end{align*}
The two above convergence results  complete the proof of \eqref{speed_proba}. 
\hfill $\Box$

\subsection{Rate of convergence in $L^2(\Omega)$} 
We finally prove the strong rate of convergence, which is also a consequence of   Propositions \ref{prop_u_V1}, \ref{lem_mom_t_H1},
\ref{prop_mom_scheme}  and \ref{th_loc_cv}. 	

 \noindent{\it Proof of Theorem \ref{th_strong_rate}}    %
For any integer $N\geq 1$ and $M\in [1,+\infty)$,  
 let $\Omega_M=\Omega_M(N)$ be defined by \eqref{def-A(M)}.
Let $p$ be the conjugate exponent of $2^q$. H\"older's inequality implies 
\begin{align}	\label{strong-1}
 \EE\Big( & 
 1_{(\Omega_M)^c} \max_{1\leq J\leq N} \big[ \|e_J\|_{V^0}^2 + \|\ee_J\|_{H^0}^2 \big] \Big)  
 \leq  \Big\{ P\big(
 (\Omega_M)^c \big) \Big\}^{\frac{1}{p}}     \nonumber \\
 &\quad \times \Big\{ \EE\Big( \sup_{s\in [0,T]} \|u(s)\|_{V^0}^{2^q} +  \sup_{s\in [0,T]} \|\theta(s)\|_{H^0}^{2^q} 		
+ \max_{1\leq j\leq N} \|u^j\|_{V^0}^{2^q} +  \max_{1\leq j\leq N} \|\theta^j\|_{H^0}^{2^q} \Big) \Big\}^{2^{-q}}		\nonumber \\
 \leq &\;  C \Big\{ P\big( (\Omega_M)^c \big)  \Big\}^{\frac{1}{p}} ,
\end{align}
where the last inequality is a consequence of \eqref{mom_u_L2}, \eqref{mom_t_L2} and \eqref{mom_scheme_0}.

Using  \eqref{mom_u_V1}
and \eqref{mom_t_H1}  we deduce
 \begin{equation}		\label{strong-2}
P\big( (\Omega_M)^c\big) \leq M^{-2^{q-1}} \EE\big( \sup_{s\in [0,T]} \|u(s)\|_{V^1}^{2^q} +\sup_{s\in [0,T]} \|\theta(s)\|_{H^1}^{2^q} \Big) = C  M^{-2^{q-1}}.
\end{equation}
Using \eqref{loc_cv}, we choose $M(N)\to \infty$ as $N\to \infty$ such that for $\eta \in (0,1)$ and  $\gamma>0$
\[ N^{-\eta} \exp\Big[ \frac{9(1+\gamma) \bar{C}_4^2 T }{8}  \Big( \frac{5}{\nu} \vee \frac{1}{\kappa}\Big) M(N) \Big] M(N) \asymp M(N)^{-2^{q-1}},
\]  
which, taking logarithms, yields
\[ -\eta \ln(N) + \frac{9(1+\gamma) \bar{C}_4^2 T}{8}  \big( \frac{5}{\nu} \vee \frac{1}{\kappa}\big) M(N)  \asymp   -  2^{q-1}   \ln(M(N)).
\]
Set 
\begin{align*} M(N)= &\,  \frac{8}{9(1+\gamma) \bar{C}_4^2 \big( \frac{5}{\nu} \vee \frac{1}{\kappa}\big) T} 
\big[ \eta \ln(N) -  \big( 2^{q-1} +1\big) 
\ln\big( \ln(N)\big) \big] \\
\asymp &\, \frac{8}{9(1+\gamma) \bar{C}_4^2 \big( \frac{5}{\nu} \vee \frac{1}{\kappa}\big) T} \eta \ln(N). 
\end{align*} 
Then 
\begin{align*}
-\eta \ln(N) + \frac{9(1+\gamma) \bar{C}_4^2 T}{8}  \big( \frac{5}{\nu} \vee \frac{1}{\kappa}\big)  M(N) + \ln (M(N)) &\asymp  - \big( 2^{q-1} +1\big) 
\ln\big( \ln(N)\big)
+ 0(1), \\
-
\big( 2^{q-1} +1\big)  \ln\big( M(N)\big)& \asymp - 
\big( 2^{q-1} +1\big) \ln(N) + 0(1).
\end{align*} 
This implies 
\[ \EE\Big( \max_{1\leq J\leq N} \big[ \|e_J\|_{V^0}^2 + \|\ee_J\|_{H^0}^2 \big] \Big) \leq C \big( \ln(N)\big)^{-\big( 2^{q-1} +1)}. 
\]
The inequalities \eqref{mom_u_V1}--\eqref{mom_t_H1} for $p=1$ and \eqref{mom_scheme_1} for $K=1$ imply 
 \[ \sup_{N\geq 1} \EE \Big( \frac{T}{N} \sum_{j=1}^N \big[ \| A^{\frac{1}{2}} u(t_j) \|_{V^0}^2 + \|A^{\frac{1}{2}} u^j\|_{V^0}^2 
 + \| \AA^{\frac{1}{2}} \theta(t_j)\|_{H^0}^2  + \| \AA^{\frac{1}{2}} \theta^j\|_{H^0}^2 \big] \Big) <\infty.
 \]
  Hence a similar argument 		
 implies 
 \[\EE\Big( \frac{T}{N} \sum_{j=1}^N \big[ \|A^{\frac{1}{2}} e_j\|_{V^0}^2 + \|\AA^{\frac{1}{2}} \ee_j\|_{H^0}^2 \big] \Big) \leq 
  \ C \big( \ln(N)\big)^{ -( 2^{q-1} +1)}
  .\] 
This yields \eqref{strong_rate} and completes the proof. 
\hfill $\Box$
\section{Conclusions} \label{sconclusion}

This paper provides the first result about the rate of convergence of a time discretization of the Navier-Stokes equations coupled with a transport equation  
for the temperature, driven by a random perturbation; this is the so-called Boussinesq/B\'enard  model. The perturbation may depend on both the velocity and  temperature of the a fluid. The rates of convergence in probability and in $L^2(\Omega)$ are similar to those obtained for the stochastic Navier-Stokes equations.  The Boussinesq equations model a variety of phenomena in environmental, geophysical, and climate systems (see e.g. \cite{Dij}  and \cite{DGS}).  Even if the outline of proof is similar to that used for the Navier-Stokes equations, 
 the interplay between the velocity and the temperature is more delicate to deal with in many places.
 This interplay, which appears in B\'enard systems,
 is crucial to describe more general hydrodynamical models.  The presence of the velocity in the bilinear term 
describing the  dynamics of the temperature makes more difficult
to prove bounds of moments for the $H^1$-norm of the temperature uniformly in time and requires higher moments of the initial condition.
 Such bounds are crucial to deduce rates
 of convergence (in probability and in $L^2(\Omega)$) from the localized one.

 This localized version of the convergence 
 is the usual first step in a  non linear (non Lipschitz and not monotonous) setting. 
 Numerical simulations, which are the ultimate aim of this study since there is no other way to "produce" trajectories of the solution,  
would require a space discretization, such as finite elements. This is not dealt with in this paper and will be done in a forthcoming work.
This new study is likely to provide results similar to those obtained for the 2D Navier-Stokes equations.
  
  Also note that another natural continuation of this work would be to consider a more general stochastic 2D magnetic B\'enard model (as discussed in 
  \cite{ChuMil}), which  describes the time evolution of the velocity, temperature and magnetic field of an incompressible fluid. 
  
  It would also be interesting to study the variance of the $L^2(D)$-norm of the error term, in both additive and  multiplicative settings, 
  for the Navier-Stokes equations and more general B\'enard systems.  This would give some information about the accuracy of the approximation.
   Proving  a.s. convergence of the scheme for B\'enard models is also a challenging question.

 \bigskip
  
 \noindent {\bf Acknowledgements.}   
The authors thank anonymous referees for valuable remarks. \\
 Annie~Millet's research has been conducted within the FP2M federation (CNRS FR 2036). 
\smallskip

\noindent{\bf Declarations.}

\noindent{\bf Funding.} Hakima Bessaih was partially supported by Simons Foundation grant: 582264 and NSF grant DMS: 2147189.

\noindent {\bf  Conflicts of interest.} The authors have no conflicts of interest to declare that are
 relevant to the content of this article.

\noindent{\bf  Availability of data and material.} Data sharing not applicable to this article as
no datasets were generated or analysed during the current study.

\end{document}